\documentclass[11pt]{amsart}
\usepackage{amssymb,amsmath, amsthm}
\usepackage{tikz, tikz-cd, tkz-euclide}
\usepackage{siunitx}
\usepackage{fullpage}
\usetikzlibrary{calc,intersections}
\usepackage{enumerate}
\usepackage{xcolor}

\newif\iflong

\DeclareMathOperator{\CAT}{CAT}
\DeclareMathOperator{\Aut}{Aut}
\DeclareMathOperator{\SL}{SL}
\DeclareMathOperator{\PSL}{PSL}
\DeclareMathOperator{\SO}{SO}
\DeclareMathOperator{\proj}{proj}

\DeclareMathOperator{\stab}{stab}
\renewcommand{\phi}{\varphi}

\numberwithin{equation}{section}
\numberwithin{table}{section}

\theoremstyle{plain}
\newtheorem{thm}{Theorem}
\numberwithin{thm}{section}
\newtheorem{theorem}[thm]{Theorem}
\newtheorem{proposition}[thm]{Proposition}
\newtheorem{corollary}[thm]{Corollary}
\newtheorem{lemma}[thm]{Lemma}

\theoremstyle{definition}
\newtheorem{definition}[thm]{Definition}
\newtheorem{example}[thm]{Example}

\newtheorem{remark}[thm]{Remark}

\newcounter{comhar}
\usepackage{hyperref}

\newcounter{comgiu}
\newcommand{\comgiulio}[1]{{\textcolor{blue}{\textrm{{\bf (\arabic{comgiu})\stepcounter{comgiu} Giulio:} #1}}}}

\title[A shadow lemma in Hilbert geometry]{A global shadow lemma and logarithm law for geometrically finite Hilbert geometries}
\author{Harrison Bray and Giulio Tiozzo}
\thanks{George Mason University, \texttt{hbray@gmu.edu}.}
\thanks{University of Toronto, \texttt{tiozzo@math.utoronto.ca}.}

\date{\today}

\begin{document}

\begin{abstract}
For geometrically finite group actions on hyperbolic metric spaces and under certain assumptions 
on the growth of parabolic subgroups, we prove a global shadow lemma for Patterson--Sullivan measures, 
as well as a Dirichlet-type theorem and a logarithm law for excursion of geodesics into cusps. 
We then apply these results to geometrically finite quotients of strictly
convex Hilbert geometries with $C^1$ boundary.  
\end{abstract}

\maketitle

\section{Introduction}

In this work, we prove a global version of the \emph{shadow lemma}
(\cite{sullivan79, sullivan84, StratmannVelani}) for the Patterson--Sullivan
measures associated to  geometrically finite, strictly convex real projective manifolds.
We then apply it to obtain a \emph{logarithm law}, as in
\cite{sullivan82}, which provides asymptotics for the maximal cusp excursion for generic geodesics
and relates it to the dimension of the limit set. 
Our results follow from more general statements we will prove in the context of hyperbolic metric spaces
which satisfy certain growth conditions for the parabolic subgroups. 

A convex real projective structure is given by a properly convex domain
$\Omega$ in real projective space $\mathbb RP^n$, with an action by a
discrete group $\Gamma$ of projective transformations preserving $\Omega$.
The quotient manifold $M = \Omega/\Gamma$ is called a {\em convex projective
manifold}, and inherits a natural metric
$d_\Omega$ called the \emph{Hilbert metric}.  If $\Omega$ is strictly
convex, geodesics for the Hilbert metric are simply straight lines.  The
moduli space of these geometries is frequently nontrivial and includes the
example of hyperbolic structures of constant negative curvature.  

A Hilbert geometry $(\Omega,d_\Omega)$ is in general only Finsler, meaning
the metric comes from a norm, but this norm does not necessarily come from
an inner product. Once $\Omega$ is preserved by a non-compact group of
projective transformations, the Hilbert geometry $(\Omega,d_\Omega)$ is
Riemannian if and only if $\Omega$ is an ellipsoid (\cite{sociemethou},
\cite[Theorem 2.2]{crampon_handbook}).  Moreover, aside from the special
case of the ellipsoid, these Hilbert geometries are not $\CAT(k)$ for any
$k\leq 0$ \cite{marquis_handbook}.  Nonetheless, as Marquis states in
\cite{marquis_handbook}, we may think of Hilbert geometries as having
``damaged nonpositive curvature.'' In particular, a strictly convex Hilbert
geometry with a large isometry group has many properties resembling
negative curvature. 

Hyperbolic manifolds are equipped with a natural boundary of their
universal cover which carries several interesting quasi-invariant measures; in particular,  the
\emph{Patterson--Sullivan} measure, obtained by taking limits of Dirac measures
supported on the group orbit
(\cite{patterson,sullivan79, sullivan84}).  In that context, Sullivan's
\emph{global shadow lemma}, also known as {\em fluctuating density
theorem},
establishes the scaling properties of such measures
near boundary points. These properties turn out to depend subtly on the
location of the parabolic points, and are related to the fine structure of
the limit set. 

In this paper we are going to study, more specifically, geometrically finite 
convex projective manifolds, and extend this result to them. 
In this context and in various degrees of generality, 
an analogue of the Patterson--Sullivan measure has been constructed on the
boundary of $\Omega$ in projective space (see \cite{Ben1, cramponthese,
cramponmarquis_flot, zhu20, bray, blayac, blayaczhu}). 

Let us now assume that $\Gamma$ is a geometrically finite group of isometries of a strictly convex domain $\Omega$ with $C^1$ boundary 
and that the convex hull $C_\Gamma$ of the limit set $\Lambda_\Gamma$ is hyperbolic in the sense of Gromov. 

This assumption applies to a large family of examples; for instance, 
if $\Omega$ is strictly convex with $C^1$ boundary, any
properly discontinuous action with cofinite volume, and more generally any 
geometrically finite 
action for which all parabolic stabilizers have maximal rank, will have a
Gromov hyperbolic convex hull \cite{clt, cramponmarquis_finitude}. 
The moduli space of finite volume convex real projective
structures on a surface of genus $g$ with $p$ punctures has real dimension $16g-16+8p$ \cite{marquis_modules}. 
In these settings, the parabolic stabilizers are conjugate into $SO(n,1)$
\cite{clt, cramponmarquis_finitude}. There are moreover
examples where the parabolic stabilizers are not conjugate into
$SO(n,1)$ \cite[Proposition 10.7]{cramponmarquis_finitude}
and yet the convex hull is still hyperbolic \cite{DGK,zimmer22}, so our results apply. 
We expand on this discussion in Subsection~\ref{S:hyp_cc}.

To state the theorem, we now introduce a few definitions. 
For a basepoint $o$ and a boundary point $\xi$, let $\xi_t$ be the point
on a geodesic ray from $o$ to $\xi$ at distance $t$ from $o$. 

\begin{definition}  
Define the {\em shadow} $V(o,\xi,t)$ from a point
$o\in \Omega$ to a boundary point $\xi$ of depth $t\geq 0$ to be the set of
all boundary points $\eta\in\partial\Omega$ such that the Gromov product $\langle \xi,\eta\rangle_o $ is at least $t$ (see
Subsection~\ref{S:gromov_prod}). 
 \end{definition}
 
Essentially, a boundary point $\eta$ belongs to the shadow $V(o, \xi, t)$ if 
some geodesic ray $[o, \eta)$ intersects a ball of bounded radius around
$\xi_t$ 
(see Lemma \ref{L:fellowtravels}).
 
 Given a group $\Pi$ of isometries of a metric space $(X, d)$, we define its \emph{critical exponent} as 
 $$\delta_\Pi := \limsup_{t \to \infty} \frac{1}{t} \log \#\{ g \in \Pi \ : \ d(o, go) \leq t \}$$
 and its \emph{critical power} as 
 $$a_\Pi := \limsup_{t \to \infty} \frac{ \log \#\{ g \in \Pi \ : \ d(o, go) \leq t \} - \delta_\Pi t}{\log t}.$$

 Our main result is the following. 

\begin{theorem}
  \label{T:intro-global_shadow_lemma-Hilbert}
  Let $\Omega$ be a strictly convex domain in $\mathbb RP^n$ with
  $C^1$ boundary and $\Gamma<\PSL(n+1,\mathbb R)$ a discrete geometrically
  finite group which preserves $\Omega$. 
  Assume the convex hull of the limit set $C_\Gamma$ is hyperbolic with respect to the Hilbert metric, let $o \in C_\Gamma$ be a basepoint, 
  and let $\mu_o$ be the Patterson--Sullivan measure.
  Then there exists a constant $C$ for which the following holds: 
  for any $\xi \in \Lambda_\Gamma$, 
  we have
   \[
    C^{-1} d(\xi_t, \Gamma o)^{a_\Pi} e^{-\delta_\Gamma t + (2 \delta_\Pi  - \delta_\Gamma) d(\xi_t, \Gamma o) }  
    \leq \mu_o(V(o,\xi,t))\leq 
    C  d(\xi_t, \Gamma o)^{a_\Pi} e^{-\delta_\Gamma t + (2 \delta_\Pi - \delta_\Gamma) d(\xi_t, \Gamma o) }  
   \]
  for any $t > 0$, where $\Pi = \{  \textup{id}\}$ if $\xi_t$ lies in the non-cuspidal part, and otherwise is equal to the 
  stabilizer of the boundary point of the horoball containing $\xi_t$.
\end{theorem}

In the setting of Theorem \ref{T:intro-global_shadow_lemma-Hilbert},
although not obvious, $a_\Pi, \delta_\Pi$, and $\delta_\Gamma$ are finite. See Section
\ref{S:Hilbert-applications} and Proposition \ref{P:mixed_growth}.

In fact, Theorem \ref{T:intro-global_shadow_lemma-Hilbert} holds in greater
generality: 
see Theorem \ref{T:intro-global_shadow_lemma}.
As in the classical case, Theorem \ref{T:intro-global_shadow_lemma-Hilbert} implies:

\begin{corollary} \label{C:doubling}
The Patterson--Sullivan measure $\mu_o$ is doubling; that is, there exists $C > 0$ such that 
for any $\xi \in \Lambda_\Gamma$ and any $r > 0$ we have 
$$\mu_o(D(\xi, 2 r)) \leq C \mu_o(D(\xi,  r)),$$
where $D(\xi, r)$ denotes the ball of center $\xi$ and radius $r$ for the
Gromov metric 
on the limit set $\Lambda_\Gamma$. 
\end{corollary}

The Gromov metric is defined in Equation~\eqref{E:shadow-ball}; see also Lemma~\ref{L:diameter}
for a comparison between balls in the Gromov metric and shadows of horoballs. 

\subsection{Shadow lemma for hyperbolic metric spaces}

In fact, Theorem \ref{T:intro-global_shadow_lemma-Hilbert} is the
consequence of a more general theorem, that we prove for a large class of
(Gromov) hyperbolic metric spaces. 

Let $(X, d)$ be a Gromov hyperbolic metric space, and let $\partial X$ be its Gromov boundary:
for background, see Section~\ref{S:hyperbolic}.
 If $\Gamma$ is a geometrically finite group of isometries of $X$,  
there is a quasi-invariant horoball decomposition (Proposition \ref{P:bowditch_horoball_decomposition}), and there are finitely many $\Gamma$-orbits of parabolic points in $\partial X$.
We pick for each such orbit a parabolic point $p_i$, let $\Pi_i$ be its stabilizer, 
and define the function $B_i(t) := \#\{ g \in \Pi_i \ : \ d(o,
go) \leq t\}$ for any $t \geq 0$ and for a  fixed basepoint $o\in X$. Moreover, we define the function $b : X \to \mathbb{R}$ as follows: for $x \in X$, let $b(x) := B_i(2 d(x, \Gamma o))$ if $x$ lies in a horoball whose boundary point belongs to $\Gamma p_i$, and $b(x) := 1$ if $x$ lies in the non-cuspidal part, i.e. it does not belong to any horoball. 

The main result in full generality, that we will prove as Theorem \ref{T:global_shadow_lemma},  is:

\begin{theorem}
  \label{T:intro-global_shadow_lemma}
    Let $(X,d)$ be a proper hyperbolic metric space and $\Gamma$ a geometrically finite group of
  isometries of $X$.
  Let $\mu$ be a quasi-conformal density of dimension $\delta$ on
  $\Lambda_\Gamma$ with no atoms, and  assume that $\Gamma$ has $\delta$-tempered parabolic subgroups. 
  Then there exists a constant $C$ such that for all
  $\xi\in \Lambda_\Gamma$ and all $t \geq 0$, we have 
  \[
    C^{-1} b(\xi_t) e^{-\delta (t + d(\xi_t, \Gamma o))}  
    \leq
    \mu(V(o,\xi,t))\leq 
    C b(\xi_t) e^{-\delta (t + d(\xi_t, \Gamma o))}. 
  \]
\end{theorem}

For the definition of $\delta$-tempered parabolic subgroups, see Section~\ref{S:tempered}. 
For the definition of quasi-conformal density, see
Section~\ref{S:conformal}. 
The statement directly generalizes the main theorem of \cite{StratmannVelani} for
hyperbolic manifolds, and of \cite{schapira04} for Riemannian manifolds with non-constant negative curvature.

\subsection{Dirichlet Theorem} 

To state the new result, let $\mathcal{P}$ be the set of parabolic points, and recall that a 
\emph{horoball} of center $p \in \partial X$ and of radius $r$ is defined as 
$ H_p(r):=\{  x\in X  : \beta_p(x,o)\leq \log r\}$ 
where $\beta_p( \cdot, \cdot)$ is the Busemann function at $p$ (see Section~\ref{S:boundary}). 
Given a quasi-invariant horoball decomposition, for each parabolic point $p$ there is a unique horoball $H_p$ centered at $p$, 
and we denote the radius of $H_p$ by $r_p$.  Finally, let
$\mathcal H_{p}(s)$ be the shadow of the horoball centered at $p$ with
radius $s$, and  $\mathcal{P}_s := \{p\in \mathcal{P} \mid r_p\geq s\}.$ The following will be proven as Theorem \ref{T:Dirichlet}.

\begin{theorem}[Dirichlet-type Theorem]   \label{T:intro_dirichlet}
  Let $(X,d)$ be a hyperbolic metric space and $\Gamma$ a geometrically
  finite group of isometries of $X$. Then there exist constants $c_1>0,c_2\geq 1$ such that for all $s\leq c_1$, the union
  \[
    \bigcup_{p\in \mathcal{P}_s}
	\mathcal H_p(c_2\sqrt{sr_p})
  \]
  covers the limit set $\Lambda_\Gamma$, and there exists $0< c_3 \leq 1$ such that the shadows 
  $\{\mathcal H_p(c_3\sqrt{sr_p})\}_{p\in \mathcal P_s} $ are pairwise disjoint.
\end{theorem}

We can see this is a Dirichlet-type theorem by considering the classical
case of $\SL(2,\mathbb Z)$ acting on the hyperbolic plane $\mathbb H^2$,
where the horoballs in the standard horoball packing are centered at
rational points $\frac{p}q$ with radii $\frac1{q^2}$. 

\subsection{Applications}

As an application of the shadow lemma (Theorem \ref{T:intro-global_shadow_lemma-Hilbert}) 
and the Dirichlet theorem (Theorem \ref{T:intro_dirichlet}), 
we prove a horoball counting theorem (Proposition \ref{P:counting}), and a Khinchin-type theorem 
(Theorem \ref{T:khin}), culminating in a version of Sullivan's logarithm law for geodesics in the setting of Hilbert geometries:

\begin{theorem}[Logarithm Law]  \label{T:intro-log-law}
Let $\Omega$ be a strictly convex domain in $\mathbb RP^n$ with 
$C^1$ boundary and $\Gamma<\PSL(n+1,\mathbb R)$ a geometrically finite group which
preserves $\Omega$. 
Assume the convex hull $C_\Gamma$ of the limit set  is hyperbolic with
  respect to the Hilbert metric. Let $o \in \Omega$ and let $\mu_o$ be the associated Patterson--Sullivan measure.
  Then for $\mu_o$-almost every $\xi$ in the limit set $\Lambda_\Gamma$, the following holds: 
  \begin{equation} \label{E:log-law}
    \limsup_{t\to+\infty} \frac{d(\xi_t,\Gamma o)}{\log t} =
    \frac1{2(\delta_\Gamma - \delta_{\max})}
  \end{equation}
  where $\delta_{\max}$ is the maximal growth rate of any parabolic
  subgroup. 
  \label{thm-intro:loglaw}
\end{theorem}

Intuitively, the logarithm law shows that a generic geodesic makes larger and larger excursions into the cusp of the quotient manifold as time goes by; 
however, note also that the $\liminf$ in Equation \eqref{E:log-law} is almost surely zero, as almost every geodesic is recurrent to the non-cuspidal part. 

In fact, we only use that the space $C_\Gamma$ is a hyperbolic metric space and the measure satisfies a shadow lemma. 
 More precisely, in Theorem \ref{T:loglaw} we will prove a general
 logarithm law that applies to any hyperbolic metric space (including e.g.
 Riemannian manifolds with pinched negative curvature), under the
 assumption that parabolic subgroups satisfy the \emph{$\delta$-tempered} and \emph{mixed exponential growth} conditions from Definitions \ref{D:tempered_growth} and \ref{D:mixed_growth}. 
A logarithm law for Riemannian manifolds of variable negative curvature
appears in \cite{hersonskypaulin04} assuming 
that parabolic subgroups have pure exponential growth. Our Theorem \ref{T:loglaw} 
generalizes their result to mixed exponential growth. 

In a different vein, we also obtain as a consequence of the shadow lemma:

\begin{corollary}[Singularity with harmonic measure]
  \label{C:singular_harmonic}
Let $\Omega$ be a strictly convex domain in $\mathbb RP^n$ with $C^1$ boundary and $\Gamma<\PSL(n+1,\mathbb R)$ 
a geometrically finite group which preserves $\Omega$. 
Assume the convex hull of the limit set $C_\Gamma$ is hyperbolic with respect to the Hilbert metric.
Let $\mu$ be a measure on $\Gamma$ with finite superexponential moment, and 
let $\nu$ be the hitting measure of the random walk driven by $\mu$. 
If $\Gamma$ contains at least one parabolic element, then $\nu$ is singular with respect to the Patterson--Sullivan measure. 
\end{corollary}

  \subsection{Historical remarks}
  The global shadow lemma and logarithm law are originally due to Sullivan in the
  constant negative curvature, finite volume setting \cite{sullivan79,sullivan84,sullivan82}. 
  The argument was generalized and expanded upon to the geometrically finite setting by Stratmann-Velani
  \cite{StratmannVelani}, and to the Riemannian setting of variable negative curvature by
  Hersonsky-Paulin \cite{hersonskypaulin02, hersonskypaulin04, hersonskypaulin07,
  hersonskypaulin10} and Paulin-Pollicott \cite{paulinpollicott}. Schapira earlier proved the global shadow lemma in the
  Riemannian setting under certain growth conditions on the parabolic
  subgroups \cite{schapira04}. In a different direction, influential work
  of Kleinbock-Margulis extends Sullivan's logarithm law to non-compact
  Riemannian symmetric spaces \cite{kleinbockmargulis98,
  kleinbockmargulis99}. 
  Fishman--Simmons--Urba\'nski prove a different
  version of a Dirichlet theorem and a
  Khinchin-type theorem in the setting of hyperbolic metric spaces
  \cite{FSU}. 
  We point the interested reader to a survey of Athreya for more historical context \cite{athreya_survey}.

The dynamics of the Hilbert geodesic flow was first studied by Benoist in
the cocompact setting. For the cocompact case, Benoist proved that the
Anosov property of the Hilbert geodesic flow, strict convexity of $\Omega$,
$C^1$-regularity of the boundary, and hyperbolicity of the Hilbert
metric are all equivalent \cite[Th\'eor\`eme 1.1]{Ben1}.  More recently, \cite[Theorem 0.15,
Theorem 11.6]{clt}
generalized this result to the non-compact, finite volume case (without the
Anosov property, which does not apply to a non-compact phase space).
Finally, \cite{cramponmarquis_finitude} introduced and studied two
definitions of geometrically finite action in Hilbert geometry, which were
then studied also by Blayac and Zhu \cite{blayac,zhu20, blayaczhu}. 
Note that the paper \cite{cramponmarquis_finitude} contains
mistakes, which however do not affect the results of this paper. 
We discuss the connections between our work and that of Crampon--Marquis and
Blayac--Zhu in Section~\ref{S:CM14}. 
Crampon, Marquis, Blayac, and Zhu study Patterson--Sullivan measures in the geometrically
finite setting \cite{cramponthese, cramponmarquis_flot, blayac, zhu20,
blayaczhu}. We discuss their work in Section~\ref{S:PS_hilbert}.

For hyperbolic groups, a version of the shadow lemma for the
Patterson--Sullivan measure associated to the word metric is proven by
Coornaert \cite{CoornaertMichel1993}.
This has been more recently generalized by Yang
for relatively hyperbolic groups \cite{YangWenyuan2013}. 

\subsection{Structure of the paper} 
In Section~\ref{S:hyperbolic}, we recall some background material on hyperbolic metric spaces and establish some properties of horoballs and projections that we will need later. 
In Section~\ref{S:geom_finite}, we define the notions of geometrical finiteness and $\delta$-tempered
 parabolic subgroups that we use. 
In Sections \ref{S:conformal} and \ref{S:proof-global} we prove the main result, Theorem \ref{T:intro-global_shadow_lemma}, 
for general hyperbolic metric spaces. The Dirichlet Theorem (Theorem
\ref{T:intro_dirichlet}) and the applications in hyperbolic metric spaces 
are addressed in Section~\ref{S:applications}, including the logarithm law (Theorem \ref{thm-intro:loglaw}). 
Finally, in Section~\ref{S:Hilbert-applications} we discuss the applications to Hilbert geometry, 
completing the proof of Theorem \ref{T:intro-global_shadow_lemma-Hilbert}.

\subsection{Acknowledgements}
We thank Pierre-Louis Blayac, Ludovic Marquis, Feng Zhu, and Andrew Zimmer
for helpful discussions on Hilbert geometry.  We also thank Ilya Gekhtman,
Sam Taylor, and Wenyuan Yang  for some comments on a draft of the paper.
Finally, we thank the referee for their helpful comments.
G. T. was partially supported by NSERC and an Ontario Early Researcher
Award. H. B.  was partially supported by the Simons Foundation.

\section{Hyperbolic metric spaces} \label{S:hyperbolic}

\begin{figure}[!b]
  \centering
  \begin{tikzpicture}[scale=.7]
  \begin{scope}[thick]
  \draw (0,0) coordinate (x);
  \draw (1,7) coordinate (z);
  \draw (5,0) coordinate (y);
  \draw (x) to [out=50,in=-80] coordinate[pos=.6] (q) coordinate[pos=.3] (b) (z);
  \draw (x) to[out=50, in=150] coordinate[pos=.4] (c) (y);
  \draw (z) to[out=-80, in=150] coordinate[pos=.6] (a) coordinate[pos=.8]
  (w) (y);
  \end{scope}

  \begin{scope}[]
  \draw (b) to[out=-10,in=90] (c);
  \draw (a) to[out=210,in=90] (c);
  \draw (b) to[out=-10,in=210] (a);
  \end{scope}

  \def\s{.05}
  \draw[fill] (x) circle (\s) node[below left] {$x$};
  \draw[fill] (y) circle (\s) node[below right] {$y$};
   \draw[fill] (z) circle (\s) node[below left] {$z$};
  \draw[fill] (a) circle (\s) node[above right] {$a$};
  \draw[fill] (b) circle (\s) node[left] {$b$};
  \draw[fill] (c) circle (\s) node[below] {$c$};
  
\end{tikzpicture}
  \caption{Inner triangles in Gromov hyperbolic metric spaces.
    The point $b$ is such that $\langle y, z \rangle_x = d(x, b) = d(x, c)$.
  }
  \label{fig:tripod-new}
\end{figure}

In this section, we discuss properties of a general hyperbolic metric space $(X, d)$, 
which we will apply to the Hilbert metric in later sections. Most results should be well-known to experts, 
but we report them here in the precise form we need them. 

\subsection{Gromov product and inner triangle} \label{S:gromov_prod}

Let $(X, d)$ be a geodesic metric space. Given $x, y \in X$, we denote as $[x, y]$ a choice of geodesic segment with 
endpoints $x$ and $y$. Note that $X$ needs not be uniquely geodesic, so there may be more than one choice, but 
for all our statements it will not matter.

Now, consider a geodesic triangle with vertices $x, y, z \in X$ and sides $[x, y]$, $[y, z]$ and $[x, z]$.
Then there exist three points $a \in [y, z]$, $b \in [x, z]$, $c \in [x, y]$ such that 
$d(x, b) = d(x, c)$, $d(y, a) = d(y, c)$, $d(z, a) = d(z, b)$. 

We define the \emph{Gromov product} of $y, z$ centered at $x$ as 
$$\langle y, z \rangle_x := \frac{1}{2} \left( d(x, y) + d(x, z) - d(y, z) \right).$$
In the above notation, $\langle y, z \rangle_x = d(x, b) = d(x, c)$.
We call the triangle with vertices $\{a, b, c \}$ the \emph{inner triangle} associated to the points $x, y, z$, and denote it as $\Delta(x, y, z)$. 
A geodesic metric space is \emph{Gromov hyperbolic} (from now on, simply \emph{hyperbolic}) if there exists a constant $\alpha$ such that for any $x, y, z \in X$, the inner triangle $\Delta(x, y, z)$ has diameter at most $\alpha$. 
We denote as $O(\alpha)$ a quantity which depends only on the hyperbolicity
constant $\alpha$. Note that $O(\alpha)$ does not need to be a linear
function of $\alpha$ here.

\subsection{Busemann functions} \label{S:busemann}

Given $z \in X$, we define the \emph{Busemann function} $\beta_z : X \times X \to \mathbb{R}$ as 
$$\beta_z(x, y) := d(x, z) - d(y, z).$$
Note that level sets of the Busemann functions are metric spheres centered at $z$. 
For each $z$, the Busemann function $\beta_z( \cdot, \cdot)$ is anti-symmetric, 1-Lipschitz with respect to the $L^1$ metric on $X \times X$, and equivariant for any group of isometries of $X$. Moreover, the Busemann function is a cocycle, meaning for $x,y,z,w \in X$,
\[
  \beta_z(x,y)=\beta_z(x,w)+\beta_z(w,y).
\]
Moreover, it satisfies
$$2 \langle y, z \rangle_x = \beta_y(x,p) + \beta_z(x, p)$$
for any $p \in [y, z]$.

\subsection{Extension to the boundary and horoballs}
\label{S:boundary}
We denote as $\partial X$ the \emph{Gromov boundary} or \emph{hyperbolic
boundary} of $X$, that is (if $X$ is proper) the set of geodesic rays from
a given basepoint $o$, where we identify rays which lie within bounded
distance of each other. If $X$ is not proper, the definition of $\partial X$ is a bit more involved (see e.g. \cite[Def. III.H.3.12]{BH}), 
but in our applications we will focus only on the proper case.

In a hyperbolic space, Gromov products extend coarsely  to the hyperbolic boundary, by setting 
for any $o \in X$, $\xi, \eta \in \partial X$
$$\langle \xi, \eta \rangle_o := \liminf_{\stackrel{y \to \xi}{z \to \eta}} \langle y, z \rangle_o.$$
Similarly, for $\xi \in \partial X$, $x, y \in X$, one defines the \emph{Busemann function} as 
$$\beta_\xi(x, y) := \liminf_{z \to \xi} \beta_z(x, y).$$
These extensions are coarsely well-defined, meaning that 
\begin{equation}
| \liminf_{\stackrel{y \to \xi}{z \to \eta}} \langle y, z \rangle_x - \limsup_{\stackrel{y \to \xi}{z \to \eta}} \langle y, z \rangle_x | \leq O(\alpha), \qquad | \liminf_{z \to \xi} \beta_z(x, y)  - \limsup_{z \to \xi} \beta_z(x, y) | \leq O(\alpha).
  \label{E:coarse_busemann}
\end{equation}

It follows from Equation \ref{E:coarse_busemann} 
that the chosen definition of Busemann function is a {\em
quasi-cocycle}, meaning for 
$\xi \in \partial X$, $x, y, z \in X$
\[
  \beta_\xi(x,y)=\beta_\xi(x,z)+\beta_\xi(z,y) +O(\alpha). 
\]
The Busemann functions are also coarsely anti-symmetric, meaning
$\beta_\xi(x,y)=-\beta_\xi(y,x)+O(\alpha)$. 
Lastly, note that taking the liminf allows us to conclude that Busemann
functions are isometry-invariant, meaning: for any
isometry $g$ of $(X,d)$ and $\xi\in\partial X,$ $x,y\in X$, 
\[
  \beta_{g\xi}(gx,gy)=\beta_{\xi}(x,y).
\]
Also, as usual, these Busemann functions are 1-Lipschitz by the triangle
inequality. 

The notion of Busemann function allows us to extend the definition of an 
inner triangle to the boundary. Namely, for $x,y,z\in X\cup \partial X$, 
there exist three points $a \in [y, z]$, $b \in [x, z]$, $c \in [x, y]$
such that 
$\beta_x(b,c),\beta_y(a,c),$ and $\beta_z(a,b)$ differ by $O(\alpha)$.
We say $a,b,c$ are
the vertices of an inner triangle $\Delta(x,y,z)$. 
Note that this definition 
includes the definition of inner triangle when $x,y,z\in X$. 

A horoball is a sublevel set of the Busemann function. More precisely, 
given $\xi\in\partial X$ and $r>0$, 
the {\em horoball centered at $\xi$ of radius $r$}  is
\[
  H_\xi(r):=\{  x\in X  : \beta_\xi(x,o)\leq \log r\}. 
\]
The {\em horosphere centered at $\xi$ of radius $r$} is the set where
equality holds.

\subsection{Projections} \label{S:projections}

The notion of closest point projection will be fundamental in our paper. 

Given a point $x \in X$ and a geodesic $[y, z]$, 
a point $p \in [y, z]$ is a \emph{closest point projection} of
$x$ onto $[y, z]$ if it minimizes its distance to $x$: that is,  $d(x, p)
\leq d(x, q)$ for any $q \in [y, z]$, or equivalently, $\beta_x(p,q)\leq 0$
for any $q\in[y,z]$. 
Similarly, for $x,y,z\in X\cup \partial X$, closest point projection of
$x$ onto $[y,z]$ is any point $p$ such that $\beta_x(p,q)\leq 0$ for all $q\in[y,z]$. 

Closest point projection is not unique, but, in hyperbolic metric spaces,
it is well-defined up to bounded distance: in fact, 
any closest point projection of $x$ onto $[y, z]$ lies within
distance $O(\alpha)$ of any point of the inner triangle $\Delta(x,y,z)$. 
Hence, any two closest point projections lie
within distance $O(\alpha)$ of each other. 

To see this, first recall that hyperbolic metric spaces satisfy the reverse triangle inequality: 

\begin{proposition}[see e.g. \cite{MaherTiozzo}, Proposition 2.2]   \label{P:reverse_tri_ineq}
Let $(X, d)$ be a hyperbolic metric space, let $\gamma$ be a geodesic in $X$, $y\in X$ a point, and $q$ a closest
  point projection of $y$ to $\gamma$. Then for any $z\in\gamma$, 
  \[
    d(y,z)= d(y,q)+d(z,q)+O(\alpha).
  \]
\end{proposition}

Then the following lemma implies, for instance, that closest point
projection is coarsely well-defined:

\begin{lemma} \label{L:gp-proj}
Let $(X, d)$ be a hyperbolic metric space, let $o \in X$, $\eta, \xi \in X \cup \partial X$, and let $p$ be a closest point projection of $\eta$ onto $[o, \xi)$. 
Then 
$$\langle \eta, \xi \rangle_o = d(o, p) + O(\alpha).$$
Consequently, $p$ is within $O(\alpha)$ of the inner triangle
$\Delta(\eta,o,\xi)$. 
\end{lemma}

\begin{proof}
Let us first suppose that $\xi, \eta \in X$. Then by the reverse triangle
inequality in Proposition \ref{P:reverse_tri_ineq}, 
$$d(o, \eta) = d(o, p) + d(p, \eta) + O(\alpha)$$
$$d(\xi, \eta) = d(p, \xi) + d(p, \eta) + O(\alpha)$$
hence 
\begin{align*}
2 \langle \eta, \xi \rangle_o & = d(o, \eta) + d(o, \xi) - d(\eta, \xi)  \\
& = d(o, p) + d(p, \eta) + d(o, p) + d(p, \xi) - d(p, \xi) + d(p, \eta)  + O(\alpha) \\
& = 2 d(o, p) + O(\alpha).
\end{align*}
The claim then follows letting $\xi, \eta$ go to the boundary. 
\end{proof}

We now look at closest point projection onto horoballs. 

\begin{lemma}[Horoball projection]
  \label{L:horoball_projection}
  Let $(X,d)$ be a hyperbolic metric space and fix $o\in X$. 
  Then for all horoballs $H$ centered at $\xi\in \partial X$ and not containing
  $o$, geodesic rays $[o,\xi)$,  
  $p\in [o,\xi)\cap \partial
  H$, and $x\in[o,\xi)$, 
  \[
    d(x,\partial H)  \leq d(x,p) \leq d(x,\partial H)+O(\alpha).
  \]
\end{lemma}

\begin{proof}
  Let $q\in \partial H$.
  First, consider $x\not\in H$. 
  Then by definition, $\beta_\xi(q,o) = \beta_\xi(p,o)$,
  hence by the quasi-cocycle property, $\beta_\xi(p,q) = O(\alpha). $
  Let $z_n\in [o,\xi)$ be a sequence converging to $\xi$. Then for each $n$ sufficiently large, 
  \[
    d(x,p)+d(p,z_n)=d(x,z_n)\leq d(x,q)+d(q,z_n)
  \]
  hence by definition of $\beta_{z_n}$ and choice of $z_n$, 
  \[
    O(\alpha) = \beta_{z_n}(p,q) \leq d(x,q)-d(x,p).
  \]
  Hence, $d(x,p)\leq d(x,q)+O(\alpha)$. 

  Now, assume $x\in H$. Then $\beta_\xi(q,x)=\beta_\xi(p,x)+O(\alpha)$ by
  the quasi-cocycle property, and similarly 
  \[
    d(x,q)\geq
    |\beta_\xi(q,x)|=|\beta_\xi(p,x)|+O(\alpha)=d(x,p)+O(\alpha)
  \]
  which concludes the proof.
\end{proof}

Using the notation in Lemma \ref{L:horoball_projection},
$\beta_\xi(p,o)=\log r$, hence:

\begin{corollary}   \label{C:horoball_projection}
  For $H$ a horoball of radius $r$, 
we have
\[
  \log r=-d(o,H)+O(\alpha).
\]
\end{corollary}

\subsection{Shadows}

We can now introduce the definition of shadow.

\begin{definition} \label{D:shadow_depth_t}
  Let $(X,d)$ be a hyperbolic metric space, $o\in X$ and $\xi\in X\cup
  \partial X$. 
  The \emph{shadow} from $o$ to $\xi$ of depth $t\geq 0$ is the set 
  $$V(o, \xi, t) := \{ \eta \in \partial X \ : \ \langle \eta, \xi \rangle_o \geq t  \}.$$
\end{definition}
Note that, for any isometry $g$ of $X$, 
$$g V(o, \xi, t) = V(g o, g \xi, t).$$
In a hyperbolic metric space $X$, shadows of varying depth generate the topology on $\partial X$.

\subsection{Fellow traveling}

In a hyperbolic metric space, geodesic rays converging to the same boundary point satisfy strong fellow traveling properties.

\begin{lemma}[Asymptotic geodesics in a hyperbolic metric space] \label{L:asymp}
    \label{L:asymptotic_geodesics}
  Let $(X, d)$ be a hyperbolic metric space.
  Fix $\xi\in X\cup \partial X$ and $x,y\in X$ and denote by $x_t,y_t$ the
  points on geodesic rays $[x,\xi), [y,\xi)$ which are 
  distance $t$
  from $x$ and $y$, respectively. Then for all $0<t\leq
  \min\{d(x,\xi),d(y,\xi)\}$,
  \[
    d(x_t,y_t)\leq d(x,y) + O(\alpha).
  \]  
\end{lemma}

\begin{proof}
Let $p$ be a closest point projection of $\xi$ onto $[x, y]$, and suppose by symmetry that $d(x, p) \leq d(y, p)$. 
If $t < d(x, p)$ then $d(x_t, y_t) = d(x, y) - 2 t + O(\alpha)$.

If $d(x, p) \leq t < d(y, p)$, then $d(x_t, p) = t - d(x, p) + O(\alpha)$ and $d(y_t, p) = d(p, y) - t + O(\alpha)$ so 
$d(x_t, y_t) \leq d(x_t, p) + d(p, y_t) = d(y, p) - d(x, p) + O(\alpha) \leq d(x, y) + O(\alpha)$.

If $t \geq d(y, p)$ then $d(x_t, p) = t - d(x, p) + O(\alpha)$ and $d(y_t, p) = t - d(y, p) + O(\alpha)$, so 
$d(x_t, y_t) = |d(x_t, p) - d(y_t, p)| + O(\alpha) = |d(x, p) - d(y, p)| + O(\alpha) \leq d(x, y) + O(\alpha)$. 
\end{proof}

Given three points $x, y, z \in X\cup \partial X$, we say that two points $p \in [x, y]$ and $q \in [x, z]$ are \emph{comparable} if
$\beta_x(p,q)=0$ and $\beta_x(a,p)\geq 0$ where $a\in [x,y]$ is a vertex of the inner triangle $\Delta(x,y,z)$. 
Lemma \ref{L:asymptotic_geodesics} and the definition of inner triangle immediately implies:

\begin{corollary} \label{C:inner-tr}
Let $(X, d)$ be a hyperbolic metric space and for any $x, y, z \in X \cup \partial X$, let $p \in [x, y]$, $q \in [x, z]$
be comparable points. Then $d(p,q)\leq O(\alpha)$. 
\end{corollary}

The next lemma follows the preceding two lemmas:

\begin{lemma}[Fellow traveling]
  \label{L:fellowtravels}
  Let $(X, d)$ be a hyperbolic metric space, let $x,y\in X$ and $\xi,\eta$ in $\partial X$. Denote by 
  $\xi_t$, $\eta_t$ geodesic rays from $x$ to $\xi$ and from $y$ to
  $\eta$, respectively, parameterized at unit  speed. 
  If $\eta\in V(x,\xi,t)$ then $d(\xi_s,\eta_s)\leq
  d(x,y)+O(\alpha)$ for all $s\in[0,t]$. 
\end{lemma}

\begin{proof}
  Since $\eta$ is in $V(x,\xi,t)$, by Lemma \ref{L:gp-proj} 
  any closest point projection of $\eta$ onto
  $(x,\xi)$ is distance greater than $t+O(\alpha)$ from $x$. Let $q$ be the
  point on a geodesic ray $(x,\eta)$ which is distance $t$ from
  $x$. Then (up to $O(\alpha)$), $q$ and $\xi_t$ are comparable points on
  the thin triangle with vertices $x,\xi$, and $\eta$, and thus by
  Corollary  
  \ref{C:inner-tr} their
  distance is bounded above by $O(\alpha)$. 
  On the other hand, the distance from $q$ to $\eta_t$ is bounded above by
  $d(x,y)$ by Lemma \ref{L:asymptotic_geodesics}.
  The conclusion follows from the triangle inequality.  
\end{proof}

\subsection{Projections and Busemann functions}

An immediate corollary of Proposition \ref{P:reverse_tri_ineq} is:

\begin{corollary} \label{C:proj-beta}
  \label{C:projection_estimate}
  Let $(X, d)$ be a hyperbolic metric space, $\gamma$ a geodesic in $X$, $y\in X\cup\partial X$ a point, and
  $q$ a closest point projection of $y$ to $\gamma$. Then for any
  $z\in\gamma$, 
  \[
    \beta_y(z,q) = d(z,q)+O(\alpha).
  \]
\end{corollary}

The next lemma readily follows from Corollary \ref{C:proj-beta}. 

\begin{lemma} \label{L:Buse-proj}
Let $(X, d)$ be a hyperbolic metric space, let $\gamma$ be a (finite or infinite) geodesic, let $\eta \in X \cup \left(  \partial X  \setminus \overline{\gamma} \right)$ and let $p$ be a closest point projection of $\eta$ to $\gamma$. Then for any $x, y \in \gamma$ we have
$$\beta_\eta(x, y) = \beta_p(x, y) + O(\alpha).$$
\end{lemma}

\begin{proof}
By the cocycle property
\begin{align*}
  \beta_\eta(x, y) & = \beta_\eta(x, p) - \beta_\eta(y, p) + O(\alpha) \\ 
\intertext{and using Corollary \ref{C:proj-beta}}
& = d(x, p)  - d(y, p) + O(\alpha) \\ 
& = \beta_p(x, y) + O(\alpha).
\end{align*}
The equality then also holds for $\eta \in \partial X \setminus \overline{\gamma}$ as the closest point projection 
extends coarsely continuously.   
\end{proof}

\begin{lemma} \label{L:beta-level}
Let $(X, d)$ be a hyperbolic metric space, $o \in X$ a basepoint, $\xi \in \partial X$ and $\xi_t$ the 
point on a geodesic ray $[o, \xi)$ at distance $t$ from $o$.
If $\eta \in V(o, \xi, t)$, then
$$\beta_\eta(o, \xi_t) = t + O(\alpha).$$
On the other hand, if $\eta \notin V(o, \xi, D)$, then
$$- t \leq \beta_\eta(o, \xi_t) \leq -t + 2 D + O(\alpha).$$
\end{lemma}

\begin{proof}
Let $p$ be a closest point projection of $\eta$ onto $[o, \xi)$. Since $\eta \in V(o, \xi, t)$ and by
Lemma \ref{L:gp-proj}, $p$ lies between $\xi_{t-O(\alpha)}$ and $\xi$. 
Then by Lemma \ref{L:Buse-proj}
$$\beta_\eta(o, \xi_t) = \beta_p(o, \xi_t) + O(\alpha) = t + O(\alpha).$$
To prove the second part, if $\eta \notin V(o, \xi, D)$, then $d(o, p) \leq D$, so 
$$\beta_p(o, \xi_t) =  - t + 2 d(o,p) \leq -t +  2 D$$
so the upper bound follows from Lemma \ref{L:Buse-proj}. The lower bound follows from the triangle inequality.
\end{proof}

\subsection{Shadows in hyperbolic spaces} 

We will now state two lemmas on hyperbolic metric spaces and shadows that we will need later.

\begin{lemma}   \label{L:contain}  \label{L:technical_shadow_containments}
Let $(X, d)$ be a hyperbolic metric space, $o \in X$ a basepoint, $x,y\in X$, and $\xi\in X \cup \partial X$. 
\begin{enumerate}
\item 
  If $\eta \in V(o, \xi, t)$, 
  \[
    V(o, \eta, t) \subseteq V(o, \xi, t-O(\alpha)). 
  \]

\item For all $M>0$, there is a constant $A>0$ such that
  if $d(x,y)\leq M$, then 
  for all $\xi \in \partial X$ and all $t>0$, 
  \[
    V(x,\xi,t+A)\subset V(y,\xi,t)\subset V(x,\xi,t-A).
  \]
\end{enumerate}
\end{lemma}

In a hyperbolic metric space $(X,d)$, 
there is a metric $d_{\partial X}$ on $\partial X$ called the {\em Gromov metric}  with the property that
\begin{equation} \label{E:shadow-ball}
  c^{-1} e^{-\epsilon\langle \xi,\eta\rangle_{o}}\leq d_{\partial X}(\xi,\eta)\leq
c e^{-\epsilon\langle \xi,\eta\rangle_o}
\end{equation}
for some uniform constant $c$ and $\epsilon>0$, and any
$\eta,\xi\in\partial X$. 
We refer the reader to \cite[Prop. III.H.3.21]{BH} for this result and additional background.

Given a basepoint $o \in X$, we now define the {\em shadow of a set} to be the set of all
endpoints $\xi \in \partial X$ of geodesic rays starting from $o$ which intersect the set. The {\em shadow of a horoball} 
centered at $\xi\in\partial X$ of radius $r$ is denoted $\mathcal H_\xi(r)$. 

\begin{lemma} \label{L:diameter}
Let $(X, d)$ be a hyperbolic metric space.
Then there exists a constant $C$ such that for all $\xi\in\partial X$ and
  $r>0$, the shadow of a horoball $\mathcal H_\xi(r)$ 
  has diameter $s$ in the Gromov metric, where $C^{-1}r^\epsilon\leq  s\leq
  Cr^\epsilon$, and contains a ball of radius $C^{-1}r^\epsilon$ in the Gromov
  metric. 
  \label{L:radii_shadows}
\end{lemma}

\begin{proof}
  Let $\xi$ be the boundary point of the horoball $H=H_\xi(r)$ and
  $\eta\in\partial X$. 
  Let $p$ be a closest point projection of $\xi$ onto $[o,\eta)$. 
  By Corollary \ref{C:projection_estimate} and Lemma \ref{L:gp-proj},   
  \[
    \beta_\xi(o,p)=d(o,p)+O(\alpha)=\langle \xi,\eta\rangle_o+O(\alpha). 
  \]
  Let $q\in [o,\xi)\cap \partial H$. Then by Lemma
  \ref{L:horoball_projection} and Corollary \ref{C:horoball_projection}, 
  \[
    \beta_\xi(o,q)=d(o,H)+O(\alpha)=-\log r+O(\alpha). 
  \]
  Since $p$ minimizes
  $\beta_{\xi}(x,o)$ for all $x\in [o,\eta)$ by definition, if
  $\eta\in\mathcal H_\xi(r)$, then $p\in H$
  hence $[o,\eta)\cap H\neq \varnothing$  and $\beta_\xi(o,q)\geq \beta_\xi(o,p)$. 
  Thus, 
  \[
    d_{\partial X}(\eta,\xi) \leq ce^{-\epsilon \langle
    \xi,\eta\rangle_o}\leq ce^{-\epsilon\beta_\xi(o,p)+O(\alpha)}
    \leq ce^{-\epsilon \beta_\xi(o,q)+O(\alpha)}\leq
    ce^{O(\alpha)}r^\epsilon.
  \]
  Analogously, if $\eta\not\in\mathcal H_\xi(r)$, then
  $\beta_\xi(o,q)<\beta_\xi(o,p)$, and the lower bound follows. 
\end{proof}

\begin{lemma}
  \label{L:horoball_shadows_compare}
  Let $(X, d)$ be a hyperbolic metric space. Then for all $\xi\in\partial X$,
  \[
    V(o,\xi,-\log r + O(\alpha))\subset \mathcal H_\xi(r) \subset
    V(o,\xi,-\log r - O(\alpha)).
  \]
\end{lemma}

\begin{proof}
  The proof follows from Lemma \ref{L:radii_shadows} and Equation \ref{E:shadow-ball}.  
\end{proof}

\subsection{Disjointness}

The following lemmas will be used in the proof of Theorem \ref{T:Dirichlet}. 

\begin{lemma} \label{L:sqrt}
Let $(X, d)$ be a hyperbolic metric space, with $o \in X$ and $\xi_1, \xi_2 \in \partial X$ with $\xi_1 \neq \xi_2$. 
Let $q_1 \in [o, \xi_1)$ and $q_2 \in [o, \xi_2)$ with $\beta_{\xi_1}(o, q_2)
\geq \beta_{\xi_1}(o, q_1)$. 
Then there exists $z \in [o, \xi_2)$ such that
$$\beta_{\xi_1}(o, z) \geq \frac{d(o, q_1) + d(o, q_2)}{2} - O(\alpha).$$
\end{lemma}

\begin{figure}[h!]
  \centering
  \begin{tikzpicture}[scale=.7]
  \begin{scope}[thick]
    \draw[thin] (-5,0)--(7.5,0) node[below right] {$\partial X$};
    \draw (-2.5,0) coordinate (xi1);
  \draw (1,4) coordinate (o);
  \draw (5,0) coordinate (xi2);
  \draw (xi1) to [out=50,in=-80] coordinate[pos=.8] (q1) coordinate[pos=.6] (x) (o);
  \draw (xi1) to[out=50, in=150] coordinate[pos=.5] (y) (xi2);
  \draw (o) to[out=-80, in=150] coordinate[pos=.35] (z) coordinate[pos=.6]
  (q2) (xi2);
  \end{scope}

  \begin{scope}[]
  \draw (x) to[out=-10,in=90] (y);
  \draw (z) to[out=210,in=90] (y);
  \draw (x) to[out=-10,in=210] (z);
  \end{scope}

  \def\s{.05}
  \draw[fill] (xi1) circle (\s) node[below left] {$\xi_1$};
  \draw[fill] (o) circle (\s) node[left] {$o$};
  \draw[fill] (q1) circle (\s) node[left] {$q_1$};
  \draw[fill] (q2) circle (\s) node[above right] {$q_2$};
  \draw[fill] (x) circle (\s) node[left] {$x$};
  \draw[fill] (y) circle (\s) node[below] {$y$};
  \draw[fill] (z) circle (\s) node[above right] {$z$};
  \draw[fill] (xi2) circle (\s) node[below right] {$\xi_2$};

\end{tikzpicture}
  \caption{An approximate tree for the proof of Lemma \ref{L:sqrt}.}
  \label{fig:tripod}
\end{figure}

\begin{proof}
By hyperbolicity, the triangle $[o, \xi_1) \cup (\xi_1, \xi_2) \cup (\xi_2, o]$ is thin. Let $x, y, z$ be the vertices of its 
inner triangle, with $x \in [o, \xi_1)$, $y \in (\xi_1, \xi_2)$, $z \in [o, \xi_2)$.
By Lemma \ref{L:gp-proj}, $z$ is within $O(\alpha)$ of 
any closest point projection of $\xi_1$ onto $[o,\xi_2)$, so by Corollary \ref{C:projection_estimate}, 
\begin{align*}
  \beta_{\xi_1}(o, q_2) & = \beta_{\xi_1}(o,z)-\beta_{\xi_1}(q_2,z)+O(\alpha) \\
  & = d(o,z) - d(q_2,z)+O(\alpha).
\end{align*}
Moreover, since $q_1$ lies on $[o, \xi_1)$, 
$$\beta_{\xi_1}(o,q_1) = d(o, q_1).$$
Hence, from $\beta_{\xi_1}(o, q_2) \geq \beta_{\xi_1}(o, q_1)$ we obtain
\[
  0\leq d(z,q_2)\leq d(o,z)-d(o,q_1)+O(\alpha)=d(o,x)-d(o,q_1)+O(\alpha)
\]
which implies either $q_1\in [o,x]$ or is distance $O(\alpha)$ from
$x$ and hence from $z$. In either case, 
\[
  d(o,z)=d(o,q_1)+d(q_1,z)+O(\alpha). 
\]
It follows that 
$$d(z, q_2) \leq d(z, q_1) + O(\alpha).$$
Moreover, 
\begin{align*} d(o, q_2) & \leq d(o, q_1) + d(q_1, z) + d(z, q_2) \\ 
&  \leq  d(o, q_1) + 2 d(q_1, z) + O(\alpha)
\end{align*}
hence
\begin{align*}
\frac{d(o, q_2) + d (o, q_1)}{2} &  \leq  d(o, q_1) +  d(q_1, z) + O(\alpha) \\ 
& = \beta_{\xi_1}(o, z) + O(\alpha)
\end{align*}
which proves the claim.
\end{proof}

We will see later that a key property of our set-up, as in Sullivan's
original one \cite{sullivan82}, is that horoballs are disjoint. This has the following consequences.

\begin{lemma}
  \label{L:average_is_gromov}
  Let $(X,d)$ be a hyperbolic metric space.  Let $\xi_1, \xi_2 \in
  \partial X$ and let $H_1, H_2$ be horoballs based at $\xi_1, \xi_2$.
  Define as  $q_i$ an intersection point of $\partial H_i$ and $[o, \xi_i)$ for $i =
  1, 2$. If $H_1 \cap H_2 = \emptyset$, then 
$$\langle \xi_1, \xi_2 \rangle_o \leq \frac{d(o, q_1) + d(o, q_2)}{2} + O(\alpha).$$
\end{lemma} 

\begin{proof}
By symmetry, let us assume that $d(o, q_1) \leq d(o, q_2)$. 
Let $x$ be a closest point projection of $\xi_2$ onto $[o,\xi_1)$
and $z$ a closest point projection of $\xi_1$ onto $[o,\xi_2)$. 
Then $x$ and $z$ are within distance $O(\alpha)$ by Lemma
\ref{L:gp-proj}.

If $q_1 \in [x, \xi_1)$, then $d(o, q_2) \geq d(o, q_1) \geq d(o,
x)=\langle \xi_1,\xi_2\rangle_o$, 
hence the claim is trivially true. 

Suppose $q_1 \in [o, x]$.
See Figure \ref{fig:average_is_gromov}. 
Since $H_1$ and $H_2$ are disjoint, then $q_2$
does not belong to $H_1$, hence $\beta_{\xi_1}(z, q_2) < \beta_{\xi_1}(z, q_1)+O(\alpha)$.
By Corollary \ref{C:proj-beta}, 
$\beta_{\xi_1}(q_2,z)=d(q_2,z)+O(\alpha)$. 
Noting that $q_1$ is within distance $O(\alpha)$ of $[o,z]$ by Corollary
\ref{C:inner-tr} gives similarly that 
$\beta_{\xi_1}(q_1,z)=d(q_1,z)+O( \alpha)$, 
hence $d(x, q_2) \geq d(x, q_1) + O(\alpha)$. Then from
\begin{align*}
d(q_1, x) & = d(o, x) - d(o,q_1) \\
d(q_2, x)  & = d(o, q_2) - d(o, x) + O(\alpha)
\end{align*}
we obtain
$$\langle \xi_1, \xi_2 \rangle_o = d(o, x) \leq \frac{d(o, q_1) + d(o, q_2)}{2} + O(\alpha).$$
\end{proof}

\begin{figure}[h!]
  \centering
  \begin{tikzpicture}[scale=.7]
  \begin{scope}[]
  \draw (0,0) coordinate (xi1);
  \draw (0,9) coordinate (o);
  \draw (4,0) coordinate (xi2);

  \draw(-3,0)--(8,0) node[below right] {$\partial X$};

  \draw (xi1) --
  coordinate[pos=.35] (x) 
  coordinate[pos=.3] (x') 
  (o);
  \draw (o) to[out=-90, in=110] 
  coordinate[pos=.65] (z') 
  coordinate[pos=.6] (z) 
  coordinate[pos=.875] (q2)
  (xi2);

  \end{scope}

  \draw[thin] (xi1) to[out=90, in=110] coordinate[pos=.5] (c) (xi2);

  \begin{scope}[red]
    \def\r{3.3}
    \draw[] (0,\r) circle (\r);

    \draw (0,2*\r) coordinate (q1); 

    \def\s{.75}
    \draw[] (4,\s) circle (\s);

  \end{scope}

  \begin{scope}[dashed]
  \draw (x') to[out=-10,in=90] (c);
  \draw (z') to[out=210,in=90] (c);
  \draw (x') to[out=-10,in=210] (z');
  \end{scope}

  \def\s{.05}
  \draw[fill] (xi1) circle (\s) node[below left] {$\xi_1$};
  \draw[fill] (xi2) circle (\s) node[below right] {$\xi_2$};
   \draw[fill] (o) circle (\s) node[above left] {$o$};
  \draw[fill] (z) circle (\s) node[above right] {$z$};
  \draw[fill] (x) circle (\s) node[left] {$x$};

  \draw[fill] (q1) circle (\s) node[above left] {$q_1$};
  \draw[fill] (q2) circle (\s) node[above right] {$q_2$};
  
\end{tikzpicture}
  \caption{For the proof of Lemma \ref{L:average_is_gromov}, in the case
    that $q_1 \in [o, x]$. Note that $x$ and $z$ are within $O(\alpha)$ of
    the inner triangle $\Delta(o,\xi_1,\xi_2)$.  }
  \label{fig:average_is_gromov}
\end{figure}

The next corollary now follows from Equation \eqref{E:shadow-ball} 
and Lemma \ref{L:average_is_gromov}.

\begin{corollary} \label{C:disjoint}
 Let $(X,d)$ be a hyperbolic metric space, $\xi_1, \xi_2 \in \partial X$, and $r_1, r_2 > 0$.
Then there exists a constant $C>0$ such that, if the horoballs $H_{\xi_1}(r_1)$ and $H_{\xi_2}(r_2)$ are disjoint, then 
$$d_{\partial X}(\xi_1, \xi_2) \geq C (r_1r_2)^{\frac\epsilon2}.$$
\end{corollary}

\section{Geometrical finiteness} \label{S:geom_finite}

Let $(X, d)$ be a proper, geodesic metric space, and $\Gamma$ a countable group of isometries of $X$ 
acting properly discontinuously on $X$. Assume that $X$ has a compactification $\overline{X}$, 
namely $X$ embeds as an open, dense, subset of a compact metrizable space $\overline{X}$, 
and the action of $\Gamma$ extends to an action on $\overline{X}$ by homeomorphisms.
The set $\partial_{top} X := \overline{X} \setminus X$ is the \emph{topological boundary} of $X$. 
Given a basepoint $o\in X$,
define the {\em limit set} of $\Gamma$ as 
\[
  \Lambda_\Gamma= \overline{\Gamma o}\smallsetminus \Gamma o. 
\]
We say that the action of $\Gamma$ on $X$ is {\em non-elementary} if $|\Lambda_\Gamma|\geq 3$, and we 
denote by $C_\Gamma$ the {\em convex hull} of $\Lambda_\Gamma$ in $X$. More
specifically, $C_\Gamma$ is the union of all biinfinite geodesics which
have both 
endpoints in $\Lambda_\Gamma$. 

Given $\gamma \in \Gamma$, we define its \emph{translation distance} as $\tau(\gamma) := \inf_{x \in X} d(x, \gamma x)$.
We define an element $\gamma$ to be {\em elliptic} if $\tau(\gamma) = 0$ and the infimum is attained, 
{\em parabolic} if $\tau(\gamma) = 0$ and the infimum is not attained, and 
{\em loxodromic} if $\tau(\gamma) > 0$ and the infimum is attained. 

A subgroup $\Pi<\Gamma$ is a {\em parabolic group} if $\Pi$ has infinite order, 
fixes a single point of $\partial_{top} X$, and no element of $\Pi$ is loxodromic. 
We call $\xi\in \Lambda_\Gamma$ a  {\em parabolic point} if its stabilizer $\stab_\Gamma(\xi)$ is a  
parabolic subgroup. We say a parabolic point
$\xi$ is {\em bounded parabolic} if the quotient $(\Lambda_\Gamma\smallsetminus
\{\xi\})/\stab_\Gamma(\xi)$ is compact. A point $\xi\in \Lambda_\Gamma$ is a
{\em conical limit point} if there exist a sequence $(\gamma_n) \subseteq \Gamma$ and distinct points $a,b\in \Lambda_\Gamma$ such
that $\gamma_n \xi\to a$ and $\gamma_n \eta \to b$ for all
$\eta\in\Lambda_\Gamma \smallsetminus \{\xi\}$. 

Let us from now on assume that the space $(X, d)$ is Gromov hyperbolic; then we can take as $\overline{X}$ its Gromov compactification, 
and we denote as $\partial X = \partial_{top} X$ its Gromov boundary as in Subsection~\ref{S:boundary}. 
Note that $(C_\Gamma, d)$ is again a proper geodesic hyperbolic metric space on which $\Gamma$ acts properly discontinuously.
As a subset of $X$, the convex hull $C_\Gamma$ is only quasi-convex in this
setting. 

It is well-known that every isometry $\gamma\in\Gamma$ is either elliptic, parabolic, or loxodromic;  every parabolic element is infinite order and has exactly one fixed point in $\partial X$, and any 
loxodromic element is infinite order and has exactly two fixed points in $\partial X$ \cite[Lemma 2.1]{bowditch99}. 
Also, $\Lambda_\Gamma$ is basepoint independent, and moreover when $\Gamma$ is non-elementary,
$\Lambda_\Gamma$ is the smallest closed $\Gamma$-invariant subset of
$\partial X$ (see e.g. \cite[Th\'eor\`eme 5.1]{CoornaertMichel1993}).

We now define geometrical finiteness as in Tukia \cite{tukia} and Bowditch
\cite[p 38]{bowditch}, inspired by the work of Beardon--Maskit
\cite[Theorems 2 and 3]{beardonmaskit} on characterizing existence of
finite-sided fundamental domains for Kleinian groups:

\begin{definition}
  Let $(X,d)$ be a proper, hyperbolic metric space and $\Gamma$ a
  non-elementary group of
  isometries acting properly discontinuously on $X$. 
  Then $\Gamma$ is {\em
  geometrically finite} if every point of $\Lambda_\Gamma$ is either
  conical or bounded parabolic. 
  \label{D:geom_finite}
\end{definition}

\begin{remark}
  \label{R:bowditch_minimality}
We will at times reference the work of Bowditch \cite{bowditch} for
geometrically finite groups $\Gamma$ acting on a hyperbolic metric space
$(X,d)$ such that $\Gamma$
acts on $\partial X$ minimally. Bowditch notes that this framework is
general by simply replacing $X$ with $C_\Gamma$ in any situation where
$\Lambda_\Gamma\neq\partial X$, as $\partial C_\Gamma=\Lambda_\Gamma$
follows from the definition. 
\end{remark}

Let $ \mathcal P$ be the collection of
parabolic fixed points in $\partial X$  for the action of $\Gamma$. 
If $\Gamma$ is a geometrically finite
group of isometries of $X$, then  
there are finitely many orbits of parabolic points in $\mathcal P$
(see Yaman's criterion \cite{yaman}, \cite[Theorem 1B]{tukia}, or 
\cite[Proposition 6.15]{bowditch}), hence 
we may write
\[
  \mathcal P=\bigsqcup_{i=1}^a \mathcal P^i
\]
where each $\mathcal P^i$ is the orbit of a parabolic point. 

\subsection{Horoball decomposition}

Let $\mathcal P$ be the set of parabolic points in $\Lambda_\Gamma$, which
we note is $\Gamma$-invariant. 
We define a {\em quasi-invariant family of horoballs} to be a collection
$\{H_p\}_{p\in\mathcal P}$ of mutually disjoint horoballs $H_p$ centered at $p$ for which there
exists a constant $C$ such that 
$d(H_{\gamma p},\gamma H_p)\leq C$ for all $\gamma\in\Gamma, p\in\mathcal P$.
If in fact $H_{\gamma p}=\gamma H_p$ then $\{H_p\}_{p\in\mathcal P}$ is an
{\em invariant family of horoballs}. 
Such a family is said to be {\em $r$-separated} if $d(H_p,H_q)\geq r$ for
all $p\neq q\in\mathcal P$. 
Given a quasi-invariant family of horoballs $\{H_p\}_{p\in \mathcal P}$, the corresponding {\em
non-cuspidal part} for the action of $\Gamma$ on $X$ is the set
\[
  X_{nc} :=C_\Gamma\smallsetminus \bigcup_{p\in \mathcal P} H_p, 
\]
and the {\em cuspidal part} for the action of $\Gamma$ on $X$ is 
\[
  X_c :=\bigcup_{p\in \mathcal P} C_\Gamma\cap H_p. 
\]
The decomposition $C_\Gamma=X_{nc}\cup X_c$ is called a {\em horoball
decomposition of} $X$ or of $C_\Gamma$.
At the level of quotient, the {\em non-cuspidal part} is $M_{nc} := X_{nc} /
\Gamma$ and the {\em cuspidal part} is $M_c :=X_c/\Gamma = M\smallsetminus
X_{nc}$. Similarly, $M=M_{nc}\cup M_c$ is a {\em horoball decomposition of}
$M$. 

\begin{proposition}
  \label{P:bowditch_horoball_decomposition}
  Let $(X,d)$ be a proper hyperbolic metric space, $\Gamma$ a group of isometries
  of $(X,d)$ acting properly discontinuously on $X$. 
  If $\Gamma$ is geometrically finite, then $ \mathcal P / \Gamma$ is finite and there
  exists an $r$-separated quasi-invariant family of open horoballs $\{ H_p\}_{p\in  \mathcal P}$
  centered at each of the parabolic fixed points such that 
  the non-cuspidal part $M_{nc}$ is compact. 
\end{proposition}

\begin{proof}
  Bowditch's \cite[Proposition 6.13]{bowditch} states the conclusion, but for
  some more general notion of horoballs arising as sublevel sets of more
  general horofunctions 
  \cite[p 29]{bowditch}:  
  given $p\in\Lambda_\Gamma$, we
  say $h_p\colon X\to\mathbb R$ is a {\em horofunction centered at}
  $p$ if for any $x\in X$ and any $a\in X$ that 
  is within distance $O(\alpha)$ of $[x,p)$, then
  $h_p(a)=h_p(x)+d(x,a)+O(\alpha)$. We will refer to a sublevel set of
  a horofunction as a {\em generalized horoball}. 
  To compare to our definition of Busemann function, 
  let $a$ be the vertex of the inner triangle $\Delta(p,x,o)$ on a
  geodesic ray $[o,p)$. Bowditch's
  definition of horofunction implies immediately that 
  \begin{align*}
    h_p(x)-h_p(o) & =d(o,a)-d(x,a)+O(\alpha) \\
    & = 
    \beta_p(o,a)+\beta_p(a,x)+O(\alpha)
    =\beta_p(o,x)+O(\alpha). 
  \end{align*}
  It follows that 
  \[
    h_p(x)=\beta_p(o,x)+h_p(o)+O(\alpha). 
  \]
  As a consequence, every generalized horoball is within distance $O(\alpha)$ of a horoball.  
  The conclusion now follows \cite[Proposition 6.13]{bowditch}
  which states that there exists an $r$-separated invariant family of generalized horoballs 
  such that $M_{nc}$ is compact. 
\end{proof}

\subsection{Tempered growth}
\label{S:tempered}

We say two positive real valued functions $f$ and $h$ are {\em coarsely 
equivalent}, denoted $f\asymp h$, if there exists a uniform constant
$k\geq1$ for which $k^{-1} h \leq f \leq k h$.

Assume $(X,d)$ is a proper hyperbolic metric space and $\Gamma$ is a
geometrically finite group of isometries of $X$. Fix a basepoint $o\in X$. 
Recall the {\em critical exponent} of $\Gamma$ is 
\[
  \delta_\Gamma:=\limsup_{t\to\infty} \frac1t\log \#\{\gamma\in\Gamma :
  d(o,\gamma o)\leq t\}. 
\]
Equivalently, $\delta_\Gamma$ is the infimum over values of $s$ for which
the Poincar\'e series $P_\Gamma(s):=\sum_{\gamma\in\Gamma} e^{-s d(o,\gamma o)}$ converges.  
Fix a parabolic group $\Pi$. Let us denote as
$$B_\Pi(t) := \#\{ g \in \Pi \ : \ d(o, go) \leq t \}.$$
Given $r > 0$, we define the \emph{annular growth function}
$$A_{\Pi,r}(t) := \frac{1}{r} \log \left( \frac{B_\Pi(t+r)}{B_\Pi(t)} \right).$$
We define the {\em lower} and {\em upper annular growth rates} of $\Pi$ as, respectively,
\begin{equation} \label{E:AR_UL}
  \delta_\Pi^- := \lim_{r\to\infty}\inf_{t>0} A_{\Pi,r}(t),
  \qquad
  \delta_\Pi^+ := \lim_{r\to\infty}\sup_{t>0} A_{\Pi,r}(t).
\end{equation}
Note that by the definitions, the limit as $r\to\infty$ 
exists for both quantities, and that $\delta_\Pi^-\leq \delta_\Pi\leq \delta_\Pi^+$.

\begin{definition}
  \label{D:tempered_growth}
We say that the parabolic subgroup $\Pi$ has $\delta$-\emph{tempered growth} if
\[
  0<\delta_\Pi^-\leq\delta_\Pi^+<\delta.
\]
If $\Pi<\Gamma$ has $\delta_\Gamma$-tempered growth where $\delta_\Gamma$
is the critical exponent of $\Gamma$, then we simply say $\Pi$ has
{\em tempered growth}. If every maximal parabolic subgroup of $\Gamma$ has
($\delta$-)tempered growth, then we say $\Gamma$ has {\em
($\delta$-)tempered parabolic subgroups.}
\end{definition}

\begin{remark}
  \label{R:BT}
  Note the following:
  \begin{enumerate}
    \item 
      Since $B_\Pi(t)$ is nondecreasing for any parabolic subgroup $\Pi$, for any $r \leq s$ we have
      $$r A_{\Pi,r}(t) \leq s A_{\Pi,s}(t) \qquad \textup{for any }t \geq 0;$$
   \item 
     Note that, for any $k \geq 1$,
     $$A_{\Pi,kr}(t) = \frac{1}{k} \sum_{i = 0}^{k-1} A_{\Pi,r}(t + r i)$$
so
$$\inf_{t > 0} A_{\Pi,r}(t) \leq \inf_{t >0} A_{\Pi,kr}(t)
\leq \sup_{t >0} A_{\Pi,kr}(t) \leq
\sup_{t >0} A_{\Pi,r}(t).$$
\item 
For fixed $s \geq 0$ and $\Pi$ a parabolic group of tempered growth,
there exists a constant $C$
such that for all $t \geq 0$ 
$$B_\Pi(t) \leq B_\Pi(t+s)\leq CB_\Pi(t).$$
Indeed, if $s > 0$, let $k$ be such that $kr > s$.
Then by (1), for each $\Pi$ we have
$$\frac{B_\Pi(t + s)}{B_\Pi(t)} = e^{s A_{\Pi,s}(t)} \leq e^{kr A_{\Pi,kr}(t)}$$
and $A_{\Pi,kr}(t)$ is bounded above since $\sup_{t>0} A_{\Pi,kr}(t)$ exists.
It is straightforward to verify similarly that for $s<0$ there exists a
constant $C$ such that for all $t\geq s$,
\[
  C^{-1}B_\Pi(t)\leq B_\Pi(t-s)\leq B_\Pi(t).
\]
\end{enumerate}
 \end{remark}

 \begin{definition}
   \label{D:mixed_growth}
   We say a parabolic group $\Pi$ 
   has {\em mixed exponential growth} if there exist
   $\delta_\Pi>0,a_\Pi\geq0$ such that 
   \[
     B_\Pi(t)\asymp e^{\delta_\Pi t}(t+1)^{a_\Pi} \qquad \textup{for any } t \geq 0.
   \]
 \end{definition}
 A straightforward calculation shows that in this case, $\delta_\Pi^-=\delta_\Pi^+=\delta_\Pi>0$. Thus 
any parabolic group with mixed exponential growth has $\delta$-tempered growth for any $\delta>\delta_\Pi$.

\begin{figure}[]
  \centering
  \begin{tikzpicture}[scale=.7]
  \begin{scope}[red]
  \draw (-4,0) -- (-4,1) -- (4,1)--(4,0);

  \draw (-4,2) -- (-4,3) -- (4,3)--(4,2)--cycle;

  {\foreach \n in {-4,-2,0,2,4}
    \draw 
    (\n,0)--(\n,1)
    (\n,2)--(\n,3)
    (\n,1.6) node {$\vdots$};
  }
  \end{scope}

  \draw (-4.2,0)--(4.2,0); 

  \def\p{2}
  \def\t{235}
  \draw (-4,0)-- ++ (\t:.5);
  \draw (-2,0)-- coordinate[pos=.75] (ainvb)++ (\t:\p);
  \draw (0,0)-- coordinate[pos=.6] (binv) coordinate[pos=.85] (bbinv)
  ++(\t:5.5);
  \draw (2,0)-- coordinate[pos=.75] (ab) ++(\t:\p);
  \draw (4,0)--++(\t:.5);

  \draw 
  (ainvb)--++(0:.5*\p+.2)
  (ainvb)--++(180:.5*\p+.2);
  \draw 
  (ab)--++(0:.5*\p+.2)
  (ab)--++(180:.5*\p+.2);
  \draw 
  (binv)--coordinate[midway] (binva) ++(0:\p)--++(0:.2)
  (binv)--coordinate[midway] (binvainv) ++(180:\p)--++(180:.2);
  \draw
  (bbinv)--++(0:.3*\p+.2)
  (bbinv)--++(180:.3*\p+.2);
  \draw 
  (binva)--++(\t:.3*\p)
  (binva)--++(\t+180:.3*\p);
  \draw 
  (binvainv)--++(\t:.3*\p)
  (binvainv)--++(\t+180:.3*\p);

  \begin{scope}[shift={(ainvb)}, red]
    \draw (-.5*\p,0)--(-.5*\p,.25*\p)--(.5*\p,.25*\p)--(.5*\p,0)
    (-.5*\p,.125*\p)--(.5*\p,.125*\p);
    \draw 
    (-.5*\p,0)--(-.5*\p,.25*\p)
    (-.25*\p,0)--(-.25*\p,.25*\p)
    (0,0)--(0,.25*\p)
    (.25*\p,0)--(.25*\p,.25*\p)
    (.5*\p,0)--(.5*\p,.25*\p)
    ;
  \end{scope}

  \begin{scope}[shift={(ab)}, red]
    \draw (-.5*\p,0)--(-.5*\p,.25*\p)--(.5*\p,.25*\p)--(.5*\p,0)
    (-.5*\p,.125*\p)--(.5*\p,.125*\p);
    \draw 
    (-.5*\p,0)--(-.5*\p,.25*\p)
    (-.25*\p,0)--(-.25*\p,.25*\p)
    (0,0)--(0,.25*\p)
    (.25*\p,0)--(.25*\p,.25*\p)
    (.5*\p,0)--(.5*\p,.25*\p)
    ;
  \end{scope}

  \begin{scope}[shift={(binv)}, red]
    \draw (-\p,0)--(-\p,.5*\p)--(\p,.5*\p)--(\p,0)
    (-\p,.25*\p)--(\p,.25*\p);
    \draw 
    (-\p,0)--(-\p,.5*\p)
    (-.5*\p,0)--(-.5*\p,.5*\p)
    (0,0)--(0,.5*\p)
    (.5*\p,0)--(.5*\p,.5*\p)
    ;
  \end{scope}

  \def\s{.05}

  \draw[decorate,decoration={brace,amplitude=5pt,mirror,raise=5pt}, blue]
  (1.9,3)--(.1,3) node[midway, yshift=20pt] {$1/f(k)$} ;

  \draw[decorate,decoration={brace,amplitude=5pt,mirror,raise=10pt}, blue]
  (4,.1)--(4,2.9) node[midway, xshift=25pt] {$k$};

  \def\s{.07}
\end{tikzpicture}
  \begin{tikzpicture}[scale=.7]
  \draw (0,0) grid (8,2) ;
  \draw (0,3) grid (8,6) ;
  \draw[ultra thick, blue] 
  (2,0) coordinate (g) -- (2,2)
  (2,3)--(2,4)--(7,4)--(7,3)
  (7,2)--(7,0) coordinate (gal);
  \draw[ultra thick, dashed, blue] 
  (2,2)--(2,3)
  (7,2)--(7,3);
  {\foreach \n in {0,1,3,4,5,6,8}
    \draw[dashed] (\n,2)--(\n,3);
  }

  \draw[decorate,decoration={brace,amplitude=5pt,mirror,raise=5pt}, blue]
  (6.8,4)--(2.2,4) node[midway, yshift=20pt] {$l/f(k)$} ;

  \draw[decorate,decoration={brace,amplitude=5pt,mirror,raise=5pt}, blue]
  (8,.2)--(8,3.8) node[midway, xshift=20pt] {$k$};

  \def\s{.07}
  \draw[fill] (g) circle (\s) node[below, yshift=-3pt ] {$g$};
  \draw[fill] (gal) circle (\s) node[below ] {$ga^l$};
\end{tikzpicture}
  \caption{Left: the space in Example~\ref{ex:GM_space},
    constructed by attaching combinatorial horoballs (in red) 
    to the Cayley graph a
  free group. Right: a detail of a combinatorial horoball, with a
geodesic path from $g$ to $ga^l$. }
  \label{fig:GM_space}
\end{figure}

\begin{example}
  \label{ex:GM_space}
We note here that the tempered growth condition is not always satisfied.
We thank the referee for suggesting the following interesting counterexample.
 Consider the function

$$h(t) := \left\{ \begin{array}{ll}
n & \textup{if }2^n \leq t \leq 2^n + n \textup{ for some }n \in \mathbb{N}\\
1 & \textup{otherwise}
\end{array} \right.$$
and
$$f(t) := \exp \left( \int_0^t h(x) \ dx \right).$$

Consider the group $\Gamma =  \mathbb{Z} \star \mathbb{Z} = \langle a, b \rangle$, with generators $a, b$.
Let $\mathbb{T}_4$ be its Cayley graph with respect to these generators, which is a regular $4$-valent tree.
Our space will be a variation of the Groves--Manning cusp space, where the
Cayley graph is augmented by combinatorial horoballs
\cite{groves_manning}. Let
$X$ be a metric graph with vertex set $V := \mathbb{T}_4 \times \mathbb{N}$ and edges of the following lengths:
for any $g \in \Gamma$, $k \in \mathbb{N}$, set edges of lengths
$$\begin{array}{ll}
\ell((g, k), (ga, k)) & = \frac{1}{f(k)} \\
\ell((g, 0), (gb, 0)) & = 1 \\
\ell((g, k), (g, k+1)) & = 1.
\end{array}$$
We then consider the path metric $d$ on this graph $X$, which makes $X$ into a hyperbolic metric space.
The group $\Gamma$ acts by isometries on $(X, d)$ by acting with the
standard action on $\mathbb{T}_4$ and trivially on $\mathbb{N}$, and the
subgroup $\Pi := \langle a \rangle$ is bounded parabolic. See
Figure~\ref{fig:GM_space}.

Note that we can join $(g, 0)$ and $(g a^{f(k)}, 0)$ by a path of length $2k + 1$, hence the ball of radius $t = 2k +1$ contains
at least all elements $(g a^l, 0)$ with $| l | \leq f(k) = f((t-1)/2)$.

On the other hand, the geodesic between $(g, 0)$ and $(g a^l, 0)$ is the union of the ``vertical" path from $(g, 0)$ to $(g, k)$, the ``horizontal" path from $(g, k)$ to $(g a^l, k)$ and the ``vertical" path from $(g a^l, k)$ and $(g a^l ,0)$, where $k$ is chosen to minimize
$$\inf_{k \in \mathbb{N}} \left\{ 2k + \frac{l}{f(k)} \right\}.$$
In particular, if $k$ is the minimizer, then $2k + \frac{l}{f(k)} \leq 2(k+1) + \frac{l}{f(k+1)} \leq 2 (k+1) + \frac{l}{f(k) e}$
which yields, if $(g a^l,0)$ has distance at most $t$ from $(g,0)$,
$$l \leq \frac{2 }{1 - e^{-1}} f(k) \leq \frac{2 }{1 - e^{-1}} f\big(\frac{t}{2}\big)$$
hence for all integers $n$, we proved the inclusions
\begin{equation}
  \label{E:fandb}
  2 f\big(\frac{n-2}{2}\big) \leq B_\Pi(n) \leq \frac{4}{1 - e^{-1}} f\big(\frac{n}{2}\big)
\end{equation}

Then, setting $r = 2 n$ and $t = 2^{n+1}$,
\begin{align*}
\delta^+_\Pi &  := \lim_{r \to \infty} \sup_{t > 0} \frac{1}{r} \log \frac{B_\Pi(t+r)}{B_\Pi(t)} \\
&  \geq \lim_{r \to \infty} \sup_{t > 0} \frac{1}{r} \log \frac{f(\frac{t+r-2}{2})}{f(\frac{t}{2})}  \\
& = \lim_{r \to \infty} \sup_{t > 0} \frac{1}{r} \int_{\frac{t}{2}}^{\frac{t+r-2}{2}} h(x) dx  \\
& \geq \lim_{n \to \infty} \frac{1}{2n} \int_{2^n}^{2^n + n-1} n \ dx \geq
\lim_{n \to \infty} \frac{n-1}2 = + \infty.
\end{align*}

On the other hand,
$$\log f(t)  = \int_0^{t} h(x) dx \leq \int_0^t 1 \ dx + \sum_{\substack{
n\in\mathbb N \\  2^n \leq t} } \int_{2^n}^{2^n +n} n \ dx
\leq t + \sum_{n \leq \log_2(t) } n^2 \leq t + \lceil \log_2(t) \rceil^3
$$
so by the upper bound in Equation \ref{E:fandb}, 
$$
\limsup_{t\to+\infty} \frac{1}{t} \log B_\Pi(t) =
\limsup_{n\in\mathbb N} \frac{1}{n} \log B_\Pi(n)
\leq \limsup_{n\in\mathbb N} \frac{1}{n} \log
f\big(\frac{n}2\big) \leq \frac{1}{2}$$
hence $\delta_\Pi \leq 1/2 < + \infty$.
Using the fact that $f(t)\geq e^t$ and the lower bound of Equation
\ref{E:fandb}, it follows that $\delta_\Pi=\frac12$. One can also verify,
similarly to the calculation for $\delta_\Pi^+$, that
$\delta_\Pi^-=1/2>0$.
Thus, we have a geometrically finite action which does not have tempered
growth, despite having many essential properties for tempered growth. Note
that the growth function of parabolic subgroups is much more flexible than the
growth function of the whole group $\Gamma$, which is coarsely bounded above
by $e^{\delta_\Gamma t}$ when $\delta_\Gamma<\infty$
\cite[Th\'eor\`eme 5.4 and Proposition 6.4]{CoornaertMichel1993}. 
\end{example}

\subsection{A technical lemma}

We end this section with a technical lemma which will be needed in Section
\ref{S:conformal} but only requires the tools and definitions from
hyperbolic metric spaces. 
The lemma closely resembles an analogous lemma of Schapira written in French \cite[Lemme 2.9]{schapira04}.

\begin{figure}[h!]
  \centering
    \def\x{1.2}
  \begin{tikzpicture}[scale=\x]
    \draw[line width=.5em, opacity=.3] (140: 2) arc (140:170:2) node[above
    left, opacity=1] {$K$} 
    coordinate[pos=.4] (eta); 
    \draw[line width=.5em, opacity=.3] (30: 2) arc (30:70:2) node[above
    right, opacity=1] {$gK$} 
    coordinate[pos=.4] (geta); 
    \draw (0,0) circle (2); 
    \coordinate (xi) at (0,-2); 
    \draw[] (0,-.75) circle (1.25); 
    \coordinate (o) at (-1.25,-.75);
    \coordinate (go) at (0,{1.25-.75});
    \draw (eta)--(xi) coordinate[pos=.375] (y);
    \draw (geta)--(xi) coordinate[pos=.375] (gy);

    \def\s{1.5}
    \node[circle,fill,inner sep=\s] at (eta) {}; 
    \node[above left] at (eta) {$\eta$};
    \node[circle,fill,inner sep=\s] at (geta) {}; 
    \node[above right] at (geta) {$g\eta$};
    \node[circle,fill,inner sep=\s] at (xi) {};
    \node[below] at (xi) {$\xi$};
    \node[circle,fill,inner sep=\s] at (o) {};
    \node[left] at (o) {$o$};
    \node[circle,fill,inner sep=\s] at (go) {};
    \node[above left] at (go) {$go$};
    \node[circle,fill,inner sep=\s] at (y) {};
    \node[left] at (y) {$y$};
    \node[circle,fill,inner sep=\s] at (gy) {};
    \node[right, xshift=.3em] at (gy) {$gy$};
  \end{tikzpicture}
  \hspace{2em}
  \begin{tikzpicture}[scale=\x]
    \draw[line width=.5em, opacity=0] (30: 2) arc (30:70:2) 
    coordinate[pos=.4] (peta); 
    \draw (0,0) circle (2); 
    \coordinate (xi) at (0,-2); 
    \coordinate (o) at (-1.25,-.75);

    \draw (o)--(xi) coordinate[pos=.4] (a);
    \draw (o)--(geta) coordinate[pos=.25] (b);
    \draw (geta)--(xi) coordinate[pos=.7] (c);

    \draw[dashed] (a)--(b)--(c)--(a);

    \tkzMarkSegment[mark=|, pos=.45](o,a)
    \tkzMarkSegment[mark=|, pos=.4](o,b)
    \tkzMarkSegment[mark=||, pos=.6](geta,b)
    \tkzMarkSegment[mark=||, pos=.6](geta,c)
    \tkzMarkSegment[mark=|||, pos=.5](xi,a)
    \tkzMarkSegment[mark=|||, pos=.5](xi,c)

    \def\s{1.5}
    \node[circle,fill,inner sep=\s] at (geta) {}; 
    \node[above right] at (geta) {$g\eta$};
    \node[circle,fill,inner sep=\s] at (xi) {};
    \node[below] at (xi) {$\xi$};
    \node[circle,fill,inner sep=\s] at (o) {};
    \node[left] at (o) {$o$};
    \node[circle,fill,inner sep=\s] at (a) {};
    \node[below left] at (a) {$a$};
    \node[circle,fill,inner sep=\s] at (b) {};
    \node[above left] at (b) {$b$};
    \node[circle,fill,inner sep=\s] at (c) {};
    \node[right] at (c) {$c$};
  \end{tikzpicture}
  \caption{
    The set-up of Lemma \ref{L:geometriclemma}.
  }
  \label{fig:geometriclemma}
\end{figure}

\begin{lemma}   \label{L:geometriclemma}
Let $(X, d)$ be a hyperbolic metric space, $\xi \in \partial X$, $K \subseteq \partial X \setminus \{ \xi \}$ compact, $o \in X$.    Let $\xi_t$ be a point on a geodesic ray $[o,\xi)$ at distance $t$ from $o$. 
Then there exists $A > 0$ such that the following holds. 
For every $g$ a parabolic isometry of $X$ fixing $\xi$ and $\eta\in K$,
we have $|\langle \xi,g\eta\rangle_o - d(o, go)/2| \leq  A$. 
In particular, 
\begin{enumerate}
  \item $$gK \subseteq V(o, \xi, d(o, go)/2 + A) \setminus V(o, \xi, d(o, go)/2 - A)$$ 
  \item and moreover, for any $t > A$ 
    we have 
    $$| \beta_{g \eta}(\xi_t, g \xi_t)  - \max \{ d(o, go)-2t, 0 \} | \leq A.$$
\end{enumerate}
\end{lemma}

\begin{proof}
  Following the set-up of Schapira; for $\eta\in K$, let
  $y$ be the point on $(\xi, \eta)$ 
  which is on the same horosphere at $\xi$ as
  $o$. Then $y$ is bounded distance from $o$ for all $\eta\in K$ by
  compactness of $K$; let $C$ be an upper bound on $d(o,y)$.

  Consider a geodesic triangle with endpoints $\xi,o,$ and $g\eta\in gK$.
  This triangle has an inner triangle 
  with vertices $a\in
  [o,\xi), b\in [o, g\eta),$ and $c\in (\xi, g\eta)$ such that
  $\beta_o(a,b),\beta_\xi(a,c),\beta_{g\eta}(b,c)$ differ by $O(\alpha)$.

  We will first compare $d(o,go)$ with $2d(o,a)$ (notice that
  $a$ of course depends on $g$). Then we will estimate $2d(o,a)$ to
  prove the containments of shadows. Then we will prove the estimate on the
  Busemann functions. 

  Note that since $c$ and $gy$ are both on
  $(\xi, g\eta)$ and $g$ preserves horospheres centered at
  $\xi$, 
  \begin{equation}
    d(gy,c) = |\beta_\xi(gy,c)|
    +O(\alpha)
    = |\beta_\xi(gy,a)|
    +O(\alpha)
    =
     |\beta_\xi(o,a)| 
    +O(\alpha)
     = d(o,a)
     +O(\alpha). 
    \label{eqn:geometriclemma1}
  \end{equation}
  By hyperbolicity of the
  space, there is a uniform constant $O(\alpha)$ which bounds the diameter of
  this inner triangle. Then by the triangle inequality and
  \eqref{eqn:geometriclemma1},
  \[
    d(o,gy) \leq
    d(o,a)+d(a,c)+d(c, gy)\leq
    2d(o,a)+O(\alpha).
  \]
Moreover, let $q$ be the comparable point on a geodesic
  $[o,g\eta)$ to $gy$ on a geodesic $(g\eta,\xi)$; this means
  $q$ is the point in the same horosphere centered at $g\eta$ as the point $gy$.  In particular,
  \[
    d(q,b) = |\beta_{g\eta}(q,b)| +O(\alpha) =
    |\beta_{g\eta}(gy,c)|+O(\alpha) =
    d(gy,c) +O(\alpha) = d(o,a) +O(\alpha) = d(o,b) +O(\alpha)
  \]
  hence $d(o, q) = 2 d(o, a)+O(\alpha)$, and since $q$ is comparable to $gy$, their distance is bounded above
  by $O(\alpha) $ by Corollary \ref{C:inner-tr}.
  Then we obtain a lower bound
  \[
    2d(o,a)-O(\alpha) \leq d(o,q)-d(gy,q)\leq
    d(o,gy).
  \]
  Then by the triangle inequality, the fact that $g$ is an isometry, and
  the upper bound on $d(o,y)$,
  \begin{equation}
    2d(o,a)-C-O(\alpha)\leq d(o,go)\leq 2d(o,a)+C+O(\alpha).
    \label{eqn:distpo}
  \end{equation}
  Noting that
  $\langle
  \xi,g\eta\rangle_o=d(o,a)+O(\alpha)=\frac12d(o,go)$ implies $|\langle
  \xi,g\eta\rangle_o-\frac12 d(o,go)| \leq A$ where $A=C+O(\alpha)$. 
  By Definition \ref{D:shadow_depth_t}, the containments of shadows in (1)
  follows.

  Let us now prove the bound on the Busemann functions in (2). By the quasi-cocycle property,
\begin{equation}
\beta_{g\eta}(o, go) = \beta_{g\eta}(o, \xi_t) + \beta_{g\eta}(\xi_t,
g\xi_t) + \beta_{g\eta}(g \xi_t, go) + O(\alpha).
  \label{E:cocycle}
\end{equation}

Now, assume that $d(o,go)>2t$. Then $g\eta \in V(o, \xi, t-A)$, so we have
by Lemma \ref{L:beta-level}
$$\beta_{g\eta}(o, \xi_t) = t + O(\alpha).$$
Moreover, since the group acts by isometries,
$$\beta_{g\eta}(g \xi_t, go) = \beta_{\eta}(\xi_t, o).$$
Further, by compactness we can choose a constant $D$ such that
$K$ is disjoint from $V(o, \xi, D)$. Hence, by Lemma \ref{L:beta-level} and
the antisymmetric and 1-Lipschitz properties of Busemann functions, 
$\eta \in K$ implies
$$t - 2D-O(\alpha) \leq \beta_{\eta}(\xi_t, o) \leq t.$$
Finally, since $q$ lies on $[o, g \eta)$,
$$\beta_{g \eta}(o, q) = d(o, q)+O(\alpha)$$
and, as discussed before, 
$$d(q, go) \leq d(q, gy) + d(gy, go) \leq O(\alpha) + C$$
hence

$$\left| \beta_{g \eta}(o, go) - d(o, go) \right| \leq O(\alpha) + 2 C$$
which yields by Equation \ref{E:cocycle} and the preceding equations
$$\left| \beta_{g \eta}(\xi_t, g\xi_t) - d(o, go) + 2t \right| \leq B$$
for a suitable choice of $B$, as required.

Now assume $d(o,go)\leq 2t$. Then $g\eta\not\in V(o,\xi,2t-A)$ so up to
bounded error, $\xi_t$ is closer to $\xi$ than $a$. Since $g$ fixes
$\xi$ and horoballs are coarsely invariant,  
\begin{align*}
|\beta_{g\eta}(\xi_t,g\xi_t)|\leq d(\xi_t,g\xi_t)\leq
d(a,ga)+O(\alpha)=d(a,b)+O(\alpha) \leq O(\alpha).
\end{align*}
\end{proof}

\section{Quasi-conformal densities and estimates near the cusps} \label{S:conformal}

In this section, we will introduce the background on quasi-conformal
densities and prove several key lemmas for the global shadow
lemma. 

\subsection{Background on quasi-conformal densities}
\label{S:quasi_conformal_background}

Let $(X, d)$ be a proper hyperbolic metric space and $\Gamma < \textup{Isom}(X, d)$ acting properly discontinuously on $X$. 
Then a {\em quasi-conformal density of dimension} $\delta>0$ is a family
$\{\mu_x\}_{x\in X}$ of mutually absolutely continuous finite non-trivial
measures on $\partial X$ with the following properties:
\begin{itemize}
  \item (quasi-$\Gamma$-invariance) 
  there exists $C > 0$ such that for all $\gamma\in\Gamma$, $x\in X$, and a.e. $\eta \in \partial X$, 
    we have 
\begin{equation}
    C^{-1} \leq \frac{d \gamma_\ast \mu_x}{ d \mu_{\gamma x}}(\eta) \leq C;
  \label{E:quasi-inv-meas}
\end{equation}
    
  \item (transformation rule) there exists $C > 0$ such that, for all $x,y\in X$ and a.e. $\eta\in \partial X$, we have
    \begin{equation}\label{E:transform}
      C^{-1} e^{-\delta \beta_\eta(x,y)} \leq \frac{d\mu_x}{d\mu_y}(\eta) \leq C e^{-\delta \beta_\eta(x,y)}.
 \end{equation} 
 \end{itemize}

 If $C$ can be chosen equal to $1$, then the density is called a {\em conformal density}.
 A measure $\mu$ is a $\delta$-(quasi-)conformal measure if
 $\mu=\mu_{x}$ for some 
 (quasi-)conformal density $\{\mu_x\}_{x\in X}$ of dimension $\delta$ (see  in 
   \cite[Proposition 2.5]{MYJ} that this definition agrees with 
   the original definition of quasi-conformal measure, as in 
 \cite[Definition 4.1]{CoornaertMichel1993}).
 Note that 
 any quasi-conformal measure with support contained in $\Lambda_\Gamma$
 must in fact have support equal to $\Lambda_\Gamma$, because 
 quasi-$\Gamma$-invariance and the transformation rule imply that the
 support is a $\Gamma$-invariant set, and $\Lambda_\Gamma$ is the smallest
 closed $\Gamma$-invariant set. 

A particularly famous example of a conformal density is the
{\em Patterson--Sullivan density}, first constructed by Patterson for Fuchsian
groups and extended by Sullivan to geometrically finite actions on
hyperbolic spaces (\cite{patterson, sullivan79, sullivan84}). 
We call any density (measure) constructed in such a way a {\em
Patterson--Sullivan density (measure)}.
A Patterson--Sullivan density, if it exists, has conformal dimension equal
to the critical exponent 
$\delta_\Gamma$,
which is also the critical exponent of the Poincar\'e series. 

  The Patterson--Sullivan construction
  has been generalized by Coornaert
  \cite{CoornaertMichel1993} to any non-elementary group $\Gamma$ of isometries of
  $X$ when $(X,d)$ is  a proper 
  hyperbolic metric space and $\delta_\Gamma$ is finite. 
  Coornaert showed under these assumptions 
  that there exists a Patterson--Sullivan density 
  on $\Lambda_\Gamma$ \cite[Th\'eor\`eme
  5.4]{CoornaertMichel1993}. 
  Coornaert recovers Sullivan's shadow lemma \cite{sullivan79} in this setting
  \cite[Proposition 6.1]{CoornaertMichel1993} for a quasi-conformal measure
  $\mu$ of any dimension $\delta>0$.
  When the action is geometrically finite, it follows that 
  (1) $\mu$ must have conformal dimension at least $\delta_\Gamma$ \cite[Corollaire
6.6]{CoornaertMichel1993}; (2) the only points in $\Lambda_\Gamma$ that
can have positive $\mu$ mass are parabolic points; 
(3)
$\delta_\Gamma>0$ \cite[Corollaire 5.5]{CoornaertMichel1993}
and 
(4) the set of parabolic points has full 
measure if and only if the Poincar\'e series converges at $\delta$ 
(see e.g. \cite[p. 114]{dop} for the case of Patterson--Sullivan measure and
\cite[Proposition 2.12]{MYJ} in generality). 
When $\Gamma$ is geometrically finite, 
the set of
parabolic points is countable \cite[Lemma 6.9]{bowditch}, 
hence (4) implies any quasi-conformal density of dimension
$\delta>\delta_\Gamma$ has
atomic part on the set of parabolic points, since the Poincar\'e series
must converge at $\delta$.  
On the other hand, (4) implies further that 
$\Gamma$ is of divergence type if and only if
Patterson--Sullivan measure has no atoms. Hence, \cite[Theorem
4.1, Theorem 5.2]{MYJ} 
imply that 
all nonatomic quasi-conformal
measures of dimension $\delta>0$ on $\Lambda_\Gamma$ are ergodic and equivalent
up to bounded Radon-Nikodym derivative. 

The existence of a Patterson--Sullivan measure with no atoms is nontrivial:
see for instance the examples of Dal'bo--Otal--Peign\'e \cite[Section 4]{dop},
which arise from geometrically finite Riemannian manifolds of pinched
negative curvature and which have atoms at parabolic points. 
Patterson--Sullivan density is known to have 
no atoms in some settings, such as for geometrically finite 
Riemannian manifolds with pinched negative curvature and parabolic gap
($\delta_\Pi<\delta_\Gamma$ for all parabolic subgroups $\Pi<\Gamma$)
\cite[Proposition 1]{dop}, for relatively
hyperbolic groups acting on their Cayley graph \cite[Proposition
4.1]{YangWenyuan2013}, and for geometrically finite Hilbert geometries
(discussed in Section~\ref{S:PS_hilbert}).  Note that the Cayley graph is
in general not hyperbolic when the group is relatively hyperbolic, but the
construction can still be adapted \cite{YangWenyuan2013}. 

In our hypotheses, we will study quasi-conformal measures on $\Lambda_\Gamma$ with no atoms. 
One appeal of results stated in this generality is that the proof
is intrinsic to these defining properties, rather than the Patterson
construction.

\subsection{The measure of shadows at parabolic fixed points}

Let us now embark on the proof of our global shadow lemma (Theorem \ref{T:intro-global_shadow_lemma}).
Let us remark that, as discussed in Subsection
\ref{S:quasi_conformal_background},
Coornaert \cite[Proposition 6.1]{CoornaertMichel1993}
proved a version of the shadow lemma for shadows that are centered at
points on the orbit $\Gamma o$ of a given basepoint: all such points belong
to the non-cuspidal part of $X$.  In this paper, we generalize this result
by considering shadows centered at \emph{any} point $\xi_t$, in particular
points that may be far from the orbit $\Gamma o$. 

\begin{lemma}
  Let $(X, d)$ be a proper hyperbolic metric space, $\Gamma$ a 
  group of isometries of $X$ acting properly discontinuously on $X$ and
  $\{\mu_x\}_{x\in X}$ a quasi-conformal density of dimension $\delta$ 
  on $\Lambda_\Gamma$. Fix $o\in C_\Gamma$ and $\xi\in \Lambda_\Gamma$, and let $\alpha$ be the hyperbolicity constant of $X$.  
  Then for all $\eta\in V(o,\xi,t)$ and $t\geq 0$,
  \[
    |\beta_\eta(o,\xi_t)-t|\leq O(\alpha)
  \]
  and thus for all $t\geq 0$ and $-t\leq s\leq 0$, 
  \[
    \mu_{\xi_t}(V(o,\xi,t))\asymp e^{-\delta s}\mu_{\xi_{t+s}}(V(o,\xi,t))
  \]
  with uniform constants, independent of $t$ and $s$.
  \label{L:change_perspective_along_geod}
\end{lemma}

\begin{proof}
 To prove the first part, let $p$ be a closest point projection of $\eta$
onto $[o, \xi)$. By Lemma \ref{L:gp-proj}, since $\eta \in V(o, \xi, t)$ we have
$d(o, p) \geq t+O(\alpha)$, hence $\beta_p(o, \xi_t) = t+O(\alpha)$ and by
Lemma \ref{L:Buse-proj} we have for $-t\leq s\leq 0$,
$$\beta_\eta(o, \xi_{t+s}) = \beta_p(o, \xi_{t+s}) + O(\alpha) = t+s + O(\alpha).$$
The first part implies the second part because, by the transformation rule
of conformal densities and the coarse cocycle property of Busemann
functions, 
\begin{align*}
    \mu_{\xi_t}(V(o,\xi,t)) & \asymp \int_{V(o,\xi,t)}
    e^{-\delta\beta_\eta(\xi_t,\xi_{t+s})}\ d\mu_{\xi_{t+s}}(\eta)
    \\
   &  \asymp \int_{V(o,\xi,t)}
    e^{
      -\delta (-\beta_\eta(o,\xi_t) + \beta_\eta(o,\xi_{t+s}))
    } \ d\mu_{\xi_{t+s}}(\eta) \\
&    \asymp e^{-\delta s}\mu_{\xi_{t+s}}(V(o,\xi,t))
\end{align*}
where the
cocycle and antisymmetric properties of the Busemann function
are applied in the second equality.
\end{proof}

Let $\Pi$ be a parabolic subgroup of $\Gamma$. 
In the following propositions we will use the following ``counting functions": for $t \geq 0$
$$f_\Pi(t) := \sum_{\substack{g\in\Pi
   \\d(o,go)\geq 2t}}
   e^{-\delta d(o,go)+\delta t}$$
 and 
$$f_\Pi^c(t) := \#\{ g\in\Pi \ : \ d(o,go) \leq 2t \} e^{-\delta t}.$$

\begin{lemma} \label{L:geom_estimate}
  Let $(X,d)$ be a proper hyperbolic metric space and $\Gamma$ a geometrically finite group 
  of isometries of $X$. 
  Assume $\{\mu_x\}_{x\in X}$ a quasi-conformal density of dimension
  $\delta$ on $\Lambda_\Gamma$ with no atoms. 
  Let $\xi$ be a bounded parabolic point in $\Lambda_\Gamma$  with stabilizer the parabolic subgroup $\Pi$, and $o\in X$. 
  Let $\xi_t$ be a point on a geodesic ray $[o,\xi)$ at distance $t$ from $o$. 
  Then there exist constants $A$ and $C$ depending on $\xi$ and $o$ such that for all $t>A$, 
  \[
 C^{-1}  \mu_{\xi_t}(V(o,\xi,t+A))  \leq   f_\Pi(t)  \leq C \mu_{\xi_t}(V(o,\xi,t-A)), 
 \]
 and 
 \[
   C^{-1}  \mu_{\xi_t}(\partial X\smallsetminus V(o,\xi,t-A))  \leq
   f_\Pi^c(t)
    \leq C \mu_{\xi_t}(\partial  X \smallsetminus V(o,\xi,t+A)). 
  \] 
  \end{lemma}

\begin{proof}
  First let us show the upper bound. By Lemma \ref{L:geometriclemma}(1),
  there is a constant $A$ such that for all $t>A$, 
  \[
    \bigcup_{\substack{g\in \Pi\\d (o,go)\geq2t}} gK \subset V(o,\xi,t-A)
   \]
  where $K$ is a compact fundamental domain for the action of the parabolic
  subgroup $\Pi$ on $\Lambda_\Gamma\smallsetminus\{\xi\}$ given by the
  definition of bounded parabolic. 
  Since  $\Pi$ acts on $\Lambda_\Gamma\smallsetminus\{\xi\}$ properly discontinuously, 
  every point of $\Lambda_\Gamma\smallsetminus\{\xi\}$ is contained in 
  finitely many translates of $K$, hence for some integer $M$
  \begin{equation}
    \sum_{\substack{g\in\Pi \\d (o,go)\geq 2t}} \mu_{\xi_t}(gK) \leq
    M \mu_{\xi_t} \left(
      \bigcup_{\substack{g\in\Pi \\ d (o,go)\geq2t}} gK
    \right)\leq
    M\mu_{\xi_t}(V(o,\xi,t-A)). 
    \label{eqn:estshadow}
  \end{equation}
  Moreover, Lemma \ref{L:geometriclemma}
  gives us control over
  $\beta_{g\eta}(\xi_t,g\xi_t)$ for all such
  $g\in\Pi$ and all $\eta\in K$:
  applying the defining properties of a quasi-conformal density,
  \begin{align}
    \label{eqn:estpK}
    \mu_{\xi_t}(gK) & 
    = \int_{gK} 
    \frac{d\mu_{\xi_t}}{d\mu_{g\xi_t}}(\lambda) 
   \ d\mu_{g\xi_t}(\lambda)
    \asymp \int_{gK} e^{-\delta \beta_\lambda(\xi_t,g\xi_t)}\ d\mu_{g\xi_t}(\lambda)
    \\
    \intertext{and, using Lemma \ref{L:geometriclemma}(2),
    }
  &   \asymp\int_{gK} e^{-\delta(d (o,go)-2t)}\ d\mu_{g\xi_t}(\lambda)
    =e^{-\delta d (o,go)+2\delta t} \mu_{g\xi_t}(gK)
    \\
     & \asymp e^{-\delta d (o,go)+2\delta t} \mu_{\xi_t}(K).
  \end{align}
 Since $K$ is compact and disjoint from $\xi$, there is a constant $D$ such
 that $K$ is disjoint from $V(o, \xi, D)$, hence by Lemma \ref{L:beta-level}, for $t$ sufficiently large
  and $\eta'\in K$, 
    \[
    t- 2 D - O(\alpha) =d (o,\xi_t)- 2 D - O(\alpha) \leq \beta_{\eta'}(\xi_t, o) \leq t.
  \]
  Then another computation using the defining properties of a conformal
  density gives 
  \begin{equation}
    \mu_{\xi_t}(K)\asymp e^{-\delta t}\mu_{o}(K). 
    \label{eqn:estpK2}
  \end{equation}
  Since $K$ is a fundamental domain for the action of the countable group
  $\Pi$ on $\Lambda_\Gamma\smallsetminus\{\xi\}$, and the quasi-conformal
  densities are absolutely continuous by definition and nonatomic by
  assumption, $\mu_o(K)$ is some positive constant, so we obtain a constant $A'$ independent of $t$ such that
  \[
    \frac1{A'}
    f_\Pi(t) \leq \mu_{\xi_t}(V(o,\xi,t-A)).
  \]
  The argument for the lower bound is similar. 
 Next, by Lemma
  \ref{L:geometriclemma}(1), and using that $K$ is a fundamental domain for
  the parabolic subgroup acting on $\Lambda_\Gamma\smallsetminus \{\xi\}$, 
  there is a constant $A$ such that
  \[
  \big(\Lambda_\Gamma\cap V(o,\xi,t+A) \big)\smallsetminus \{\xi\} \subset
    \bigcup_{\substack{g\in\Pi\\d (o,go)\geq 2t}} gK.
  \]
  Then by subadditivity and since the quasi-conformal density is supported
  on $\Lambda_\Gamma$ with no atomic part,  
  \[
    \mu_{\xi_t}(V(o,\xi,t+A))\leq
    \sum_{\substack{g\in\Pi\\d (o,go)\geq 2t}}
    \mu_{\xi_t}(gK).
  \]
  Now, by applying the estimates from Equations \eqref{eqn:estpK} and
  \eqref{eqn:estpK2}, and adjusting the previous constant $A'$ if needed, we
  have 
  \[
    \mu_{\xi_t}(V(o,\xi,t+A)) \leq A' f_\Pi(t). 
  \]
  The estimate for the complement of the shadow is similar and uses Lemma
  \ref{L:geometriclemma} as well, hence the proof is omitted. For more
  details, see \cite[Proposition 3.6]{schapira04}. 
\end{proof}

\begin{lemma}  \label{L:shadowlem_at_parabolic_pts}
 Let $(X,d)$ be a proper hyperbolic metric space and $\Gamma$ a geometrically finite  
 group of isometries of $X$. 
 Assume $\{\mu_x\}_{x\in X}$ is a quasi-conformal density of dimension
 $\delta$ on $\Lambda_\Gamma$ with no atoms. 
For any bounded parabolic fixed point $\xi$ whose stabilizer $\Pi$ has $\delta$-tempered growth, and any
  $o\in X$, there exists a constant $C$ (note that it depends on all the above) such that for all
  $\xi_t$ on a geodesic ray $[o,\xi)$ distance $t$ from $o$,
  \[
   C^{-1}
    B_\Pi(2t) e^{-\delta t} 
    \leq \mu_{\xi_t}(V(o,\xi,t))\leq C 
    B_\Pi(2t) e^{-\delta t} 
  \]
  and
  \[
    C^{-1}
    B_\Pi(2t) e^{-\delta t} 
    \leq \mu_{\xi_t}(\partial X \smallsetminus V(o,\xi,t))\leq C
    B_\Pi(2t) e^{-\delta t} . 
  \]
 \end{lemma}

Lemma \ref{L:shadowlem_at_parabolic_pts} is the crucial point where we assume tempered growth to prove that $f_\Pi(t)\asymp f_\Pi^c(t)$.
Note that, by summing the two equations above, we obtain that the total mass of the measure $\mu_{\xi_t}$ grows like $B_\Pi(2t) e^{-\delta t}$. This is possible because $\xi_t$ is far from the orbit $\Gamma o$. 

\begin{proof}
  Note that we may prove the claim for all $t$ sufficiently large 
  since by adjusting constants, the claim then applies to all $t\geq 0$. 
  Let us write $\Pi$ as the disjoint union 
  \[
    \Pi=\bigcup_{n\in\mathbb N} \{ g\in\Pi : d (o,go)\in[Rn-R,Rn) \}
  \]
and denote for any $n \geq 1$
$$a_n := \#\{ g \in \Pi \ : \ d (o, go) \in [Rn-R, Rn)\}.$$
First, we claim that, by choosing $R$ large enough, we can make sure that 
\begin{equation} \label{E:E3}
\limsup_{n \to \infty} \frac{1}{R} \log \frac{a_{n+1}}{a_n} < \delta.
\end{equation}
Indeed, since $\delta_\Pi^->0$, 
for any $\epsilon > 0$, there exists an $r$ such that for $R\geq r$, 
we have $B_\Pi(T-R) \leq (1-\epsilon) B_\Pi(T)$ for all $T>R$. 
Then  
$$\frac{1}{R} \log \frac{a_{n+1}}{a_n}  = \frac{1}{R} \log \frac{B_\Pi(n R +R)-B_\Pi(n R)}{B_\Pi(n R)-B_\Pi(n R-R)} \leq \frac{1}{R} \log \frac{B_\Pi(n R+R)}{\epsilon B_\Pi(n R)} \leq \frac{1}{R} \log \frac{B_\Pi(n R+R)}{B_\Pi(n R)} + \frac{\log(1/\epsilon)}{R}$$
so, since $\delta_\Pi^+<\delta$, by choosing $R$ large enough, we make sure the right hand side is $< \delta$.
Second, we show that 
$$\sup_n \sum_{k \geq 0} \frac{a_{n+k}}{a_n} e^{- \delta R k} < \infty.$$
Indeed, from \eqref{E:E3} there exists $\delta' < \delta R$, $C > 0$ and $N$ such that 
$$\frac{a_{n+1}}{a_n} \leq e^{\delta'} \qquad  \forall n \geq N \qquad \textup{ and } \frac{a_{n+1}}{a_n} \leq C \qquad \forall n \leq N.$$
Thus, 
$$\sum_{k \geq 0} \frac{a_{n+k}}{a_n} e^{- \delta R k}  = \sum_{k \geq 0} \prod_{j = 1}^k \frac{a_{n+j}}{a_{n+j-1}}  e^{- \delta R k} 
\leq C^N \sum_{k \geq 0} e^{(\delta' - \delta R) k} < \infty
$$
as desired. Then for $t$ sufficiently large, a short calculation gives 
\begin{align*}
    \sum_{d (o,go)\geq 2t} e^{-\delta d (o,go)+\delta t} 
    & \asymp e^{\delta t}\sum_{Rn-R\geq 2t} \sum_{\substack{
      d (o,go)\in\\ [Rn-R, Rn)
      }
    } e^{-\delta
    Rn}  \\
    & 
    = e^{\delta t}\sum_{Rn-R\geq 2t}  a_{n} e^{-\delta Rn}  \\
    \intertext{and, setting $n_0 := \lceil \frac{2t}{R} + 1\rceil$, we have }
    & 
    = e^{\delta t} a_{n_0} e^{-\delta R n_0} \sum_{n \geq n_0} \frac{a_n}{a_{n_0}} e^{-\delta R (n-n_0)}
    \intertext{and, using that $n_0 R = 2 t +O(1)$ and $a_{n_0} \asymp B_\Pi(2t)$, }
    & \asymp B_\Pi(2t) e^{- \delta t} 
\end{align*}
thus 
\begin{equation} \label{E:new-growth}
f_\Pi(t)  \asymp B_\Pi(2 t) e^{- \delta t}.
\end{equation}
Finally, by Lemma \ref{L:geom_estimate}, 
  \[
    \mu_{\xi_{t}}(V(o,\xi,t-A)) \geq C^{-1} B_\Pi(2t)e^{-\delta t}.
 \]
  An analogous argument for the upper bound gives
  \[
    \mu_{\xi_t}(V(o,\xi,t+A))\leq CB_\Pi(2t) e^{-\delta t}.
  \]
 Then by the transformation rule and using that 
  $|\beta_\eta(\xi_t,\xi_{t\pm A})| \leq \pm A$
 we compare $\mu_{\xi_{t\pm A}}(V(o,\xi,t))$ with $\mu_{\xi_t}(V(o,\xi,t))$ to conclude 
  \[
    C^{-1} e^{-\delta A} B_\Pi(2(t+A))e^{-\delta t}\leq
    \mu_{\xi_t}(V(o,\xi,t))\leq 
    C e^{-\delta A} B_\Pi(2(t-A))e^{-\delta t}
  \]
  and the result follows from the fact that $B_\Pi$ is nondecreasing. 

To estimate the complement of the shadow, Lemma \ref{L:geom_estimate} 
immediately yields, by definition, 
   \[
    \sum_{d (o,po)\leq 2t} e^{-\delta t} =  \#\{ g\in\Pi \ : \ d(o,go) \leq 2t \} e^{-\delta t} = B_\Pi(2 t ) e^{- \delta t} 
  \]
 from which the claim follows.
\end{proof}

\subsection{Uniform control over all parabolic fixed points}

Note that so far, the constants depend on a particular parabolic point $\xi$. 

Recall by \cite[Proposition 6.15]{bowditch} that there are finitely many orbits of
parabolic points, hence we express the set of parabolic points $\mathcal P$
as a disjoint union of these orbits $\mathcal P^1,\ldots,\mathcal P^a$. For
each $i=1,\ldots, a$, pick $p_i\in\mathcal P^i$ with stabilizer $\Pi_i$,
and denote 
\[
  B_i(t):=B_{\Pi_i}(t). 
\]
Moreover, we choose a quasi-invariant horoball decomposition
$\{H_\xi\}_{\xi\in \mathcal P}$ of $X$ as given by Proposition
\ref{P:bowditch_horoball_decomposition}. 
We now prove a version of the previous lemma where the constants no longer depend on the particular parabolic point chosen. 

\begin{lemma} \label{L:shadow-parabolic}   \label{L:shadowlem_orbits_parabolics}
  Let $(X,d)$ be a proper hyperbolic metric space and $\Gamma$ a geometrically
  finite group of isometries of $X$.   
  Assume $\{\mu_x\}_{x\in X}$ is a quasi-conformal density
  of dimension $\delta$ on $\Lambda_\Gamma$ with no atoms,
  and that $\Gamma$ has $\delta$-tempered parabolic subgroups. 
  Fix a basepoint $o \in X$ and $i$ with $1 \leq i \leq a$.
For $\xi\in\partial X$, let $\xi_t$ denote a point on a 
  geodesic ray $[o,\xi)$ at distance $t$ from $o$. 
  Then, there exists a constant $C$ such that for all
  $\xi\in\mathcal P^i$ and all times $t > 0$ such that  $\xi_t\in H_\xi$, we have 
  \[
    C^{-1} 
    B_i(2 d( \xi_t,\Gamma o) ) e^{-\delta d( \xi_t,\Gamma o)}
    \leq \mu_{\xi_t}(V(o,\xi,t))\leq C 
    B_i(2 d( \xi_t,\Gamma o) ) e^{-\delta d( \xi_t,\Gamma o)}
  \]
 and 
  \[
    C^{-1} 
    B_i(2 d( \xi_t,\Gamma o) ) e^{-\delta d( \xi_t,\Gamma o)}
    \leq \mu_{\xi_t}(\partial X \smallsetminus V(o,\xi,t))\leq C 
    B_i(2 d( \xi_t,\Gamma o) ) e^{-\delta d( \xi_t,\Gamma o)}. 
  \]
\end{lemma}

\begin{proof}
  Let $\eta=p_i$, a fixed element of $\mathcal P^i$. Let $\xi_s$ be the intersection of $[o,\xi)$ with $\partial H_\xi$. 
  Similarly, let $\eta_{s'}$ be the intersection of $[o,\eta)$ with $\partial H_\eta$. 
  Since the non-cuspidal part is quasi-$\Gamma$-invariant (Proposition \ref{P:bowditch_horoball_decomposition}), 
  any group element $\gamma$ for which $\gamma \eta =\xi$ also takes $H_{\eta}$ within distance $O(\alpha)$ of $H_\xi$. 
  Hence, for any such $\gamma$, we have that $\gamma \eta_{s'}$
  and $\xi_s$ are both within distance $O(\alpha)$ of the $\partial H_\xi$. 
  Since the parabolic stabilizer of $\xi$ acts cocompactly on 
  $\partial
  H_\xi\cap C_\Gamma$, 
  and we can choose a fundamental domain with diameter uniformly bounded
  over all translates of $\eta$, we can choose a particular $\gamma$ such
  that $\gamma \eta_{s'}$ and $\xi_s$ are uniformly bounded distance apart. Denote this bound by $M$, and thus
  \begin{equation}
    d(\xi_{s},\gamma o)
    \leq d(\xi_{s},\gamma \eta_{s'})+d(\gamma
    \eta_{s'}, \gamma o) 
    \leq M+s' =: M'. 
    \label{eqn:final_lemma1}
  \end{equation}
  Then since geodesic rays meeting at the same boundary point $\xi$ are
  asymptotic in a hyperbolic metric space (Lemma \ref{L:asymp}),  
  \[
    d ( \xi_t,\gamma \eta_{t-s})\leq d(\xi_{s},\gamma
    \eta_0)+O(\alpha)=
    d(\xi_{s},\gamma o)+O(\alpha) \leq M'+O(\alpha)
  \]
  as well, and by quasi-conformality of the measures 
  (Equation \eqref{E:transform}) and since 
  Busemann functions are 1-Lipschitz, for any measurable set
  $E\subset \partial X$, 
  \begin{equation}
    \label{eqn:final_lemma2}
    \mu_{\xi_t}(E)\asymp \mu_{\gamma \eta_{t-s}}(E). 
  \end{equation}
  On the other hand, Equation \eqref{eqn:final_lemma1} suffices to apply 
  Lemma \ref{L:technical_shadow_containments}(2); for all points such as $\xi_s$ 
  and $\gamma o$ which are
  bounded distance, there is a constant $C$ depending on this bound
  such that 
  \begin{equation}
    V(\gamma o,\xi,t-s+C) \subset V(\xi_s,\xi,t-s) \subset
    V(\gamma o,\xi, t-s-C),
    \label{eqn:final_lemma3}
  \end{equation}
  and the containments apply in reverse to the complementary shadow. 
 It follows from the definition that there exists a positive $t_0 = O(\alpha)$ such that 
  for any $t \geq s +t_0$ we have 
  $$V(\xi_s, \xi, t - s + t_0) \subseteq V(o, \xi, t) \subseteq V(\xi_s, \xi, t - s - t_0).$$
  Hence  applying Equation \eqref{eqn:final_lemma2},  Equation \eqref{eqn:final_lemma3},  
  quasi-$\Gamma$-invariance of the conformal measures, $\Gamma$-equivariance of the shadows, and Lemma
  \ref{L:shadowlem_at_parabolic_pts} (it does apply because $\xi_t$ is in $H_\xi$, so $t-s$ is positive), 
  \begin{align*}
    \mu_{\xi_t}(V(o,\xi,t)) & \asymp \mu_{\xi_t}(V(\xi_s,\xi,t-s))
    \asymp 
    \mu_{\gamma
      \eta_{t-s}}(V(\gamma o,\gamma \eta,t-s)) 
    \\
      & \asymp
      \mu_{\eta_{t-s}}(V(o, \eta, t-s))\asymp
      B_i(2(t-s))e^{-\delta(t-s)},
  \end{align*}
and again, with similar expressions for the complementary shadow. 
To conclude the proof, see that $t-s+O(\alpha)$
is the distance of $\xi_t$ to $\partial H$ by Lemma
\ref{L:horoball_projection}, which is equal to $d(\xi_t,\Gamma o)$ up to uniform additive constants. 
\end{proof}

\section{Proof of the global shadow lemma}  \label{S:proof-global}

In this section, we complete the proof of the first main result. 
Recall the quasi-invariant horoball decomposition
$\{H_\xi\}_{\xi\in \mathcal P}$ of $X$ as given by Proposition
\ref{P:bowditch_horoball_decomposition} and the decomposition $\mathcal
P=\mathcal P^1\cup\cdots\cup\mathcal P^a$ into the finitely many distinct
orbits of parabolic points in $\partial X$.
Define the {\em $i$th cuspidal part} of $X$ to be 
\[
  X_c^i:=\bigcup_{p\in \mathcal P^i} H_p\cap C_\Gamma.
\]

Recall for each $i$ we choose $p_i\in\mathcal P^i$ with stabilizer
$\Pi_i$, and denote $B_i :=B_{\Pi_i}$. 
Define $b\colon C_\Gamma\to\mathbb R$ by
\begin{equation}
  \label{E:master_B}
  b(x) := \left\{ \begin{array}[]{ll}
      1 & \text{ if }x\in X_{nc} \\ 
      B_i(2d(x,\Gamma o)) & \text{ if }x\in X_c^i. 
  \end{array} \right.
\end{equation}

The main result in full generality is:

\begin{theorem} \label{T:global_shadow_lemma}
    Let $(X,d)$ be a proper hyperbolic metric space and $\Gamma$ a
    geometrically finite group of isometries of $X$. 
  Assume $\{\mu_x\}_{x\in X}$ is a quasi-conformal density
  of dimension $\delta$ on  $\Lambda_\Gamma$ 
  with no atoms, and that $\Gamma$
  has $\delta$-tempered parabolic subgroups. 
  Let $o \in C_\Gamma$, and let $\xi_t$ denote the point on a geodesic ray
  from $o$ to $\xi$ which is distance $t$ from $o$.
Then there exists a constant $C$ such that for any
  $\xi\in \Lambda_\Gamma$ and any $t > 0$ we have 
  \[
    C^{-1} b(\xi_t) e^{-\delta (t + d(\xi_t, \Gamma o))}  
    \leq
    \mu_o(V(o,\xi,t))\leq 
    C b(\xi_t) e^{-\delta (t + d(\xi_t, \Gamma o))}. 
  \]
\end{theorem}

Theorem \ref{T:global_shadow_lemma} is the same as \ref{T:intro-global_shadow_lemma} from the introduction. 

\subsection{Shadows in the non-cuspidal part}

Let us start proving a weak form of the shadow lemma, as in \cite[Proposition 6.1]{CoornaertMichel1993}.

\begin{lemma} \label{L:thickpart}
Let $(X, d)$ be a hyperbolic metric space and $\Gamma$ a
    geometrically finite group of isometries of $X$. Let $\{\mu_x\}_{x\in X}$ be a quasi-conformal density
  of dimension $\delta$ on  $\Lambda_\Gamma$ with no atoms. 
  Then for any $t_0 > 0$ there is a constant $C > 0$ 
such that for all $x$ in $X_{nc}$,  and any $\xi \in \Lambda_\Gamma$, 
  \[
    C^{-1} \leq \mu_x (V(x,\xi, t_0)) \leq C.
  \]
\end{lemma}

\begin{proof}
  Every point in the non-cuspidal part $X_{nc}$ is uniformly bounded distance from the
  $\Gamma$-orbit of $o$ for any fixed point $o$ in $X_{nc}$. Let
  $\gamma o$ be some closest point to $x$ which is in the $\Gamma$-orbit of
  $o$. Then by
  quasi-$\Gamma$-invariance of the measures and equivariance of shadows,
  \[
    \mu_x(V(x,\xi,t_0))\asymp
    \mu_{\gamma o}(V(\gamma o,\xi,t_0) ) = \mu_o(V(o,\xi',t_0))
  \]
  where $\xi'=\gamma^{-1}\xi$ varies over $\Lambda_\Gamma$. 
  
First, we claim that there exists $t >0$ such that for any $\xi, \eta \in \Lambda_{\Gamma}$, 
if $\eta \in V(o, \xi, t)$ then $V(o, \xi, t_1) \subseteq V(o, \eta, t_0)$.

Indeed, $\zeta \in V(o, \xi, t)$ implies $\langle\xi, \zeta \rangle_o \geq
t$. Moreover, if $\eta \in V(o, \xi, t)$ then $\langle\eta, \xi \rangle_o \geq t$,  
hence by Equation \eqref{E:shadow-ball} and the fact that the Gromov metric
is a metric, one gets
\begin{equation*}
  \label{E:compact_shadow_lem1}
  c^{-1}e^{-\epsilon \langle \zeta,\eta\rangle_o}\leq d_{\partial
  X}(\zeta,\eta)\leq d_{\Lambda_\Gamma}(\zeta,\xi)+d_{\Lambda_\Gamma}(\xi,\eta)\leq
  c e^{-\epsilon \langle \zeta,\xi \rangle_o}+ce^{-\epsilon \langle
  \xi,\eta\rangle_o}\leq 2ce^{-\epsilon t}. 
\end{equation*}
Thus, $\langle \zeta,\eta\rangle_o\geq t-\frac{\log(2c^2)}\epsilon\geq t_0$ by taking $t$ large enough, which proves the claim. 

Then by compactness we can cover $\Lambda_\Gamma$ with finitely many shadows of type
$V(o, \xi_i, t)$ for $i = 1, \dots, k$. Now note that, since the support of $\mu_o$ is $\Gamma$-invariant by quasi-conformality 
and the action of $\Gamma$ on $\Lambda_\Gamma$ is minimal, then $\mu_o$ has full support on $\Lambda_\Gamma$. 
Then we have
$$C:= \inf_i \mu_o(V(o, \xi_i, t)) > 0.$$

Now, let $\xi \in \Lambda_\Gamma$. Then there is a $\xi_i$ such that $\xi \in V(o, \xi_i, t)$, hence by the above claim 
we have $V(o, \xi_i, t) \subseteq V(o, \xi, t_0)$, so 
$$\mu_o(V(o, \xi, t_0)) \geq  \mu_o(V(o, \xi_i, t)) \geq C.$$  
  
Now the upper bound is clear, since $\mu_o$ is a finite measure. 
\end{proof}

\begin{figure}[h!]
  \centering
   \begin{tikzpicture}[scale=1.7]
    \def\s{1.5}
    \draw[line width=.5em, 
    opacity=.3] (200: 2) arc
    (200:280:2 );
    \draw (0,0) circle (2); 
    \draw (0,-1) circle (1); 
    \coordinate (xii) at (0,-2); 
    \draw[dashed] (0,-.75) circle (1.25); 
    \draw[draw=none] (240:2) coordinate (xi') -- (50:2) coordinate (xi)
    coordinate[pos=.525] (o') 
    coordinate[pos=.3] (xit) coordinate[pos=.038] (o); 
    \coordinate (gamo) at (50:.58);
    \draw (o)--(xi);
    \draw (gamo) 
    -- (xii) coordinate[pos=.20]
    (gamx) coordinate[pos=.55] (gamxis);

    \node[circle,fill,inner sep=\s] at (xii) {}; 
    \node[below] at (xii) {$\eta$};
    \node[circle,fill,inner sep=\s] at (xi) {};
    \node[above right] at (xi) {$\xi$};
    \node[circle,fill,inner sep=\s] at (xit) {};
    \node[circle,fill,inner sep=\s] at (o') {};
    \node[above left] at (o') {$o'$};
    \node[above left] at (xit) {$\xi_t$};
    \node[circle,fill,inner sep=\s] at (o) {};
    \node[above left, xshift=-.4em] at (o) {$o$};
    \node[circle,fill,inner sep=\s] at (gamo) {};
    \node[above] at (gamo) {$\gamma o$};
    \node[circle,fill,inner sep=\s] at (gamx) {};
    \node[right,yshift=.3em,xshift=.2em] at (gamx) {$\gamma x$};
    \node[circle,fill,inner sep=\s] at (gamxis) {};
    \node[right, xshift=.2em] at (gamxis) {$\gamma \eta_{s}$};

    \node at (-2.3,-1.2) {$V(o',o,t')$};
    \node at (-1.8,1.5) {$X$};
    \node at (.6,-1.5) {$H_\eta$};
  \end{tikzpicture}
   \caption{
    Case 2 in the proof of Theorem \ref{T:global_shadow_lemma}.
  }
  \label{fig:final_case_of_shadowlemma}
\end{figure}

\subsection{Proof of Theorem \ref{T:global_shadow_lemma}}

With all the ingredients established so far, the proof now follows quite closely the strategy of \cite{schapira04}.

\begin{proof}[Proof of Theorem \ref{T:global_shadow_lemma}]
  First, by Lemma \ref{L:change_perspective_along_geod} comparing
  $\mu_{\xi_t}(V(o,\xi,t))$ with $\mu_{\xi_0}(V(o,\xi,t)) =
  \mu_o(V(o, \xi,t))$, it suffices to show that there is a constant $C$ such
  that 
  \begin{equation} \label{E:final}
    C^{-1} b(\xi_t) e^{-\delta d(\xi_t,\Gamma o)}
    \leq \mu_{\xi_t}(V(o,\xi,t)) 
    \leq C b(\xi_t) e^{-\delta d(\xi_t,\Gamma o)}. 
  \end{equation}
  The case where $\xi_t$ is in $X_{nc}$ now follows from Lemma \ref{L:thickpart}; 
 from the definition of shadows, there exists $t_0 = O(\alpha)$ such that 
  \[
    V(o,\xi, t ) \supseteq V(\xi_{t},\xi,t_0 )
  \]
 so Lemma \ref{L:thickpart} applied with $x = \xi_t$ gives the estimate 
\[
    \mu_{\xi_t}(V(o,\xi,t)) \asymp \mu_{\xi_t}(V(\xi_t, \xi,  t_0)) \asymp 1. 
  \]
  The conclusion follows for $\xi_t$ in $X_{nc}$, since all such
  $\xi_t$ are bounded distance from $\Gamma o$ and $b(\xi_t)=1$. 

   It remains to consider the case where $\xi_t$ is in the cuspidal part of $ X$. 
  Let $i$ be such that $\xi_t\in X_c^i$. We denote as $\eta\in\mathcal P^i$ the
  boundary point of the horoball to which $\xi_t$ belongs.  We have three
  cases. 
 
  \smallskip
  \textbf{Case 1}: If $\eta \in V(o, \xi, t+O(\alpha))$, then by Lemma
  \ref{L:contain}(1)
  $$V(o, \eta, t+O(\alpha)) \subseteq V(o, \xi, t) 
  \textrm{ and }
  V(o,\xi,t)\subseteq V(o,\eta,t-O(\alpha)).$$
  By Lemma \ref{L:fellowtravels},  
  $$|\beta_\zeta(\eta_t, \xi_t)| \leq d(\eta_t, \xi_t) \leq O(\alpha)$$
  hence quasi-conformality yields
  $$ C^{-1} \mu_{\eta_t}(V(o,\eta, t + O(\alpha) )) \leq
  \mu_{\xi_t}(V(o,\xi,t))  \leq C \mu_{\eta_t}(V(o,\eta, t-O(\alpha))) $$
  where $C$ depends on $\alpha$ and the quasi-conformality constant. 
  The lower bound follows by Lemma
  \ref{L:shadow-parabolic} and the fact that $d(\eta_t, \Gamma o) = d(\xi_t, \Gamma o) + O(\alpha)$.

  \smallskip

  \textbf{Case 2}: 
  Suppose that $\eta\not\in V(o,\xi,t-O(\alpha))$. 
  Let us introduce some notation; see Figure
  \ref{fig:final_case_of_shadowlemma} for guidance. 
  Let $o'$ denote the intersection point of a geodesic $[o,\xi)$ 
  with the horosphere $\partial H_\eta$ 
  centered at $\eta$ bounding $X_{nc}$, where $o'$ is closer to
  $\xi$ than $\xi_t$. 
  Let $t'=d(\xi_t,o')$. 
By Lemma~\ref{L:gp-proj}, we see $t'$ is chosen so that, for some
$t_1,t_2$ within $O(\alpha)$ of $t'$, 
  \[
    \partial X \smallsetminus V(o',o,t_1) \subset V(o,\xi,t)\subset
    \partial X \smallsetminus V(o',o,t_2).
  \]
  We will now estimate the measure of the left and right hand side by
  comparing them to complements of shadows centered at $\eta$. 
  Let $\eta_t$ be the point on a geodesic ray $[o,\eta)$ which is distance $t$ from $o$. 
  Let $\gamma$ be an element of the stabilizer of $\eta$ such that
  the geodesic ray $\gamma[o,\eta)$ from $\gamma o$ to $\eta$ 
  intersects the same fundamental domain for the action of the stabilizer of $\eta$ on
  $\partial H_\eta$ as $o'$. Let $x\in [o,\eta)$ be such that $\gamma x$ is the
  intersection of the geodesic $\gamma[o,\eta)$ with the horosphere $\partial H_\eta$. 
  In particular, the distance between 
  $\gamma x$ and $o'$ is uniformly bounded, independently of $\eta$. 

  The Case 2 assumption and Lemma~\ref{L:gp-proj} imply
  $\eta$ is in $V(o',o,t'-O(\alpha))$, 
  and by Lemma \ref{L:contain}(1), there exists $t_3,t_4$ within
  $O(\alpha)$ of $t'$,
  such that 
    \[
    V(o',\eta,t_3)\subset V(o',o,t_1)
  \]
  and
      \[
	V(o',o,t_2)\subset V(o',\eta,t_4).
  \]

  Thus, to estimate $\mu_{\xi_t}(V(o,\xi,t))$, it suffices to estimate 
  $$\mu_{\xi_t}(\partial  X \smallsetminus V(o',\eta,t''))$$ 
  for any $t''$ within $O(\alpha)$ of $t'$.
  In order to do so, set $s = t'+d(o,x)$. 
  Chose geodesic representatives $[o,o']\subset [o,\xi)$ and $[x,\eta)\subset [o,\eta)$. 
  Then $\eta$ is in $V(o',o,t'-O(\alpha))$, so 
  by the fellow traveler property of Lemma \ref{L:fellowtravels},   
  $\xi_t$ and $\gamma \eta_s$, which are the points at time $t'$ along
  $[o,o']$ and $\gamma[x,\eta)$ respectively, are uniformly bounded distance apart. 
  Since $\gamma x$ is close to $o'$, we have 
  \begin{align}
    \label{eqn:l3}
    \mu_{\xi_t}(\partial X \smallsetminus  V(o',\eta,t')) 
    &\asymp \mu_{\gamma \eta_s} (\partial  X \smallsetminus  V(\gamma x,\eta,t'))  \\
    \intertext{and, by shifting perspective along the geodesic, we obtain}
    \label{eqn:l34}
    &\asymp \mu_{\gamma \eta_s} (\partial X
    \smallsetminus V(\gamma
    o,\eta,s))  \\
     \intertext{hence, since $\gamma \eta = \eta$, and by quasi-$\Gamma$-invariance, }  
    \label{eqn:l4}
    & \asymp \mu_{\eta_s}(\partial X \smallsetminus
    V(o,\eta,s)) \\
    \intertext{thus recalling that $\eta\in\mathcal P^i$, we have}
   \label{eqn:l5}
   & \asymp B_i(2d(\eta_s,\Gamma o))e^{-\delta d(\eta_s,\Gamma o)}
    \intertext{by direct application of Lemma
      \ref{L:shadowlem_orbits_parabolics}. Finally
    }
    \label{eqn:l6}
   & \asymp B_i(2d(\xi_t,\Gamma o))e^{-\delta d(\xi_t,\Gamma o)}
    \end{align}
    again because $\xi_t$ is uniformly bounded distance from
    $\gamma\eta_s$.
    Note that the tempered growth property implies that the above estimate
    also holds replacing $t'$ with any $t''$
    within $O(\alpha)$ of $t'$. 
    Recalling that $b(\xi_t)=B_i(2d(\xi_t,\Gamma o))$ 
    yields \eqref{E:final}, thus completing case 2. 

    \textbf{Case 3}: Assume $\eta\in V(o,\xi,t-O(\alpha))\smallsetminus
    V(o,\xi,t+O(\alpha))$. 
    Then there exist times $t_1\leq t_2$ such that  $|t_2-t_1|\leq O(\alpha)$ and 
    $\eta \in V(o,\xi,t_1)$  falls into case 1, and 
    $\eta\in V(o,\xi,t_2)$ falls into case 2.
    Then for all $t\in[t_1,t_2]$, see that $|\beta_{\zeta}(\xi_{t_i},\xi_t)|\leq O(\alpha)$ so by monotonicity of shadows and quasi-conformality, 
    we have
    \begin{align*}
      \mu_{\xi_{t_2}}(V(o,\xi,t_2))
      \leq \mu_{\xi_t}(V(o,\xi,t)) 
      \leq       \mu_{\xi_{t_1}}(V(o,\xi,t_1)).
    \end{align*}
    The result now follows because the times $t_i$ were chosen to satisfy
    cases 1 and 2, and because of the 
    tempered growth assumption and 
    boundedness of $|t-t_i|$. 
  \end{proof}

\subsection{Corollaries of the global shadow lemma}

We now prove Corollaries \ref{C:doubling} and \ref{C:singular_harmonic}
from the introduction.

\begin{proof}[Proof of Corollary \ref{C:doubling}]
By Equation \eqref{E:shadow-ball}, there exists $A > 0$ such that 
$$V(o, \xi, \epsilon^{-1}\log(r^{-1}) + A) \subset  D(\xi, r) \subset 
D(\xi,2r)\subset  V(o, \xi,  \epsilon^{-1}\log(r^{-1}) - A)$$
for any $\xi \in \Lambda_\Gamma$, any $r > 0$. 
By the triangle inequality, for any $t \geq A$
$$|d(\xi_{t-A}, \Gamma o) - d(\xi_{t+A}, \Gamma o) | \leq d(\xi_{t-A},
\xi_{t+A}) = 2A.$$
Then setting $t=\epsilon^{-1}\log(r^{-1})$, 
\[
  1\leq \frac{\mu_o(D(\xi,2r))}{\mu_o(D(\xi,r))} \leq C^2 e^{\delta 4A} 
  \frac{b(\xi_{t-A})}{b(\xi_{t+A})}. 
\]
If $\xi_{t-A}$ and $\xi_{t+A}$ are in a horoball $H_\xi$ centered at the same
parabolic point $\xi$ in $\mathcal P^i$, then 
by Remark~\ref{R:BT}(3), there is a constant $C'$ such that for all $t\geq A$, 
\[
  \frac{b(\xi_{t-A})}{b(\xi_{t+A})} = \frac{B_i(2d(\xi_{t-A},\Gamma
  o))}{B_i(2d(\xi_{t+A},\Gamma o))} \leq C'
\]
which yields the estimate. Else, $t \leq A$  hence $d(\xi_{t+A}, \Gamma o) \leq d(o, \Gamma o) + 2A$ is uniformly bounded above, 
so there exist $C'' > 0$ independent of $\xi$ such that $\mu_o(D(\xi, r)) \geq \mu_o(V(o, \xi, t + A)) \geq C''$, 
which completes the proof, since $\mu_o$ is a finite measure. 
\end{proof}

\begin{proof}[Proof of Corollary \ref{C:singular_harmonic}] 
Suppose by contradiction that the harmonic measure $\nu$ and the Patterson--Sullivan measure $\mu$ are in the same measure class. 
By \cite[Proposition 5.1]{GekhtmanTiozzo}, the Radon-Nykodim derivative $\frac{d\mu}{d\nu}$ is bounded away from $0$ and infinity. 
Now, for any $g \in \Gamma$ let $\xi \in \Lambda_\Gamma$ such that $go$ lies within distance $O(\alpha)$ of a geodesic ray $[o, \xi)$. 
By the shadow lemma for the hitting measure (\cite[Proposition 2.3]{GekhtmanTiozzo}), we have 
$$\nu(V(o, \xi, d(o, go))) \asymp e^{- d_G(e, g)}$$
where $d_G$ is the Green distance (see e.g. \cite[Section 2.5]{GekhtmanTiozzo}).
On the other hand, by Theorem \ref{T:intro-global_shadow_lemma},
$$\mu(V(o, \xi, d(o, go))) \asymp e^{- \delta_\Gamma d(o, go)}$$
so, since $\frac{d\mu}{d\nu}$ is bounded above and below, the difference $d_G(e, g) - \delta_\Gamma d(o, go)$
is bounded independently of $g$. 
Since the Green metric $d_G$ is quasi-isometric to any word metric on $\Gamma$, this implies that the orbit map $g \mapsto go$ 
is a quasi-isometric embedding; however, by letting $g = h^n$ with $h$ a parabolic element and taking the limit as $n \to \infty$, we obtain a contradiction. 
\end{proof}

\section{Applications of the shadow lemma} \label{S:applications}

Recall that if $p$ is a boundary point, then $H_p(r)$ is the unique horoball centered at $p$ with radius $r$
and $\mathcal H_{p}(r)$ is the shadow in $\partial X$ of $H_p(r)$.
Recall that $\mathcal P$ is the set of all parabolic fixed points in
$\partial X$, which we decompose as a disjoint union of orbits $\mathcal
P=\mathcal P^1\cup \cdots \cup \mathcal P^a$. 
Then we fix a quasi-$\Gamma$-invariant horoball decomposition of $X$ as given by 
Proposition \ref{P:bowditch_horoball_decomposition}, where each parabolic point $p$ in the set $\mathcal P$
determines a unique horoball $H_p$ centered at $p$ in the decomposition, and we denote the radius of $H_p$ by $r_p$. 

\subsection{Dirichlet Theorem}

We now prove the Dirichlet-type theorem, which does not rely on the shadow
lemma. For fixed $s>0$, recall the set of parabolic points with large radius is denoted by 
\[
  \mathcal P_s:= \{p\in \mathcal P \mid r_p\geq s\}. 
\]
We also denote the set of 
parabolic points in a given orbit with large radius by
\[
  \mathcal P_s^i:= \{p\in \mathcal P^i \mid r_p\geq s\}.
\]

\begin{theorem}[Dirichlet-type theorem] \label{T:Dirichlet}
  Let $(X,d)$ be a proper hyperbolic metric space and $\Gamma$ a geometrically finite group of isometries of 
  $X$ with parabolic elements. 
  Then there exist constants $c_1>0,c_2\geq 1$ such that for all $s\leq c_1$, the set 
  \[
    \bigcup_{p\in \mathcal P_s} \mathcal H_p(c_2\sqrt{sr_p})
  \]
  covers the limit set $\Lambda_\Gamma$, and there exists $0< c_3 \leq 1$ such that the shadows 
  $\{\mathcal H_p(c_3\sqrt{sr_p})\}_{p\in \mathcal P_s} $ are pairwise disjoint.
  \end{theorem}

Note that Theorem \ref{T:Dirichlet} is effectively the same statement as
Theorem \ref{T:intro_dirichlet}.

\begin{proof}
First, by cocompactness of the action of $\Gamma$ on the 
non-cuspidal part, note that we can rescale all
horoballs in each of the finitely many $\Gamma$-orbits of parabolic points
by a multiplicative constant $c$ 
so that the convex hull of the limit set $C_\Gamma$ is covered by the
horoballs rescaled by $c$, i.e. 

\begin{equation}
  C_\Gamma \subseteq \bigcup_{\substack{p\in \mathcal P}} 
  H_{p}(cr_p).
  \label{E:cvxhull_containment}
\end{equation}
Now fix $0<s \leq c_1:= \frac1c $ and $\xi\in\Lambda_\Gamma$.
Let $w\in[o,\xi)$ such that $e^{-d(o,w)}=cs$. 
By the above Equation \eqref{E:cvxhull_containment}, there is some $p\in \mathcal P$ such that 
$w\in H = H_{p}(cr_p)$.  Let $q$ be a point on the intersection of $[o,p)$ with $\partial H$, 
so that $cr_p = e^{-\beta_p(o,q)}$ by the definition of radius of a horoball.  
Since $w\in H$, we have $\beta_p(o,w)\geq \beta_p(o,q)$. 
Since $w\in[o,\xi)$, we apply Lemma \ref{L:sqrt} and conclude
there exists a point $z$ on $[o, \xi)$ with 
\[
  \beta_p(o, z) \geq \frac{d(o, q) + d(o, w)}{2} - O(\alpha).
\]
Then there exists a constant $c_2$ such that 
\[
  e^{-\beta_p(o, z)} \leq e^{ - \frac{d(o, q) + d(o, w)}{2}} e^{O(\alpha)}
  = c_2 \sqrt{r_p s}
\]
which shows that $z$ belongs to $H_{p}(c_2 \sqrt{r_p s})$, 
hence also $\xi$ belongs to $\mathcal{H}_{p}(c_2 \sqrt{r_p s})$.
Finally, observe that $s\leq r_p$ since
\[
  cs = e^{-d(o,w)} \leq  e^{-\beta_p(o,w)}\leq   e^{-\beta_p(o,q)} = 
  cr_p.
\]
To prove the second part, note that, since the horoballs are disjoint, we have by Corollary \ref{C:disjoint}
that there exists a constant $C > 0$ for which 
$$d_{\partial X}(p_1, p_2) \geq C (r_1 r_2)^{\frac\epsilon2}.$$
Now, by Lemma \ref{L:diameter}, there exists a constant $c_3$ such that for each $i = 1,2$ 
one has the following bound on the diameter of the shadow
$$\textup{diam}\ \mathcal{H}_{p_i}(c_3 \sqrt{r_i s}) \leq \frac{C}{4} (r_i
  s)^{\frac\epsilon2}.$$
Hence, using that $r_i \geq s$, the inequalities 
$$d_{\partial X}(p_1, p_2) \geq  C (r_1 r_2)^{\frac\epsilon2}  \geq \frac{C}{2} (s
r_1)^{\frac\epsilon2}  + \frac{C}{2} (s r_2)^{\frac\epsilon2} > \textrm{diam }\mathcal
H_{p_1}(c_3 \sqrt{r_1s}) + \textrm{diam }\mathcal H_{p_2}(c_3 \sqrt{r_2s})$$
show that the shadows $\mathcal{H}_{p_1}(c_3 \sqrt{r_1 s})$ and $\mathcal{H}_{p_2}(c_3 \sqrt{r_2 s}) $ are disjoint. 
\end{proof}

\subsection{Horoball counting} 

Now we will apply the Dirichlet theorem to produce horoball counting
estimates. We need a version of the shadow lemma for shadows of horoballs
rather than traditional shadows. The following condition will be the main hypothesis 
on the measures for the remaining application, so we introduce it as a definition. 

\begin{definition}[Horoball shadow lemma] 
  \label{D:horoball_shadow_lemma}
Let $(X,d)$ be a proper hyperbolic metric space and $\Gamma$ a geometrically finite group of isometries of $X$.
We say that a measure $\mu$ on $\partial X$ \emph{satisfies the horoball
shadow lemma with dimension} $\delta$ if for all 
$c_1 < c_2$ there there exists a multiplicative constant such that 
  \begin{equation}
    \mu(\mathcal H_p(c\theta r_p))\asymp
    B_i(-2\log\theta)\theta^{2\delta} r_p^\delta. 
    \label{E:horoball_shadow_lemma1}
  \end{equation}
  for any $0<\theta \leq 1$, any $c \in [c_1, c_2]$, any $p \in \mathcal P^i$, and any $i = 1, \dots, a$.
\end{definition}

The measures we have considered so far satisfy this property: 

\begin{corollary}
  \label{C:horoball_shadow_lemma}
  Let $(X,d)$ be a proper hyperbolic metric space and $\Gamma$ a
  geometrically finite group of isometries of $X$ with $\delta$-tempered
  parabolic subgroups.  
  Then any $\delta$-quasi-conformal measure $\mu$ on
  $\Lambda_\Gamma$ with no atoms 
  satisfies the horoball shadow lemma with dimension $\delta$.
  \end{corollary}

\begin{proof}
  Letting $t=-\log(c\theta r_p)$, see that by definition of $r_p$, we have 
  $-\log(\theta)-k\leq d(\xi_t,\Gamma o)\leq -\log(\theta)+k$ 
  for some constant $k$ depending only on $\alpha$ and the fixed interval containing $c$. 
  Then Equation \ref{E:horoball_shadow_lemma1} follows from 
  the global shadow lemma (Theorem \ref{T:global_shadow_lemma}), Lemma \ref{L:horoball_shadows_compare}, and Remark~\ref{R:BT}(1) and (3). 
\end{proof}

We will now prove a horoball counting statement, analogous to \cite[Theorem 3]{StratmannVelani}.

\begin{proposition}[Horoball counting] \label{P:counting}
  Let $(X,d)$ be a proper hyperbolic metric space and $\Gamma$ a geometrically finite group of isometries of $X$. 
  Assume $\Gamma$ has $\delta$-tempered parabolic subgroups, and $\mu$ is a
  measure on $\Lambda_\Gamma$ that satisfies the horoball shadow lemma with
  dimension $\delta$ (Definition \ref{D:horoball_shadow_lemma}). 
  Let us define
  $$\mathcal P_n(\lambda) := \{ p \in \mathcal P,  \  \lambda^{n+1} \leq r_p \leq \lambda^n \}.$$
Then there exist $\lambda<1$ and constants such that 
$$\# \mathcal P_n(\lambda) \asymp \lambda^{-n \delta}$$
for all $n\in\mathbb N$. 
\end{proposition}

\begin{proof}
  Let $c_1,c_2$ and $c_3$ be as in Theorem \ref{T:Dirichlet}. 
  Then for all $0<s<c_1$, 
  \[
    \mu(\Lambda_\Gamma) \leq  \sum_{p\in \mathcal P_s} \mu(\mathcal H_p(c_2\sqrt{sr_p}))
  \]
  and 
  \[
    \sum_{p\in \mathcal P_s}
    \mu(\mathcal H_p(c_3\sqrt{sr_p}))
    \leq \mu(\Lambda_\Gamma). 
  \]
  Then applying the horoball shadow lemma (Definition
  \ref{D:horoball_shadow_lemma}) with $\theta=\sqrt{s/r_p}$ to get 
  \[
    \sum_{i=1}^a\sum_{p\in\mathcal P_s^i} B_i \left(-\log\left({\frac{s}{r_p}}\right) \right) \left(\frac{ s}{r_p}\right)^{\delta}r_p^\delta  \asymp 1
  \]
  hence there is a constant $c>0$ such that 
  \begin{equation}
    \label{E:sdelta}
    c^{-1}s^{-\delta}\leq 
    \sum_{i=1}^a\sum_{p\in\mathcal P_s^i} B_i
    \left(\log\frac{r_p}s\right) \leq cs^{-\delta}.
  \end{equation}
  Now by finiteness of the upper annular growth rate, for $u$ sufficiently
  large, $t\geq \log(u)$, and $i=1,\ldots,a$, 
  \[
    \frac{B_i(t)}{B_i(t-\log u)}\leq u^{\delta_{\Pi_{i}}^+}
  \]
  where $\Pi_i$ is the stabilizer of a fixed $p_i\in \mathcal P$.

  Then since $\Gamma$ has $\delta$-tempered parabolic subgroups, we may fix 
  a sufficiently large $u$ such that 
\begin{align*}
    \sum_{i=1}^a\sum_{\substack{p\in\mathcal P^i\\ su \leq r_p } }
    B_i(\log(\frac{r_p}s)) = \sum_{i=1}^a\sum_{\substack{p\in\mathcal P^i\\ su \leq r_p } } 
    B_i(\log(\frac{r_p}{su})) \frac{B_i(\log(\frac{r_p}s))}{B_i(\log(\frac{r_p}{su}))} 
    \leq 
    c (su)^{-\delta} u^{\max_{i}{\delta_{\Pi_i}^+}}
    \leq \frac12 c^{-1}s^{-\delta}
\end{align*}
where $c$ is given by Equation \eqref{E:sdelta}. Then 
  \begin{align*}
    \# \{p\in \mathcal P : s\leq r_p\leq su\} & \asymp
    \sum_{i=1}^a\sum_{\substack{p\in\mathcal P^i\\ s\leq r_p < su}}
    B_i(\log(\frac{r_p}s)) \\ 
    & = 
    \sum_{i=1}^a\sum_{\substack{p\in\mathcal P^i\\ s\leq r_p }}
    B_i(\log(\frac{r_p}s)) 
    - \sum_{i=1}^a\sum_{\substack{p\in\mathcal P^i\\ su \leq r_p } }
    B_i(\log(\frac{r_p}s)) \\
    &\geq c^{-1}s^{-\delta} - \frac12c^{-1} s^{-\delta} \geq
    \frac12 c^{-1} s^{-\delta}. 
  \end{align*}
  Equation \eqref{E:sdelta} and nonnegativity of $B_i$
  implies the same expression is bounded above by $cs^{-\delta}$. 
  Taking $\lambda=u^{-1}$ and $s=\lambda^{n+1}$ proves the statement. 
\end{proof}

\begin{remark}
  In fact, Proposition~\ref{P:counting} also follows from
  \cite[Theorem 1.7]{yang19}, though our approach is different.
  Yang proves that 
the exponential growth of horoballs in Proposition
\ref{P:counting} is equivalent to the
{\em Dal'bo--Otal--Peign\'e (DOP) condition}. We say that $\Gamma$
satisfies the DOP condition if for every parabolic subgroup $\Pi$ of
$\Gamma$, 
\[
  \sum_{g\in\Pi} d(o,go) e^{-\delta_\Gamma d(o,go)}<\infty.
\]
It is straightforward to verify that if $\Gamma$ has tempered growth then
$\Gamma$ has the DOP
condition. Additionally, in the language of Yang, geometrically finite actions are cusp-uniform and
measures satisfying  Definition~\ref{D:horoball_shadow_lemma} have no
atoms, which implies $\Gamma$ is of divergent type (see
Section~\ref{S:quasi_conformal_background}).
Thus, \cite[Theorem 1.7]{yang19} implies Proposition~\ref{P:counting}.
  \label{R:compare_to_yang}
\end{remark}

\begin{proposition}[Horoball counting for distinct orbits] \label{P:distinct-orbits}
  Let $(X,d)$ be a proper hyperbolic metric space and $\Gamma$ a
  geometrically finite group of isometries of $X$. 
  Assume $\Gamma$ has $\delta$-tempered parabolic subgroups, 
  and $\mu$ is a measure on $\Lambda_\Gamma$ that satisfies the horoball shadow lemma with
  dimension $\delta$ (Definition \ref{D:horoball_shadow_lemma}). 
  For each $i = 1, \dots, a$, let us define
  $$\mathcal P_n^i(\lambda) := \{ p \in \mathcal P^i,  \  \lambda^{n+1} \leq r_p \leq \lambda^n \}.$$
  Then there exists a multiplicative constant such that for $\lambda<1$ sufficiently small, 
$$\# \mathcal P_n^i(\lambda) \asymp \lambda^{-n \delta}$$
for all $n\in\mathbb N$ and all $i = 1, \dots, a$. 
\end{proposition}

\begin{proof}
For each $i = 1, \dots, a$, choose a parabolic point $p_i \in \mathcal P^i$, and let $\Pi_i := \textup{Stab}(p_i)$ be its stabilizer. 
For each $t > 0$, consider the function 
$$f_i(t) := \# \{ g \Pi_i \in \Gamma/\Pi_i  \ : \ d(o, g \Pi_i o) \leq t \}.$$
Then, by \cite[Theorem 3.1]{hersonskypaulin04},  for any $i, j$ and any $t > 0$ we have
\begin{equation} \label{E:fi-fj}
f_i(t) \asymp \#\{ g \in \Gamma \ : \ d(o, go) \leq t \} \asymp f_j(t).
\end{equation}
Now, note that, there exists a constant $C$, depending on $\alpha$ and the diameter of a fundamental 
domain for the action of $\Pi_i$ on the horosphere containing $\Pi_i o$, such that for any $g \in \Gamma$,  
$$| d(o, g \Pi_i o) + \log (r_{g p_i}) | \leq C$$
hence, by definition of $\mathcal P_n^i(\lambda)$ and Equation \eqref{E:fi-fj}, for any $i, j$ and any  $n \geq 0$, 
$$\# \mathcal P_n^i(\lambda) \asymp  f_i(n \log(1/\lambda)) \asymp  f_j(n \log(1/\lambda))  \asymp \# \mathcal P_n^j(\lambda).$$
Since $\# \mathcal P_n(\lambda)$ is the sum of the finitely many $\# \mathcal P_n^i(\lambda)$, we obtain by Proposition \ref{P:counting}
$$\# \mathcal P_n^i(\lambda) \asymp \# \mathcal P_n(\lambda) \asymp \lambda^{-n \delta}.$$
\end{proof}

\subsection{Khinchin functions}

A \emph{Khinchin function} is a positive, increasing function $\varphi : \mathbb{R}^+ \to (0, 1]$ 
such that there exist constants $b_1<1, b_2 >0 $ for which 
$$\varphi(b_1 x) \geq b_2 \varphi(x) \qquad \textup{for any }x \in \mathbb{R}^+.$$
Note that it follows that for any $k_1\in(0,1)$ there exists a $k_2$ such that
$\phi(k_1x)\geq k_2 \phi (x)$ for all $x\in\mathbb R^+$. 

Khinchin functions have been introduced in diophantine approximation: Khinchin's classical theorem 
\cite{Khintchine} states that the set of reals $x$ such that $|x - \frac{p}{q}| < \frac{\psi(q)}{q}$ for infinitely many rationals $\frac{p}{q}$ has measure zero if and only if $\sum_{q=1}^\infty \psi(q) < \infty$, 
and full measure otherwise. The function $\phi$ we are using in this paper is related to $\psi$ by the formula $\phi(x) = \psi(\log x^{-1})$. 
As a famous example, define $\varphi(x) = (\log x^{-1})^{-(1+\epsilon)}$
for $x<e^{-1}$ and equal to 1 otherwise. Then using Khinchin's original theorem
one proves that the set $\{x \in \mathbb{R} \ : \  \exists \textup{
infinitely many }\frac{p}{q} \textup{ with } |x - \frac{p}{q}| <
\frac{1}{q^{2 + \epsilon}} \}$
has zero Lebesgue measure if $\epsilon > 0$, and full measure if $\epsilon = 0$.

Now fix $(X, d)$ a hyperbolic metric space, and $\Gamma$ a geometrically finite group of isometries of $X$. 
Let $\mu$ be a quasi-conformal measure  on $\Lambda_\Gamma$ with no atoms. 
Recall that $H_{p}(r)$ is the unique horoball centered at the boundary point
$p$ with radius $r$. Note that for any measure $\mu$ which satisfies the horoball shadow lemma 
(Definition \ref{D:horoball_shadow_lemma}) and 
for any Khinchin function $\phi$,
\begin{equation} \label{E:shad}
\mu(\mathcal{H}_{p}(r_p \varphi(r_p))) \asymp r_p^\delta (\varphi(r_p))^{2 \delta} B_i(-2 \log \varphi(r_p))
\end{equation}
where $p$ belongs to $\mathcal P^i$. 

\subsection{Quasi-independence}

For $i = 1, \dots, a$, let $S_n^i$ be the union of the shadows of $H_{p}(r_p\phi(r_p))$ for $\lambda^n\leq r_p \leq \lambda^{n+1}$ and $p \in \mathcal{P}^i$. 
Note that $S_n^i$ depends on $\lambda$ and $\phi$.

Given a horoball $H$ of radius $r$ and a function $f : \mathbb{R}^+ \to \mathbb{R}^+$, we denote as 
$f H$ the horoball with the same boundary point as $H$ and radius $r f(r)$. 

\begin{lemma}[Quasi-independence]  \label{L:quasi_indp}
  Let $(X, d)$ be a hyperbolic metric space and $\Gamma$ a geometrically
  finite group of isometries of $X$. 
  Assume $\Gamma$ has $\delta$-tempered parabolic subgroups, and 
  $\mu$ is a
  measure on $\Lambda_\Gamma$ that satisfies the horoball shadow lemma with
  dimension $\delta$ (Definition \ref{D:horoball_shadow_lemma}). 
Fix a Khinchin function $\phi$. 
  Then there exists a positive constant $C$ such that for all $i, j$, for all $n,m\in\mathbb N$ sufficiently large, 
  and for all $\lambda<1$ sufficiently small, 
  \[
    \mu(S_n^i \cap S_m^j)\leq C\mu(S_n^i)\mu(S_m^j).
\]  
\end{lemma}

\begin{proof} 
  Let $\lambda<1$ be sufficiently small 
  as given by Proposition \ref{P:distinct-orbits}.
We denote as $S(H)$ the shadow of the horoball $H$. 

Let $r_i=r_{p_i}$ for $i=1,2$. 
Let $H_{p_1} = H_{p_1}(r_1)$  and $H_{p_2} = H_{p_2}(r_2)$ be two disjoint horoballs. 
By Corollary \ref{C:disjoint}, we obtain
  \[ 
    d_{\partial X}(p_1, p_2) \geq C_\alpha (r_1 r_2)^{\frac\epsilon2}. 
  \]
where $C_\alpha > 0$ only depends on the hyperbolicity constant. 

\medskip
\textbf{Claim.} 
There exists a constant $c$ such that, if $r_1 > r_2$ and 
$S(\phi H_{p_1}) \cap S( \phi H_{p_2})\neq \emptyset$,
then 
\begin{equation} \label{E:inclusion}
S (H_{p_2}) \subseteq S (c \varphi H_{p_1}).
\end{equation}

\begin{proof}[Proof of the claim]
Since $\phi$ is increasing,
$$d_{\partial X}(p_1, p_2) \leq C  (\varphi(r_1) r_1)^\epsilon + C
(\varphi(r_2) r_2)^\epsilon \leq 2 C (\varphi(r_1) r_1)^\epsilon$$
where $C$ comes from Lemma \ref{L:diameter}, hence 
 $$C_\alpha (r_1 r_2)^{\frac\epsilon2} \leq 2 C(\varphi(r_1) r_1)^\epsilon$$
thus since $\varphi$ is increasing and $r_1$ is bounded, $\phi$
and hence $\phi^\epsilon$ is bounded by some constant $M$ and 
  $$\frac{M(\varphi(r_1) r_1)^{\epsilon}}{r_2^\epsilon} \geq \left( 
  \frac{\varphi(r_1)^2 r_1}{r_2} 
\right)^{\epsilon}
  \geq \frac{C_\alpha^2}{4C^2}.$$
Hence, if $\xi \in \mathcal{H}_{p_2}(r_2)$, we estimate
$$d_{\partial X}(\xi, p_1) \leq d_{\partial X}(\xi, p_2) + d_{\partial
X}(p_2, p_1) \leq C r_2^\epsilon + 2 C (\varphi(r_1) r_1)^\epsilon \leq c
(\varphi(r_1)r_1)^\epsilon$$
with $c = \frac{M4 C^3}{C_\alpha^2} + 2 C$, which implies Equation \eqref{E:inclusion}.
\end{proof}

Now let $m > n$, and pick an element $p_\star$ of $\mathcal P_n^i(\lambda)$. 
Let us consider the set 
$$I(p_\star) := \{ p \in  \mathcal P^j_m(\lambda) \ : \ S(\varphi H_p) \cap S(\varphi H_{p_\star}) \neq \emptyset \}.$$
By the horoball shadow lemma (Definition \ref{D:horoball_shadow_lemma}), 
for any $p \in \mathcal P_m^j(\lambda)$ we have 
\begin{equation} \label{E:e1}
\mu(S(H_p)) \asymp \lambda^{m \delta}
\end{equation}
while by the counting lemma (Prop. \ref{P:counting})
\begin{equation} \label{E:e2}
\# \mathcal P_m^j(\lambda) \asymp \lambda^{-m \delta}.
\end{equation}

By Theorem \ref{T:Dirichlet}, setting $s = \lambda^{m+1}$, there exists $c_2$ such that the shadows in the set 
$$\Sigma := \{ S(c_2 H_{p}) \ : \  p \in \mathcal P_m^j(\lambda) \}$$ 
are mutually disjoint. Since $\phi(x) \leq 1$, we also have that the shadows in the set
$$\Sigma_\phi := \{ S(c_2 \phi H_{p}) \ : \  p \in \mathcal P_m^j(\lambda) \}$$ 
are mutually disjoint. 
Now, by the horoball shadow lemma (Definition \ref{D:horoball_shadow_lemma}), we have 
$$\mu(S(c_2 H_p)) \asymp \mu(S( H_p)), \qquad \mu(S(c_2 \phi H_p)) \asymp \mu(S( \phi H_p))$$
for any $p \in \mathcal P_m^j(\lambda)$.
Hence, since the elements of $\Sigma_\phi$ are pairwise disjoint, applying Equations \eqref{E:e1} and \eqref{E:e2} we obtain, 
for any $p \in \mathcal P_m^j(\lambda)$,
\begin{equation}
  \mu(S_m^j) \asymp \# \mathcal P_m^j(\lambda) \cdot \mu(S( \varphi H_p) ) \asymp \frac{\mu(S(\varphi H_p) )}{\mu(S(H_p))}. 
  \label{E:countsm}
\end{equation}
Note that the same argument implies $\mu(S_n^i)\asymp \sum_{p\in\mathcal P_n^i(\lambda)}
\mu(S(\phi H_p))$, even if the union is not disjoint.
Now, note that, if $p \in I(p_\star)$, then by Equation \eqref{E:inclusion}
$$S(H_p) \subseteq S( c \varphi H_{p_\star}).$$
Moreover, since the elements of $\Sigma$ are disjoint and $\mu(S(c_2 H_p)) \asymp \mu(S( H_p))$, 
$$\mu(S(\varphi H_{p_\star})) \asymp \mu(S(c \varphi H_{p_\star})) \gtrsim \# I(p_\star) \inf_{p \in I(p_\star)} \mu(S(H_p))$$
hence 
\begin{align*}
\mu(S_n^i \cap S_m^j) & \leq \sum_{p_\star \in  \mathcal P_n^i(\lambda)} \sum_{p \in I(p_\star)} \mu(S( \varphi H_p )) \\
 & \leq \sum_{p_\star \in  \mathcal P_n^i(\lambda)} \# I(p_\star) \sup_{p \in I(p_\star)} \mu(S(\varphi H_p)) \\
& \lesssim \sum_{p_\star \in \mathcal P_n^i(\lambda)}  \frac{\mu(S(\varphi H_{p_\star}))}{\inf_{p \in I(p_\star)} \mu(S(H_p))} \sup_{p \in I(p_\star)}   \mu(S(\varphi H_p)) \\
& \lesssim  \mu(S_n^i) \frac{ \sup_{p \in I(p_\star)}   \mu(S(\varphi H_p))}{ \inf_{p \in I(p_\star)} \mu(S(H_p))}  \asymp  \mu(S_n^i) \mu(S_m^j)
\end{align*}
where the last comparison follows by Equation \eqref{E:countsm}.
This completes the proof.
\end{proof}

\subsection{Khinchin theorem}

Given a Khinchin function $\varphi$, a small enough $\lambda<1$, and $i = 1, \dots, a$, we define the set
$$\Theta_{\lambda}^{i}(\phi) := \limsup_{n\to\infty} S_n^i =  \bigcap_{n = 0}^\infty \bigcup_{m \geq n}
\bigcup_{\substack{p \in \mathcal P_m^i(\lambda)}} \mathcal{H}_{p}(r_p
\varphi(r_p)). $$
Moreover, we have the \emph{Khinchin series} 
$$K_{\lambda}^i(\phi)  :=  \sum_{n = 0}^\infty \phi(\lambda^n)^{2\delta} B_i(-2 \log \varphi(\lambda^n)).$$
Similarly, we define 
$$\Theta_{\lambda}(\phi) := \bigcup_{i = 1}^a \Theta_{\lambda}^i (\phi)
\qquad \textup{and}  \qquad K_{\lambda}(\varphi) := \sum_{i = 1}^a K^i_{\lambda}(\varphi).$$
We are now ready to state the main theorem of this subsection.

\begin{theorem}[Khinchin-type theorem] \label{T:khin}
  Let $(X, d)$ be a proper hyperbolic metric space and $\Gamma$ a geometrically
  finite group of isometries of $X$. 
  Assume $\Gamma$ has $\delta$-tempered parabolic subgroups, and 
  $\mu$ is a quasi-conformal probability measure of dimension $\delta$ on $\Lambda_\Gamma$ with no atoms. 
Let $\varphi$ be a Khinchin function. Then there exists a $\lambda<1$ such
that 
for each $i = 1, \dots, a$:
\begin{enumerate}
\item $\mu(\Theta_{\lambda}^{i}(\phi) ) = 0$ if $K_{\lambda}^i(\phi)  < \infty$;
\item
$\mu(\Theta_{\lambda}^{i}(\phi)) = 1$ if $K_{\lambda}^i(\phi) = \infty$.
\end{enumerate}
As a consequence, $\mu(\Theta_{\lambda}(\phi)) = 0$ if and only if $K_{\lambda}(\phi)  < \infty$, and otherwise $\mu(\Theta_{\lambda}(\phi) ) = 1$.
\end{theorem}

To prove this, let us recall the Borel-Cantelli lemma and its converse (for a proof see e.g. \cite{Lamperti}): 
\begin{lemma} \label{L:BC}
Let $(S, \mathbb{P})$ be a probability measure space, and $(A_n) \subseteq S$ a sequence of measurable subsets. Then: 
\begin{enumerate}
\item
If $\sum_{n = 0}^\infty \mathbb{P}(A_n) < \infty$, then $\mathbb{P}(\limsup A_n) = 0$.
\item
If  $\sum_{n = 0}^\infty \mathbb{P}(A_n) = \infty$ and there exists $c > 0$ such that
$\mathbb{P}(A_n \cap A_m) \leq c \mathbb{P}(A_n) \mathbb{P}(A_m)$ for any distinct $n, m \geq 0$, then $\mathbb{P}(\limsup A_n) > 0$. 
\end{enumerate}
\end{lemma}

\begin{proof}[Proof of Theorem \ref{T:khin}]
  Fix $\lambda<1$ small as given by Proposition \ref{P:distinct-orbits}. 
  Note that by Corollary \ref{C:horoball_shadow_lemma} and 
  Equation \eqref{E:shad}, for any $p \in \mathcal P_n^i(\lambda)$, 
  \[ 
  \mu(S_n^i)\asymp \# \mathcal P_n^i(\lambda) \cdot\mu( S(\phi H_p) ) \asymp
    \lambda^{-n\delta}\lambda^{n\delta}\phi(\lambda^n)^{2\delta} B_i(-2 \log \varphi(\lambda^n))= \phi(\lambda^n)^{2\delta} B_i(-2 \log \varphi(\lambda^n)).
  \]
Now, (1) follows from Lemma \ref{L:BC} (1). 

Conversely, (2) follows Lemma \ref{L:BC} (2):  using the quasi-independence from Lemma \ref{L:quasi_indp} we obtain that
$\mu(\Theta_{\lambda}^{i}(\phi)) > 0$. 
Moreover, from \cite[Lemma 1.2.3]{stratmann}, since $\phi$ is a Khinchin function, we have that $\Theta_{\lambda}^{i}(\phi)$ is $\Gamma$-invariant up to measure zero, meaning that 
for any $g \in \Gamma$ we have $\mu(g \Theta_{\lambda}^{i}(\phi) \Delta \Theta_{\lambda}^{i}(\phi)) = 0$
(their proof is stated in the convex cocompact case, but the same proof applies here). 
Thus, from ergodicity of nonatomic quasi-conformal densities \cite[Theorem 4.1]{MYJ} 
we conclude that $\mu(\Theta_{\lambda}^{i}(\phi)) = 1$. 
\end{proof}

\subsection{The logarithm law}

We now state and prove the logarithm law in the general case of hyperbolic metric spaces. 
The following result compares to \cite[Proposition 4.9]{StratmannVelani}.

\begin{theorem}[Logarithm Law] 
  \label{T:loglaw}
  Let $(X, d)$ be a proper hyperbolic metric space and $\Gamma$ a geometrically
  finite group of isometries of $X$. 
  Assume $\Gamma$ has $\delta$-tempered parabolic subgroups, and 
  $\mu$ is a quasi-confomal measure of dimension $\delta$ on
  $\Lambda_\Gamma$ with no atoms. 
  If the parabolic subgroups of $\Gamma$ moreover have mixed exponential growth, 
  then for $\mu$-almost every $\xi$ in the limit set $\Lambda_\Gamma$, 
  \[
    \limsup_{t\to+\infty} \frac{d(\xi_t,\Gamma o)}{\log t} =
    \frac1{2(\delta-\delta_{\max})}
  \]
  where $\delta_{\max}$ is the maximal growth rate of any parabolic subgroup, 
  and $\xi_t$ is the point on a geodesic ray $[o,\xi)$ that is distance $t$ from $o$.
\end{theorem}

We will see in Section \ref{S:final_result_hilbert}
  how
Theorem \ref{T:loglaw} implies Theorem \ref{thm-intro:loglaw}.

\begin{proof}
 We recall the set-up for the proof provided in Stratmann-Velani
 \cite{StratmannVelani}. For any $\epsilon$ we define for $0\leq x\leq
 e^{-1}$
  \[
    \phi_{\epsilon}(x) := (\log x^{-1})^{-\frac{1+\epsilon}{2(\delta-  \delta_{\max})} }
  \]
  where $\delta_{\max} := \max \{ \delta_{\Pi_i}, 1 \leq i \leq a \} <
  \delta$, and for $x\geq e^{-1}$ we let $\phi_{\epsilon}(x):=1$.
  Observe that $\phi_\epsilon$ is a Khinchin function
  and that $\phi_\epsilon $, hence
  $\Theta_\lambda(\phi_\epsilon)$, is decreasing in $\epsilon$. 
  
 Now, we claim that the Khinchin series $K(\phi_{\epsilon})$ converges if $\epsilon > 0$ and 
 diverges if $\epsilon = 0$. 
 To see this, recall that mixed exponential growth implies that for any $i$ there exist $\delta_i = \delta_{\Pi_i}$ and $a_i \geq 0$ such that 
$B_i(t) \asymp e^{\delta_i t} (t + 1)^{a_i}$, so we compute
\begin{align*}
K^i_{\lambda}(\phi_{\epsilon})  & = \sum_{n  = 1}^\infty \phi_{\epsilon}(\lambda^n)^{2 \delta} B_i( - 2 \log \phi_{\epsilon}(\lambda^n) ) \\
& \asymp   \sum_{n  = 1}^\infty (n \log \lambda^{-1})^{ -\frac{(1+\epsilon) (\delta - \delta_i)}{\delta-  \delta_{\max}}}  \left(\frac{1+\epsilon}{\delta-  \delta_{\max}} \log (n \log \lambda^{-1}) + 1 \right)^{a_i} \\
& \asymp 
\sum_{n = 1}^\infty n^{-(1+\epsilon)  \frac{\delta - \delta_i}{\delta-  \delta_{\max}}} \log(n+1)^{a_i}.
\end{align*}
Now, if $\epsilon > 0$ the above series converges for any $i > 0$, while if $\epsilon = 0$ it diverges if $\delta_i = \delta_{\max}$.

  Then by Theorem \ref{T:khin}, the limsup set $\Theta_{\lambda}(\phi_\epsilon)$ with respect to
  $\phi_\epsilon$ is $\mu$-null for all $\epsilon>0$, and has full measure for $\epsilon=0$. 
  Choose a boundary point $\xi$ in the full measure set
  $\Theta_{\lambda}(\phi_0)\smallsetminus \bigcup_{\epsilon > 0}
  \Theta_{\lambda}(\phi_\epsilon)$ and a geodesic ray $[o,\xi)$. 

  Define for each $p \in \mathcal{P}$ the enlarged horoball $\tilde{H}_p := H_p(c r_p)$, where $c = O(\alpha)$ 
  is chosen so that, if a geodesic ray from $o$ to $\xi \in \partial X$ intersects $H_p$, 
  then \emph{any} geodesic ray from $o$ to $\xi$ intersects $\tilde{H}_p$. 
  
  By definition of the limsup sets, there exists a sequence of parabolic points $p_n$ in $\mathcal P$ 
  with $r_{p_n} \leq 1$   such that $[o,\xi)$ passes through horoballs  $\phi_0 \tilde{H}_{p_n}$ in order,  
  and passes through no other horoballs of the form $\phi_0 H_p$;
  in other words, the radii $r_{p_n} \phi_0(r_{p_n})$ 
  are monotone decreasing in $n$, and $[o,\xi)\cap \phi_0 H_{p} \neq\varnothing$ implies $p=p_n$ for some $n$.

  \begin{figure}[]
  \centering
  \begin{tikzpicture}[scale=.7]
  \draw (0,0) coordinate (pn);
  \draw (0,8) coordinate (o1);
  \draw (4,8) coordinate (o2);
  \draw (4,0) coordinate (xi);
  \draw (2,13) coordinate (o);

  \draw(-6,0)--(8,0) node[below right] {$\partial X$};

  \draw (pn) --
  (o1);
  \draw (xi)-- coordinate[pos=1] (xitn') (o2);

  \draw (o) to [out=210 , in = 90] (o1);
  \draw (o) to [out= -30, in = 90] (o2);

    \def\rr{4}
    \draw[] (0,\rr) circle (\rr);

    \draw (0,2*\rr) coordinate (q); 
    \draw (\rr,\rr) coordinate (xitn);

    \begin{scope}[red, thick]

  \def\r{5}
    \draw[] (0,\r) circle (\r);

    \def\s{.85}
    \draw[] (4.75,\s) circle (\s);

    \def\ss{.1}
    \draw[] (3.95,\ss) circle (\ss);
    \end{scope}

    \draw[dashed] (pn)
    to [out=90,in=180]  (xitn); 

  \draw (-5,8) node {$\tilde H_{p_n}$};
  \draw (-2.5,5) node {$\phi_{\epsilon_n}H_{p_n}$};

  \def\s{.05}
  \draw[fill] (pn) circle (\s) node[below ] {$p_n$};
  \draw[fill] (o) circle (\s) node[above left] {$o$};
  \draw[fill] (xi) circle (\s) node[below ] {$\xi$};
  \draw[fill] (xitn) circle (\s) node[ above left ] {$\xi_{t_n}$};
  \draw[fill] (xitn') circle (\s) node[ above right] {$\xi_{t_n}^{-}$};
  
\end{tikzpicture}
  \caption{For the proof of Theorem \ref{T:loglaw}. The red horoballs
  correspond to the collection of horoballs $\tilde H_p$ which have been
rescaled by $O(\alpha)$ so that if some choice of geodesic $[o,\xi)$ cuts a  
horoball $H_p$, then any other choice of geodesic cuts $\tilde H_p$. }
  \label{F:loglaw}
\end{figure}

  For each $n$, choose $\epsilon_n$
  so that the geodesic $[o,\xi)$ is tangent to
  the horoball $\phi_{\epsilon_n} H_{p_n }$. 
  More precisely, let $\xi_{t_n}$ be a closest point projection of $p_n$ onto $[o, \xi)$, and choose $\epsilon_n$ 
  such that the boundary of $\phi_{\epsilon_n} H_{p_n }$ contains $\xi_{t_n}$. 
  See Figure \ref{F:loglaw} for an illustration. 
  
  See that $\log r_{p_n}^{-1}\leq t_n + O(\alpha)$ because 
  by Corollary \ref{C:horoball_projection}
  $\log r_{p_n}^{-1} + O(\alpha)$ is the
  distance from $o$ to the horoball $\tilde{H}_{p_n }$, which contains the point $\xi_{t_n}$. 
  Also, note that by 
  Corollary \ref{C:horoball_projection} and the 
  definition of the horoball $\phi_{\epsilon_n} H_{p_n }$, the distance from $o$ to the
  horoball  $\phi_{\epsilon_n} H_{p_n }$ is $-\log (r_{p_n}\phi_{\epsilon_n}(r_{p_n})) + O(\alpha)$. 
 
 Let $\xi_{t_n^-}$ denote the entry point of the geodesic $[o, \xi)$ in the horoball $\tilde{H}_{p_n}$. 
First, since each $r_p$ is chosen so that the union of all $H_p$ is the non-cuspidal part, $\xi_{t_n}^-$ is within 
  uniform bounded distance of $\Gamma o$, so there exists $C_1 $ such that 
\begin{equation} \label{E:log1}
d( \xi_{t_n}, \partial H_{p_n}) - C_1 \leq d(\xi_{t_n}, \Gamma o) \leq d(\xi_{t_n}, \xi_{t_n^-}) + C_1.
\end{equation} 
By Corollary \ref{C:projection_estimate}, since $\xi_{t_n}$ is a closest point projection of $p_n$ onto $[o, \xi)$, 
\begin{equation} \label{E:log2}
 d(\xi_{t_n}, \xi_{t_n^-}) = \beta_{p_n}(\xi_{t_n^-}, \xi_{t_n}) + O(\alpha). 
 \end{equation}
 Moreover, for any $x$ on the boundary of $H_{p_n}$ we have, since Busemann functions are $1$-Lipschitz, 
$ d(x, \xi_{t_n}) \geq \beta_{p_n}(x, \xi_{t_n}) = \beta_{p_n}(\xi_{t_n^-}, \xi_{t_n}) + O(\alpha),$ hence
 \begin{equation} \label{E:log3}
 d(\xi_{t_n}, \partial H_{p_n}) \geq \beta_{p_n}(\xi_{t_n^-}, \xi_{t_n}) + O(\alpha).
 \end{equation}
 Finally, since $\xi_{t_n}$ is on the boundary of $\phi_{\epsilon_n} H_{p_n }$ and $\xi_{t_n^-}$ is on the boundary of $\tilde{H}_{p_n }$, 
   by the quasi-cocycle property of Busemann functions and the
 definition of horospheres 
 we have
 \begin{equation} \label{E:log4}
\beta_{p_n }(\xi_{t_n^-}, \xi_{t_n})=-\log \phi_{\epsilon_n}(r_{p_n}) + O(\alpha).
\end{equation}
  Thus, combining Equations \eqref{E:log1}, \eqref{E:log2},
  and \eqref{E:log4} yields the estimate
  \begin{align}
   \nonumber
    d(\xi_{t_n}, \Gamma o) &  \leq \beta_{p_n}( \xi_{t_n^-}, \xi_{t_n}) + C_1 + O(\alpha)\\
    \nonumber
    &= -\log \phi_{\epsilon_n}(r_{p_n})+C_1 + O(\alpha) \\
    \nonumber
     & = \left(\frac{1+\epsilon_n}{2(\delta-\delta_{\max})}\right)\log \left( \log
    (r_{p_n}^{-1}) \right)+C_1 + O(\alpha) \\
    \label{E:bnd_for_xi_t}
    &\leq \left( \frac{1+\epsilon_n}{2(\delta-\delta_{\max})}  \log(t_n + O(\alpha)) \right)+C_1 + O(\alpha).
  \end{align}
On the other hand, let  $q$ be the closest point to $o$ on $[o, p_n ) \cap H_{p_n }(r_{p_n})$. 
Since $q$ lies on the boundary of the horoball, by the quasi-cocycle
property of the Busemann function, 
  \[ 
  - \log \varphi_{\epsilon_n}(r_{p_n}) = \beta_{p_n}(q, \xi_{t_n}) +
  O(\alpha).
  \]
Since $\xi_{t_n}$ is a closest point projection of  $p_n$ onto $[o,\xi)$,
Corollary \ref{C:projection_estimate} 
and the quasi-cocycle property of Busemann functions
gives us that 
 \[
   t_n+\log(\phi_{\epsilon_n}(r_{p_n})) = d(o, \xi_{t_n}) - \beta_{p_n}(q, \xi_{t_n}) + O(\alpha) = d(o, q) + O(\alpha).
   \]
   Thus, by Corollary \ref{C:horoball_projection}
   we obtain 
  \begin{equation} \label{E:tn}
    t_n+\log(\phi_{\epsilon_n}(r_{p_n}))\leq\log r_{p_n}^{-1} + C_2
  \end{equation}
  where $C_2$ is a constant depending only on the hyperbolicity constant.
  Thus, using Equations \eqref{E:log1}, \eqref{E:log3},
  \eqref{E:log4} and \eqref{E:tn}, 
 {\small  \begin{align*}
    d(\xi_{t_n},\Gamma o) &\geq -\log(\phi_{\epsilon_n}(r_{p_n}))-C_1 - O(\alpha)\\
    &=\left( \frac{1+\epsilon_n}{2(\delta-\delta_{\max})} \right)\log(\log
    r_{p_n}^{-1})-C_1 - O(\alpha)\\
    & \geq \left( \frac{1+\epsilon_n}{2(\delta-\delta_{\max})} \right)
    \log(t_n+\log(\phi_{\epsilon_n}(r_{p_n}))-C_2)-C_1 - O(\alpha)\\
    & = \left( \frac{1+\epsilon_n}{2(\delta-\delta_{\max})} \right)
    \log\left(t_n-\left( \frac{1+\epsilon_n}{2(\delta-\delta_{\max})}
    \right)\log(\log(r_{p_n}^{-1})) - C_2 \right)-C_1- O(\alpha) \\
    & \geq \left( \frac{1+\epsilon_n}{2(\delta-\delta_{\max})} \right)
    \log\left(t_n-\left( \frac{1+\epsilon_n}{2(\delta-\delta_{\max})}
    \right)\log(t_n - O(\alpha))-C_2 \right)-C_1 - O(\alpha).
  \end{align*}
  } Thus, noting that $\epsilon_n \to 0$ as $n \to \infty$, 
  \[
    \frac1{2(\delta-\delta_{\max})}
    \leq 
    \limsup_{t\to+\infty} \frac{d(\xi_t,\Gamma o)}{\log t}.
  \]
  It remains to prove the upper bound on the limsup. For values of
  $t$ such that $\xi_t\in X_{nc}$, the result is trivial.
  Recall that each $t_n$ is chosen so that for all values $t$ so that
  $\xi_t\in H_{p_n }(r_{p_n})$, the distance $d(\xi_t,\Gamma o)$
  is maximized up to $O(\alpha)$ at $t=t_n$. 
  Then for such $t\geq t_n$, $\frac{d(\xi_t,\Gamma o)}{\log(t)} \leq
  \frac{d(\xi_{t_n},\Gamma o)}{\log(t_n)}$ as desired by Equation
  \eqref{E:bnd_for_xi_t}. Now consider
  $t\leq t_n$.  
  Then, applying Equations \eqref{E:log1}, \eqref{E:log2}, \eqref{E:log3}, 
 \begin{align*}
  d(\xi_{t_n},\Gamma o) & \geq \beta_{p_n }(\xi_{t_n^-},\xi_{t_n}) - C_1 - O(\alpha) =d(\xi_{t_n},\xi_{t_n^-})- C_1 - O(\alpha) \\
   & =|t_n^- -t_n|- C_1 - O(\alpha)\geq    t_{n}-t- C_1 - O(\alpha).
  \end{align*}
  Thus, $t\geq t_n-d(\xi_{t_n},\Gamma o) - C_1 - O(\alpha)$, and by Equation
  \eqref{E:bnd_for_xi_t}, 
  {\small \[
    \frac{d(\xi_t,\Gamma o)}{\log(t)} \leq
    \frac{d(\xi_{t_n},\Gamma o)}{\log(t_n-d(\xi_{t_n},\Gamma o)- C_1 -O(\alpha))}
    \leq 
    \frac{d(\xi_{t_n},\Gamma o)}{\log(t_n- C_3 \log(t_n) - C_3)}
  \] }
  for some constant $C_3>0$. The result follows Equation
  \eqref{E:bnd_for_xi_t}.
\end{proof}

\section{Applications to Hilbert geometry} \label{S:Hilbert-applications}

In this section, we will apply the results to a class of geometries called
{\em Hilbert geometries}. These geometries generalize hyperbolic geometry
to a non-Riemannian setting in which the metric is not $\CAT(0)$
\cite[Appendix B]{egloff} but, for a large family of examples of interest,
is Gromov hyperbolic. We first introduce the preliminary background.

A domain $\Omega$ in real projective space $\mathbb RP^n$ is {\em properly convex} 
if there exists an affine chart in which $\Omega$ is bounded and convex,
meaning its intersection with any line segment is connected.
We say $\Omega$ is {\em strictly convex} if, moreover, the projective
boundary $\partial_{\text{proj}} \Omega$ in an affine chart does not contain any open line segments. 
Any properly convex domain admits a natural, projectively invariant metric
called the {\em Hilbert metric} which is central to this application. 
The Hilbert metric is defined as follows. Choose an affine chart in which
$\Omega$ is bounded; then for each $x,y\in\Omega$, any projective line
passing through $x$ and $y$ must intersect 
$\partial_{\text{proj}} \Omega$ at exactly two points, $a,b$. Then 
\[
  d_\Omega(x,y):=\frac12 \big| \log [a;x;y;b] \big|
\]
where $[a;x;y;b]:=\frac{|a-y| |b-x| }{|a-x||b-y|}$ is the cross-ratio with
respect to the ambient affine metric inherited from the chart. 
The normalization factor of $\frac12$ ensures that if $\Omega$ is an
ellipsoid, then $(\Omega,d_\Omega)$ is the Beltrami--Klein model for 
hyperbolic space of constant curvature $-1$. 

The cross-ratio is a projective invariant, hence the metric does not depend on
the chart, and projective transformations which preserve $\Omega$ are
isometries with respect to $d_\Omega$. 
Straight lines are geodesics for this metric, and are the only geodesics
when $\Omega$ is strictly convex. Evidently, the Hilbert metric 
is proper and the  
topological boundary $\partial_{\text{proj}} \Omega$ in $\mathbb RP^n$ is a
compactification of $\Omega$ on which projective transformations that
preserve $\Omega$ act as
homeomorphisms. 
If $\Gamma<\PSL(n+1,\mathbb R)$ preserves $\Omega$ and is
discrete then its action for the Hilbert metric is properly discontinuous.
Thus, the definition of geometrical finiteness and all the related notions
(Section \ref{S:geom_finite}, Definition \ref{D:geom_finite}) are coherent for 
the action of a discrete group of projective
transformations $\Gamma$  on $\Omega$. The limit set $\Lambda_\Gamma$ is
again the smallest closed invariant set, and is hence basepoint
independent,
when $|\Lambda_\Gamma|\geq 3$ and
$\Omega$ is strictly convex with $C^1$ boundary (for more, see
\cite[D\'efinition 4.1, Lemme 4.2]{cramponmarquis_finitude}). 
We note the following lemma: 

\begin{lemma}
  \label{L:identify_boundaries}
  Let $\Omega\subset \mathbb RP^n$ be strictly convex  and $\Gamma<\PSL(n+1,\mathbb R)$ a discrete group preserving
  $\Omega$. 
  If the convex hull $C_\Gamma$ is a hyperbolic metric space when endowed
  with the Hilbert metric, then $\Lambda_\Gamma\subset \partial_{\text{proj}}\Omega$ is
  naturally identified with the hyperbolic boundary of  $C_\Gamma$.
\end{lemma}

\begin{proof}
  Fix a basepoint $o\in C_\Gamma$ to define $\Lambda_\Gamma$. Recall  
  the hyperbolic boundary $\partial C_\Gamma$ is the set of geodesic rays at $o$ up to bounded equivalence. See
  that by strict convexity of $\Omega$ and the definition of the Hilbert
  metric, if two projective line segments starting at $o$ and going to
  $\partial_{\text{proj}}\Omega$ are bounded distance apart then they coincide. 
  Thus, the map $\partial C_\Gamma\to \Lambda_\Gamma$ defined by associating to
  each geodesic ray based at $o$ and contained in $C_\Gamma$ its unique 
  intersection with
  $\partial_{\text{proj}}\Omega$ is well-defined.
  This point of intersection must lie in
  $\Lambda_\Gamma$ by definition of $C_\Gamma$. By convexity of $C_\Gamma$
  in this setting, 
  the map is surjective.  
\end{proof}

A pair $(\Omega,\Gamma)$ where $\Gamma<\PSL(n+1,\mathbb R)$ is a discrete
group that preserves
$\Omega$ is called a {\em convex real projective structure} on the quotient
manifold $\Omega/\Gamma$, and when $\Omega$ is strictly convex, we
specify that the structure is a {\em strictly convex real projective
structure}. 

\subsection{Relation to work of Crampon--Marquis} \label{S:CM14}

Geometrical finiteness in Hilbert geometry was first studied by 
Crampon--Marquis \cite{cramponmarquis_finitude}. They showed, for example,
that when $\Omega$ is strictly convex with $C^1$ boundary, the isometries
of $(\Omega,d_\Omega)$ can be classified as elliptic, parabolic, loxodromic
as in the setting of hyperbolic metric spaces 
as in Section~\ref{S:geom_finite} \cite[Theorem 3.3, Section 3.5]{cramponmarquis_finitude}. 
Crampon--Marquis used two definitions of geometrical finiteness: 

\begin{definition}[{Crampon--Marquis \cite{cramponmarquis_finitude}}]
  \label{D:cm_gf}
  For a strictly convex $\Omega\subset\mathbb RP^n$ with $C^1$ boundary and
  a non-elementary discrete group $\Gamma<\PSL(n+1,\mathbb R)$ which
  preserves $\Omega$,  the action of $\Gamma$
  is {\em geometrically finite on }$\partial_{\text{proj}} \Omega$ if every
point $\xi$ in $\Lambda_\Gamma$ is either a bounded parabolic
point or a conical limit point. This is the same as Definition
\ref{D:geom_finite} and so we will say that in this case, $\Gamma$ is a
{\em geometrically finite} group. 
More strongly, Crampon-Marquis define $\Gamma$ to be {\em geometrically finite with hyperbolic cusps}
if every point $\xi$ in $\Lambda_\Gamma$ is either a bounded parabolic point with stabilizer conjugate into $\SO(n,1)$ 
or a conical limit point.
Crampon--Marquis refer to geometrical finiteness with hyperbolic cusps as
``$\Gamma$ acting geometrically finitely on $\Omega$". We will avoid this
language to reduce confusion with Definition \ref{D:geom_finite}. 
\end{definition}

Crampon--Marquis show that these two conditions are not equivalent. 
In \cite[Proposition 10.7]{cramponmarquis_finitude}, they produce a group $\Gamma$ in
$\PSL(5,\mathbb R)$ which preserves a strictly convex set $\Omega$ with $C^1$
boundary in $\mathbb RP^4 $ such that $\Gamma$ 
is geometrically finite, but it is not geometrically
finite with hyperbolic cusps. 

\subsection{Examples with hyperbolic convex hull} \label{S:hyp_cc}

There is a large family of examples which are geometrically finite with
hyperbolic cusps, and these examples will have hyperbolic convex hull:

\begin{theorem}[{\cite[Th\'eor\`eme 1.8]{cramponmarquis_finitude}}]
  \label{T:hyp_convex_core}
  If $\Omega$ is strictly convex with $C^1$ boundary and $\Gamma$ is 
  geometrically finite with hyperbolic cusps, then $\Gamma$ is relatively hyperbolic,
  and the convex hull $C_\Gamma$ with the Hilbert metric is a hyperbolic metric space.
\end{theorem}

On a strictly convex $\Omega$ with $C^1$ boundary, any group acting with cofinite
volume, and more generally any geometrically finite group for which all parabolic
stabilizers have maximal rank, will be geometrically finite with hyperbolic
cusps (\cite[Theorem 0.4, Theorem 0.5]{clt}, \cite[Th\'eor\`eme 7.14]{cramponmarquis_finitude}). 
In fact when $\Omega$ admits a finite volume quotient, 
it is enough to assume that either $\Omega$ is strictly convex or $\Omega$
has $C^1$-boundary since these criteria are equivalent in that case \cite[Theorem 0.15]{clt}. 

More explicitly, examples include all 
geometrically finite $\Gamma<\PSL(3,\mathbb R)$ preserving
a strictly convex $\Omega\subset\mathbb RP^2$ with $C^1$ boundary. 
There are many such actions: for instance, 
the moduli space of finite volume strictly convex real projective
structures on a surface of genus $g$ with $p$ punctures has real dimension $16g-16+8p$
\cite{marquis_modules}, and contains the $6g-6+2p$ dimensional Teichm\"uller space via the Beltrami--Klein model. 
In higher dimensions, the moduli space of finite volume strictly convex
real projective structures can be nontrivial even though the Teichm\"uller
space is trivial. In every dimension, there are deformable examples
\cite{ballasmarquis,marquis12} via
the Johnson-Millson bending construction \cite{johnsonmillson}. In dimension three, there
are deformable examples that arise from a generalization of 
Thurston's gluing equations 
\cite{ballas_casella}. There are also examples of closed topological
manifolds that admit a strictly
convex projective structure but do not admit a Riemannian constant curvature 
hyperbolic metric \cite{Ben5, Kapovich2007}. 
It is plausible that there is a corresponding
finite volume non-compact example which admits a strictly convex real
projective structure but does not admit a metric of constant negative
curvature. Our results apply to any such examples.

The convex hull can be hyperbolic even when 
the action of $\Gamma$ only satisfies the
weaker, standard notion of geometrical
finiteness as in Definition \ref{D:geom_finite} without having hyperbolic
cusps. 
For instance, in the above-mentioned example (\cite[Proposition
10.7]{cramponmarquis_finitude}), the convex hull is hyperbolic
\cite{DGK,zimmer22}.
It seems plausible that  
for any Hitchin representation 
of a geometrically finite Fuchsian group which preserves a properly convex
subset of $\mathbb RP^d$, there exists some, possibly different, strictly
convex set $\Omega$ with $C^1$ boundary preserved by $\Gamma$ such that the
convex hull of the limit set in $\partial_{\text{proj}}\Omega$ is a hyperbolic metric
space, but at the moment it is not known (see \cite{CZZgf} for more
details).

\begin{remark}
  \label{rem:comments_CM}
Let us note that  \cite[Th\'eor\`eme 9.1]{cramponmarquis_finitude} also
claims that hyperbolicity of the convex hull $C_\Gamma$ of $\Gamma$
implies $\Gamma$ is geometrically finite with hyperbolic cusps, 
but as discussed above, that is not true. 
However, we do not need this implication in this paper, so this is irrelevant for our purposes.
Corrections to \cite{cramponmarquis_finitude} are expected in a forthcoming
erratum by Blayac--Marquis.

On the other hand, one might optimistically hope that whenever $\Gamma$ is geometrically
finite, the convex hull is hyperbolic. However, 
in the same forthcoming article, Blayac--Marquis 
produce examples such that $\Gamma$ is geometrically finite and fails to have
hyperbolic convex hull in $\Omega$. 
Interestingly, for the same provided examples, 
they produce another $\Gamma$-invariant $\Omega'$ for which the convex hull 
is hyperbolic. Whether or not this phenomenon holds
in general is unclear. 
\end{remark}

\subsection{Patterson--Sullivan measures for geometrically finite Hilbert geometries}
\label{S:PS_hilbert}

Crampon showed in his thesis that Patterson's construction can be adapted to the setting of
geometrically finite groups with hyperbolic cusps
when $\Omega$ is
strictly convex with $C^1$ boundary \cite[Theorem 4.2.1]{cramponthese}.
We call a measure arising from this construction a {\em Patterson--Sullivan
measure}. 
Crampon proves the measures are supported on the limit set
\cite[Section 4.2.1]{cramponthese}, and then proves in the case of surfaces
that the Patterson--Sullivan measures have no atoms \cite[Lemma 4.3.3,
Proposition 4.3.5]{cramponthese}. These arguments generalize to higher
dimensions due to \cite[Corollaire 7.18]{cramponmarquis_flot}, which
generalizes \cite[Lemma 1.3.4]{cramponthese}. In recent work, Zhu confirms
that these results extend to higher dimensions in the strictly convex with
$C^1$ boundary setting (see \cite[Lemma 11, Proposition 12, Corollary
13]{zhu20}).  These results hinge on a calculation that any
bounded parabolic group preserving $\Omega$ with rank $r$ and conjugate
into $\SO(n,1)$ has critical exponent $\delta_\Pi=\frac{r}2$, and if $\Pi$ is a subgroup of a
geometrically finite group $\Gamma$, then
$\delta_\Pi<\delta_\Gamma$ \cite[Lemma 11]{zhu20}, \cite[Lemme
9.8]{cramponmarquis_flot}.  The work of Zhu was further generalized by
Blayac to the rank one setting, without the strictly convex with $C^1$
boundary condition \cite[Theorem 1.6]{blayac} and by Blayac--Zhu when
$\Gamma$ is geometrically finite and $\Omega$ is strictly convex with $C^1$ boundary \cite[Theorem 9.1, Lemma 9.13,
Proposition 9.14]{blayaczhu}. Blayac--Zhu elaborate after \cite[Theorem 5.4]{blayaczhu} on why finiteness 
of Patterson--Sullivan measure given by \cite[Theorem 9.1]{blayaczhu} implies that Patterson--Sullivan measure has full support on $\Lambda_\Gamma$. 

\subsection{Growth independence of domain}

We observe in this section that the critical exponent and the 
upper and lower annular growth rates do not depend on the domain. 
One consequence of this is if $\Gamma$ is geometrically finite with hyperbolic cusps, then all parabolic subgroups 
have exponential growth and their critical exponent is equal to half of the rank of the group, as for hyperbolic
space. We will not need this observation as our applications will be more general. 

For $G\in\SL(d,\mathbb R)$, let  $\mu_1(G),\ldots,\mu_d(G)$ be the singular
values of $G$, listed in decreasing order. Then for $g\in\PSL(d,\mathbb
R)$, define $\kappa(g) :=\frac12(\log\mu_1(G)-\log\mu_n(G))$ for any lift $G$
of $g$.
\begin{proposition}[{Proposition 10.1 \cite{DGK}}]   \label{P:dgk}
  For any properly convex domain $\Omega$ in $\mathbb RP^n$ and any
  $o\in\Omega$, there exists a constant $C$ such that for all $g\in
  \Aut(\Omega)$,
  \[
    |d_\Omega(o,go)-\kappa(g)|\leq C.
  \]
\end{proposition}

\begin{lemma}
  When $\Pi<\PSL(n+1,\mathbb R)$ preserves some properly convex domain
  $\Omega$, the upper
  and lower annular growth rates $\delta_\Pi^-,\delta_\Pi^+$ and the
  critical exponent $\delta_\Pi$ do not depend
  on $\Omega$.
  In particular, if $\Omega'$ is another properly convex domain in $\mathbb RP^n$ and
  $\Pi<\Aut(\Omega)\cap \Aut( \Omega')$, then the action of $\Pi$ on
  $\Omega$ has $\delta$-tempered growth if and only if the action of
  $\Pi$ on $\Omega'$ has $\delta$-tempered growth.
  \label{L:tempered_invariant}
\end{lemma}

\begin{proof}
  By Proposition \ref{P:dgk}, fixing an $o\in\Omega$ and $o'\in\Omega'$,
  we have $d_\Omega(o,go)=d_{\Omega'}(o',go')+O(1)$ where $g\in\Pi$. 
  Let
  \[
    B_\Omega(t)=\#\{g\in \Pi \mid d_\Omega(o,go)\leq t\}, \qquad
    B_{\Omega'}(t)=\#\{g\in \Pi \mid d_\Omega(o',go')\leq t\},
  \]
  and
  \[
    A_{\Omega,r}(t)=\frac1r\log \frac{B_\Omega(t+r)}{B_\Omega(t)}.
  \]
  Then there exists a constant $C$ such that
  \[
    B_{\Omega'}(t-C)\leq B_\Omega(t)\leq B_{\Omega'}(t+C).
  \]
  It follows that $\delta_\Pi$ does not depend on $\Omega$. Similarly, for $r>2C$,
  \[
    \frac{r-2C}r A_{\Omega',r-2C}(t+C)\leq A_{\Omega,r}(t)\leq \frac{r+2C}r
    A_{\Omega',r+2C}(t-C).
  \]
  It is then straightforward to verify that $\delta_{\Pi}^{\pm}$ is independent of $\Omega$.
\end{proof}

\subsection{Growth of parabolic subgroups}

We prove in this section that parabolic subgroups in strictly convex Hilbert geometry 
have mixed exponential growth, as defined in Definition \ref{D:mixed_growth}. As a consequence, we will see that
if $\Gamma$ is geometrically finite and preserves a strictly convex domain with $C^1$ boundary, then
$\Gamma$ has tempered parabolic subgroups. 

\begin{proposition}
  Let $\Omega$ be a strictly convex domain in $\mathbb RP^n$ with $C^1$
  boundary. 
  Then every discrete 
  parabolic subgroup of $\PSL(n+1,\mathbb R)$ preserving $\Omega$ has 
  mixed exponential growth for the Hilbert metric.  
  \label{P:mixed_growth}
\end{proposition}

\begin{proof}
Let $\Pi$ be a parabolic subgroup in $\textup{Aut}(\Omega)$. Then by
\cite[Proposition 9.6]{blayaczhu} (which is a consolidation of \cite[Proposition
7.1]{cramponmarquis_finitude} and \cite[Lemme 7.6]{cramponmarquis_finitude}), 
$\Pi$ is a uniform lattice in its Zariski closure $\mathcal{N}$ and
moreover, $\mathcal{N}$ can be written as $\mathcal{N} = K \times U$ 
where $K$ is compact and $U$ is unipotent. If one considers the projection $p_U : \mathcal{N} = K \times U \to U$, the image $\Pi' = p_U(\Pi)$ 
is a lattice in $U$, and the kernel of the restriction of $p_U$ to $\Pi$ is finite. 

  Fix on $G$ the norm $\Vert g \Vert :=  \textup{tr} (g^t g) ^{1/2}$, which is submultiplicative. 
  Let us introduce, for $g \in G$, the notation 
  $$| g |  := \frac{\log \Vert g \Vert + \log \Vert g^{-1} \Vert}{2}.$$
  Since the norm is submultiplicative, we have $|g| \geq 0$ for any $g$. Moreover, we also have 
  $$|gh| \leq |g| + |h| \qquad \textup{for any }g, h \in G.$$
  Notice by Proposition \ref{P:dgk}, using that all matrix norms are equivalent, that 
  \begin{equation} \label{E:dgk}
    d_\Omega(o, g o)= |g| +O(1) \qquad \textup{for any }g \in \Pi.
  \end{equation}
  Let $\mathfrak u$ denote the Lie algebra of $U$. The exponential map $\exp:  \mathfrak u \to U$ is a diffeomorphism, and the pushforward 
  of a Lebesgue measure on $\mathfrak u$ is the Haar measure on $U$. 
   Let $P\colon \mathfrak
   u\to\mathbb R$ be given by 
   $$P(x)=\| \exp(x)\|^2\|\exp(-x)\|^2.$$
   Note that $\log P(x)=4|\exp(x)|$. 
   Since the norm is submultiplicative,
  $P(x)\geq  C>0$ for all $x\in\mathfrak u$ where $C = \Vert \textup{Id} \Vert^2 $ is a constant
  depending only on $n$. Since $U$ is a unipotent matrix group, 
  $\mathfrak u$ is a set of nilpotent matrices with bounded degree. 
  Then by the definition of matrix exponentiation, the entries of
  $\exp(x)$ are polynomials in the entries of $x$, and it 
  then follows from the definition that $P$ is a polynomial in $\dim (\mathfrak u)$ many variables. 
  Note that $P$ is proper since $\exp$ is a diffeomorphism and the norm
  function is a proper map for any choice of matrix norm on the finite dimensional
  vector space $\mathfrak u$.
  Letting $\lambda$ denote Lebesgue measure on $\mathfrak u$, see that
  the pushforward of $\lambda$ by $\exp$ is Haar measure on $U$
  (see e.g. \cite[Theorem 1.2.10]{corwin_greenleaf}).
 
Now, by Benoist--Oh \cite[Corollary 7.3(a)]{benoist_oh}, there exist
$a\in\mathbb Q, a>0$ and $ b\in\mathbb Z, b\geq 0$ 
  such that 
  $$\lambda(\{ x \in \mathbb{R}^d \ : \ P(x) \leq t \}) \asymp t^a (\log t)^b \qquad \textup{for any } t > 0.$$
  
  We will use this to show that 
 \begin{equation}   \label{E:hilbert_mixed_growth}
 \# \{ g \in \Pi' \ : \ |g| \leq t \} \asymp e^{4at} t^b \qquad \textup{for any } t > 0.
 \end{equation}
 Let $F$ be a compact fundamental domain for the action of $\Pi'$ on $U$, which exists since $\Pi'$ is a uniform lattice. 
Let $m$ be the Haar measure on $U$. 
If $\Delta := \sup\{ |f|, f \in  F\}$ is the diameter of $F$, then 
$$m (\{ u \in U \ : \  | u | \leq t - \Delta \})  \leq m(F) \# \{ g \in \Pi' \ : \ |g| \leq t \} \leq m (\{ u \in U \ : \  |u| \leq t + \Delta \}).$$
Moreover, since the pushforward of the Lebesgue measure under the exponential map is the Haar measure, 
$$m (\{ u \in U \ : \  | u | \leq t  \}) = \lambda (\{ x \in \mathfrak u \ : \ P(x) \leq e^{4 t} \}).$$
Equation \eqref{E:hilbert_mixed_growth} follows. 

Finally, we show that, for any $t > 0$, 
\begin{equation}
  \label{E:hilbert_mixed_growth2}
\# \{ g \in \Pi \ : \ d_\Omega(o, go) \leq t \} \asymp e^{4at}t^b.  
\end{equation}
Let $K_0$ be the kernel of $p_U\vert_\Pi$, $c$ the cardinality of
$K_0$, and $A = \sup \{ d_\Omega(o, ko) : k \in K_0\}$. 
Then, since any $g \in \Pi$ can be written as $g = k u$ for $k \in K$ and
$u \in \Pi'$, and at most $c$ values of $g$ correspond to a given value of $u$, 
$$\# \{ g \in \Pi \ : \ |g| \leq t - A\} \leq \# \{ u \in \Pi' \ : \ |u| \leq t \} \leq c \# \{ g \in \Pi \ : \ |g| \leq t + A\}$$
hence, for any $t > 0$, 
$$\# \{ g \in \Pi \ : \ |g| \leq t \} \asymp e^{4at}t^b.$$

Finally, Equation \eqref{E:dgk} now implies Equation
\eqref{E:hilbert_mixed_growth2}, as desired. 
\end{proof}

We will now see how to compute the growth rate for rank one parabolic subgroups.

\begin{lemma} \label{L:rankone}
  Let $\Omega\subset \mathbb RP^n$ be a properly convex domain
  and $\Pi<\PSL(n+1,\mathbb R)$ a discrete group preserving
  $\Omega$.
  If $\Pi$ is a parabolic group of rank one, generated (up to finite index) by $g$, 
  then $\Pi$ has pure exponential growth, i.e.  for any $T > 0$ one has 
$$\# \{ h \in \Pi \ : \ d_\Omega(o, h o) \in [T, T+1] \} \asymp e^{\delta_\Pi T}$$ 
with $\delta_\Pi = \frac{1}{k-1}$, where $k$ is the size of the largest Jordan block of $g$.
\end{lemma}

For example, Lemma \ref{L:rankone} applies to the Crampon--Marquis example
\cite[Proposition 10.7]{cramponmarquis_finitude} with $n=k=4$.

\begin{proof}[Proof of Lemma \ref{L:rankone}]
  By \cite[Proposition 2.13]{clt}, which applies to any properly convex
  $\Omega$, every eigenvalue of $g$ is equal to 1.
  Thus by taking the Jordan form, $g$ is conjugate to a unipotent matrix.
By taking powers, we have
$$\Vert g^n \Vert \asymp |n|^{k-1}\qquad \textup{for any }n \in \mathbb{Z}$$
where $k$ is the size of the largest Jordan block of $g$ and $\Vert \cdot
\Vert$ is any invariant norm, such as the leading singular
value. Then by Proposition \ref{P:dgk} (see also \cite[Proposition 2.6]{blayaczhu}), we obtain
$$d_\Omega(o, g^n o) = \frac{ \log \Vert g^n \Vert + \log \Vert g^{-n} \Vert}{2} + O(1) = (k-1) \log n + O(1)$$
thus
$$\# \{ n \in \mathbb{Z} \ : \ d_\Omega(o, g^n o) \in [T, T+1] \} \asymp e^\frac{T}{k-1}$$
for any $ T > 0$, which proves the claim.
\end{proof}

From the previous two results, we obtain that parabolic subgroups have tempered growth. 

\begin{corollary}
  Let $\Omega \subset \mathbb RP^n$ be a strictly convex domain with
  $C^1$ boundary, and $\Gamma< \PSL(n+1,\mathbb R)$ a 
  geometrically finite group preserving $\Omega$. 
  Then $\Gamma$ has tempered parabolic subgroups. 
  \label{C:hilbert_tempered}
\end{corollary}

\begin{proof}
  By Proposition \ref{P:mixed_growth} we obtain $\delta_\Pi^{\pm}=\delta_\Pi$; since every parabolic group contains 
  a rank one parabolic subgroup, 
  from Lemma \ref{L:rankone} we obtain that $\delta_\Pi>0$.
  Blayac--Zhu prove \cite[Lemma 8.13]{blayaczhu} that for any non-elementary group $\Gamma$
  acting on a strictly convex $\Omega$ with $C^1$ boundary, any parabolic
  subgroup $\Pi<\Gamma$ has $\delta_\Pi<\delta_\Gamma$. This concludes the
  proof. 
\end{proof}

\subsection{Statement of result}
\label{S:final_result_hilbert}

We are now ready to verify that geometrically finite Hilbert geometries
satisfy the global shadow lemma and logarithm law. 
The theorem below applies to all the examples discussed in Subsection
\ref{S:hyp_cc}, and implies Theorems
\ref{T:intro-global_shadow_lemma-Hilbert} and
\ref{T:intro-log-law} from the introduction. 

\begin{theorem}
  \label{T:apply_hilbert}
  Let $\Omega$ be a strictly convex domain in $\mathbb RP^n$ with $C^1$
  boundary and $\Gamma<\PSL(n+1,\mathbb R)$ a geometrically finite group
  which preserves $\Omega$.  
  Assume the convex hull of the limit set $C_\Gamma$ is hyperbolic with
  respect to the Hilbert metric. Then any Patterson--Sullivan measure $\mu$ satisfies the global 
  shadow lemma (Theorem \ref{T:global_shadow_lemma}), and for $\mu$-a.e.
  $\xi\in\Lambda_\Gamma$, 
    \[
    \limsup_{t\to+\infty} \frac{d(\xi_t,\Gamma o)}{\log t} =
    \frac1{2(\delta-\delta_{\max})}
  \]
  where $\delta_{\max}$ is the maximal growth rate of any parabolic subgroup,
  and $\xi_t$ is the point on the geodesic ray $[o,\xi)$ that is distance $t$ from $o$.
\end{theorem}

\begin{proof}
  We need only verify the hypotheses. First, $(C_\Gamma, d_\Omega)$ is a
  proper hyperbolic metric since the Hilbert metric is proper on $\Omega$.  
  Then $C_\Gamma$ has boundary $\Lambda_\Gamma$ by Lemma
  \ref{L:identify_boundaries}, 
  and $\Gamma$ acts minimally on
  $\Lambda_\Gamma$ since the action is non-elementary, as discussed in the
  beginning of this section. The 
  Patterson--Sullivan measures constructed by Blayac--Zhu 
  are a conformal density of dimension $\delta_\Gamma$ on $\Lambda_\Gamma$
  with no atoms 
  \cite[Proposition 9.14, Theorem 9.1]{blayaczhu} (see Subsection~\ref{S:PS_hilbert} for elaboration), 
  and $\Gamma$ has $\delta_\Gamma$-tempered parabolic subgroups by Corollary 
  \ref{C:hilbert_tempered}. Thus, the hypotheses of the global shadow lemma
  (Theorem \ref{T:global_shadow_lemma}) are satisfied. Finally, since
  $\Gamma$ has mixed exponential growth by Proposition \ref{P:mixed_growth}, the
  hypotheses of the logarithm law (Theorem \ref{T:loglaw}) are
  satisfied, completing the proof.
\end{proof}

\bibliographystyle{alpha}
\bibliography{refs}

\end{document}